\documentclass[12 pt,a4paper,leqno]{article}

\usepackage[french,english]{babel}
\usepackage{latexsym,amsmath,amscd,amsfonts,amssymb,mathrsfs,amsthm}
\usepackage{graphicx}
\usepackage[all]{xy}
\usepackage[T1]{fontenc}

\RequirePackage{amsmath}
\RequirePackage{amssymb}
\RequirePackage{ifthen}
\RequirePackage{pslatex}

\newtheorem{thm}{Théorème}
\newtheorem{defn}{Définition}
\newtheorem{lem}{Lemme}
\newtheorem{prop}{Proposition}
\newtheorem{cor}{Corollaire}
\newtheorem{rem}{Remarque}

\newtheorem{ex}{Exemple}

\author{Jacques Sauloy}

\title{Théorie analytique locale des équations aux $q$-différences de pentes
  arbitraires}

\def\C{{\mathbf C}}
\def\Q{{\mathbf Q}}
\def\Z{{\mathbf Z}}
\def\R{{\mathbf R}}
\def\N{{\mathbf N}}

\def\Cs{{\mathbf C^*}}
\def\Sr{{\mathbf{S}}}

\def\div{\text{div}}
\def\dim{\text{dim}}

\def\Diag{\text{Diag}}

\def\Ker{\text{Ker}}
\def\Id{\text{Id}}
\def\Sp{\text{Sp}}

\def\lmod{\left |}            % valeurs absolues cote gauche
\def\rmod{\right |}           % valeurs absolues cote droit
\def\tq{~ | ~}              % ensemble des x tels que E = \{x \tq P(x) \}
\def\ie{\emph{i.e.}}

\def\cf{\emph{cf.}}

\def\Lin{\mathcal{L}}         % espace des A.L. \Lin_{K}(E,F)
\def\Hom{\text{Hom}}          % ensemble des morphismes
\def\Aut{\text{Aut}}          % groupe des automorphismes
\def\1{\underline{1}}         % objet unite
\def\Ker{\text{Ker~}}         % comme chez Nicolas
\def\Id{\text{Id}}            % comme chez Nicolas
\def\Mat{{\text{Mat}}}        % espace des matrices
\def\GL{{\text{GL}}}          % groupe lineaire
\def\GLn{{\GL_n}}          % groupe lineaire
\def\GLnc{{\GLn(\C)}}          % groupe lineaire
\def\gl{{\text{gl}}}          % son algebre de Lie
\def\glnc{{\text{gl}_n(\C)}}          % son algebre de Lie
\def\ii{{\text{i}}}

\def\sq{\sigma_q}

\def\F{{\mathcal{F}}}

\def\B{{\mathcal{B}}}

\def\O{{\mathcal{O}}}
\def\M{{\mathcal{M}}}

\def\D{{\mathscr{D}_{q,K}}}
\def\gr{{\text{gr}}}

\def\Eq{{\mathbf{E}_{q}}}
\def\EE{{\mathscr{E}}}
\def\G{{\mathfrak{G}_{A_{0}}}}
\def\g{{\mathfrak{g}_{A_{0}}}}
\def\Gd{{\mathfrak{G}^{\geq \delta}_{A_{0}}}}
\def\gd{{\mathfrak{g}^{\geq \delta}_{A_{0}}}}
\def\Gds{{\mathfrak{G}^{> \delta}_{A_{0}}}}
\def\gds{{\mathfrak{g}^{> \delta}_{A_{0}}}}
\def\gde{{\mathfrak{g}^{(\delta)}_{A_{0}}}}

\def\Ra{{\C\{z\}}}
\def\Ka{{\C(\{z\})}}

\def\Raq(d){{\C\{\xi\}_{q,(\delta)}}}
\def\Kaq(d){{\C(\{\xi\})_{q,(\delta)}}}

\def\Rf{{\C[[z]]}}
\def\Kf{{\C((z))}}

\def\Kg{{\C(z)}}
\def\Rw{{\mathcal{O}(\C^{*})}}

\def\Oc{{\hat{\mathcal{O}}}}
\def\Kc{{\hat{K}}}

\def\Lie{{\text{Lie}}}
\def\St{\mathfrak{St}}
\def\Stq{\mathfrak{St}_q}
\def\Stqr{\mathfrak{St}_q^r}
\def\Stqu{\mathfrak{St}_q^1}
\def\Stqp{\mathfrak{St}_q'}
\def\st{\mathfrak{st}}
\def\stq{\mathfrak{st}_q}

\def\stqu{\mathfrak{st}_q^1}

\def\stt{{\tilde{\st}}}
\def\sttq{{\tilde{\st}_q}}

\def\sttqu{{\tilde{\st}_q^1}}

\def\Sttq{{\widetilde{\St}_q}}

\def\Sttqu{{\widetilde{\St}_q^1}}

\def\Rep{{\text{Rep}}}

\def\Repc{{\text{Rep}_{\C}}}
\def\U{{\mathfrak U}}
\def\Ua{{\mathfrak U}_{A_{0}}}
\def\Res{{\text{Res}}}
\def\res{{\text{res}}}

\def\limproj{{\underset{\longleftarrow}{\lim}}}
\def\limind{{\underset{\longrightarrow}{\lim}}}

\def\FTA{{famille trivialisante adaptée}}
\def\FTAs{{familles trivialisantes adaptées}}

\def\equivd{{~\equiv_\delta~}}
\def\equivdp{{~\equiv_{\delta'}~}}
\def\isom{{\overset{\sim}{\longrightarrow}}}
\def\thq{{\theta_q}}
\def\thqc{{\theta_{q,c}}}
\def\thqd{{\theta_q^\delta}}

\def\e{{\mathbf{e}}}

\def\Gq{{\mathbf{G}_{q}}}
\def\Gqp{{\mathbf{G}_{q,p}}}
\def\Gqps{{\mathbf{G}_{q,p,s}}}
\def\Gqpu{{\mathbf{G}_{q,p,u}}}
\def\Gqr{{\mathbf{G}_{q}^r}}
\def\Gqpr{{\mathbf{G}_{q,p}^r}}
\def\Gqpsr{{\mathbf{G}_{q,p,s}^r}}

\def\Gqprime{{\mathbf{G}'_{q}}}
\def\Gqpprime{{\mathbf{G}'_{q,p}}}

\def\Eqvee{{\mathbf{E}_{q}^\vee}}
\def\Eqr{{\mathbf{E}_{q_r}}}

\begin{document}

\maketitle

\tableofcontents

\begin{abstract}
  La classification analytique locale et la description du groupe de Galois
  pour les équations aux $q$-différences linéaires analytiques complexes ont
  été obtenues par Ramis, Sauloy et Zhang \cite{RSZ,RS3} sous l'hypothèse que
  les pentes du polygone de Newton sont entières. Nous relâchons ici cette
  hypothèse.
\end{abstract}

\selectlanguage{english}

\begin{abstract}
  Title: ``Local analytic theory of $q$-difference equations with arbitrary slopes''. \\
  The local analytic classification and the description of the Galois group
  for complex linear analytic $q$-difference equations have been obtained by
  Ramis, Sauloy and Zhang \cite{RSZ,RS3} under the assumption that the slopes
  of the Newton polygon are integral. In this work, we relax that assumption.
\end{abstract}

\selectlanguage{french}

%%%%%%%%%%%%%%%%%%%%%%%%%%%%%%%%%%%%%%%%%%%%%%%%%%%%%%%%%%%%%%%%%%%%%%%%%%%%%

% 1

\section{Introduction}
\label{section:Introduction}

% 1.1

\subsection{La théorie analytique locale des équations aux $q$-différences}
\label{subsection:TALEQD}

Nous considérons ici des équations aux $q$-différences \emph{linéaires
complexes analytiques} écrites sous forme matricielle-vectorielle:
\begin{equation}
\label{eqn:EQD}
X(q z) = A(z) X(z).
\end{equation}
Le complexe $q \in \Cs$ est fixé une fois pour toutes et tel que\footnote{La
théorie pour le cas où $0 < \lmod q \rmod < 1$ est essentiellement la même,
mais le cas où $\lmod q \rmod = 1$ présente de tout autres difficultés (\og
petits diviseurs \fg), voir \cite{LDV2}.} $\lmod q \rmod > 1$.
Notant $\Ka$ le corps des germes méromorphes en $0$, on suppose de plus que
$A \in \GLn(\Ka)$. (Voir \ref{subsection:Notations} pour les conventions et
notations générales dans cet article.)

\begin{rem}
Les équations aux $q$-différences du $q$-calcul \og classique \fg\ s'écrivent
le plus souvent sous forme \emph{scalaire}
$f(q^n z) + a_1(z) f(q^{n-1} z) + \cdots + a_n(z) f(z) = 0$,
où $a_1,\ldots,a_n \in \Kg$ et où l'on peut supposer $a_n \neq 0$. L'équivalence
avec la forme matricielle ci-dessus est due à Birkhoff \cite{Birkhoff1}; pour
une preuve moderne du \emph{lemme du vecteur cyclique} correspondant, voir
\cite{JSAIF} (dans l'esprit de Birkhoff) ou \cite{LDV} (dans l'esprit de
Deligne et Katz).
\end{rem}

Comme pour les équations différentielles depuis Riemann, la motivation de départ
concerne des systèmes \emph{rationnels} $X(q z) = A(z) X(z)$, $A \in \GLn(\Kg)$;
mais l'étude globale de ces derniers sur la sphère de Riemann 
$\Sr := \C \cup \{\infty\}$ commence
par l'étude locale, analytique ou formelle. Or, les seuls points fixes de
$\Sr$ par la $q$-dilatation $z \mapsto qz$ sont $0$ et $\infty$, et l'étude
en $\infty$ se ramène facilement à celle en $0$ par le changement de variable
$w = 1/z$. En ce qui concerne les \og singularités intermédiaires \fg,
c'est-à-dire dans $\Cs = \Sr \setminus \{0,\infty\}$, nous ne disposons pas
encore d'une approche cohérente pour les étudier localement; voir par exemple
des tentatives en ce sens dans \cite{Roques2,RoquesSauloy,ORS}. \\

La théorie analytique locale\footnote{La théorie analytique \emph{globale},
initiée par Birkhoff dans \cite{Birkhoff1} (mais pour le cas fuchsien seulement)
tourne autour de deux thèmes: groupe de Galois et correspondance de Riemann-Hilbert;
voir par exemple \cite{Etingof,vdPS1,JSGAL,Roques2,ORS}.} (en $0$) telle qu'elle
s'est développée au troisième millénaire, largement sous l'influence de Jean-Pierre
Ramis, comporte les volets suivants:
\begin{enumerate}
\item Classification: voir \cite{RSZ}, ainsi que la thèse de Anton Eloy à
l'Université Paul Sabatier;
\item Théorie de Galois: voir \cite{JSGAL,RS3,VB,vdPS1} ainsi que divers articles
de Yves Andr\'e, Anne Duval, Lucia Di Vizio et Julien Roques;
\item Théorie asymptotique: voir \cite{RZ,RSZ};
\item Sommation de solutions divergentes: voir \cite{RZ,LDVZ}, ainsi que
l'article \cite{DREL} de Thomas Dreyfus et Anton Eloy.
\end{enumerate}

On n'abordera dans cet article que les deux premiers volets. Ceux-ci reposent
sur la découverte d'un $q$-analogue du phénomène de Stokes. Parmi les divers
avatars de ce $q$-analogue, nous nous appuierons sur celui décrit dans
\cite{JSStokes} et utilisé dans \cite{RSZ,RS3}.

% 1.2

\subsection{L'hypothèse antérieure des pentes entières}
\label{subsection:Hypothèse}

À toute équation \eqref{eqn:EQD} est attaché un invariant formel, son
\emph{polygone de Newton}, consistant essentiellement en la donnée de
pentes $\mu_1 < \cdots < \mu_k$, qui sont des rationnels; affectées de
multiplicités $r_1,\ldots,r_k$, qui sont des entiers naturels non nuls
tels que $r_1 \mu_1,\ldots,r_k \mu_k \in \Z$ et $r_1 + \cdots + r_k = n$.
Les résultats les plus forts de la théorie analytique locale, en
particulier ceux de \cite{RSZ,RS3}, ont été obtenus sous l'hypothèse
\og technique \fg\ que les pentes sont entières: $\mu_1,\ldots,\mu_k \in \Z$.
Cette hypothèse intervient principalement à travers le fait que l'on
dispose alors
\begin{itemize}
\item de formes normales explicites (dites de Birkhoff-Guenther),
\item et d'une description explicite du phénomène de Stokes.
\end{itemize}
Hors l'hypothèse d'intégrité des pentes, les seuls progrès notables ont
été, à ma connaissance, les suivants:
\begin{itemize}
\item Marius van der Put et Reversat ont obtenu dans \cite{vdPR} une
classification \emph{formelle} complète et la description implicite du groupe
de Galois \emph{formel} universel; ceci sans aucune condition.
\item Virginie Bugeaud a obtenu (article déjà mentionné), dans certains
cas, une généralisation de la forme normale explicite de Birkhoff-Guenther
et une description des opérateurs de Stokes et de leur rôle galoisien.
De plus, elle a rendu plus explicite la description du groupe de Galois formel
de \cite{vdPR}.
\end{itemize}
C'est d'ailleurs sur cette description explicite que nous nous appuierons,
car notre approche est elle-même assez calculatoire. \\

Dans ce travail, nous relâchons totalement l'hypothèse d'intégrité des
pentes et généralisons certains des principaux résultats de \cite{RSZ,RS3},
plus précisément:
\begin{itemize}
\item le $q$-analogue du théorème de classification de
Birkhoff-Malgrange-Sibuya\footnote{Ce théorème est traditionnellement appelé
\og de Malgrange-Sibuya \fg, mais selon Bernard Malgrange il se trouve en
substance dans l'\oe uvre de Birkhoff.} \cite{RSZ};
\item la description complète du groupe de Galois analytique universel,
du $q$-analogue du groupe de monodromie sauvage et par conséquent la
correspondance de Riemann-Hilbert qui s'en déduit, au sens de \cite{RS3}.
\end{itemize}
Nous n'obtenons \emph{pas}, en particulier, une généralisation de la
description explicite des opérateurs de Stokes; ni de la solution du
problème inverse qui figure dans \cite{RS3}. Ce dernier point est sans
doute accessible à partir de nos résultats, j'espère que d'autres s'en
chargeront.

% 1.3

\subsection{Contenu de cet article}
\label{subsection:Contenu}

Décrivons maintenant plus en détail le contenu de cet article. Les notations
et conventions générales seront résumées en \ref{subsection:Notations}. \\

La section \ref{section:structuregénérale} est consacrée au rappel
des principaux résultats de structure de la catégorie des équations
aux $q$-différences (ou catégorie des modules aux $q$-différences)
sur $\Ka$: propriétés abéliennes et tensorielles, polygone de Newton,
filtration et graduation par les pentes; références pour ces prérequis:
\cite{JSFIL,JSStokes,RSZ,RS3}. On explique aussi brièvement comment
cette structure fournit un cadre pour la classification (sous-section
\ref{subsubsection:prérequisclassificationlocale}) et la théorie de Galois
(sous-section \ref{subsubsection:prérequisGaloislocal}). Enfin, on introduit
en \ref{subsection:prérequisRamification} le formalisme de la ramification
qui permettra l'extension des résultats antérieurs au cas de pentes
arbitraires. \\

La section \ref{section:classification} concerne la classification analytique
locale. En \ref{subsection:classifpentesentières}, nous rappelons les résultats
obtenus dans \cite{JSStokes,RSZ} dans le cas des pentes entières (nous y précisons
également la description de la structure affine sur l'espace des classes). En
\ref{subsection:classifpentesarbitraires}, nous prouvons en toute généralité
le principal résultat de la première partie, le $q$-analogue du théorème de
Birkhoff-Malgrange-Sibuya (théorème \ref{thm:qBMSPA}). \\

Dans la section \ref{section:groupedeGaloislocal}, nous abordons la
description du groupe de Galois local. Ce groupe proalgébrique a, dans
tous les cas, la forme d'un produit semi-direct $\St \ltimes \Gqp$, où $\Gqp$
est le groupe formel\footnote{L'indice p signifie \og pur \fg, nous réservons
l'indice f pour \og fuchsien \fg. Notons d'ailleurs que le \emph{groupe formel}
$\Gqp$ est en réalité un groupe (pro)algébrique \dots} 
et où le \emph{groupe de Stokes} $\St$ est pro-unipotent.
Suivant Ramis (dans sa définition du \og groupe de monodromie sauvage \fg\
des équations différentielles irrégulières), nous formulons la correspondance
de Riemann-Hilbert en termes de représentations d'un objet hybride 
$L \ltimes \Gqp$,
où $L$ est une sous-algèbre de Lie graduée libre de $\st := \Lie(\St)$ telle que
$\exp(L) \ltimes \Gqp$ soit Zariski-dense dans $\St \ltimes \Gqp$. Dans le cas
des pentes entières, rappelé en \ref{subsection:groupedeGaloispentesentières},
la description de $\Gqp$ était simplette et tout le travail portait sur
la construction de $L$. Dans le cas général, nous devons recourir à la
description donnée par van der Put et Reversat, que nous rappelons en
\ref{subsection:groupedeGaloisformel} sous la forme complétée et précisée
par Virginie Bugeaud. Nous décrivons de manière détaillée le groupe $\Gqpr$
associé à la catégorie des objets dont la pente a pour dénominateur $r$; puis
l'action de ce groupe sur sur l'algèbre de Lie du groupe de Stokes dans
\ref{subsection:groupedeGaloispentesarbitraires}. Nous pouvons alors énoncer
et prouver le résultat complet en \ref{subsection:structgroufonsau}. Les
résultats principaux sont les théorèmes \ref{thm:conjuguésmeilleurebasecanonique}
et \ref{thm:groupefondamentalsauvage}.

% 1.4

\subsection{Notations et conventions générales}
\label{subsection:Notations}

Le complexe $q$ ayant été fixé tel que $\lmod q \rmod > 1$, on notera $\sq$
l'opérateur aux $q$-différences défini par $(\sq f)(z) := f(q z)$. Ainsi
l'équation \eqref{eqn:EQD} se réécrit:
\begin{equation}
\label{eqn:eqd}
\sq X = A X,
\end{equation}
forme sous laquelle nous l'envisagerons désormais. \\

On notera de manière abrégée $\O$ et $K$ l'anneau et le corps d'intérêt: 
$\O := \Ra$ (anneau des germes holomorphes en $0$, ou, de façon équivalente,
des séries entières) et $K := \Ka = \O[1/z]$ l'anneau et le corps d'intérêt.
Leurs complétés $\Oc := \Rf$ et $\Kc := \Kf$ interviendront également. \\

On considérera $\Eq := \Cs/q^\Z$ parfois comme un groupe, parfois comme une
surface de Riemann (tore complexe), parfois comme une courbe elliptique, et
cela, sans changer de notation; par souci de l\'egèret\'e et parce que cela
ne pose aucun problème. Soit $\pi: \Cs \rightarrow \Eq$ la projection canonique:
c'est un morphisme de groupes et un revêtement non ramifié de degré infini.
On notera $\overline{a}$ l'image $\pi(a) \in \Eq$ de $a \in \Cs$; et
$[a;q] := \pi^{-1}(\overline{a}) \subset \Cs$ la spirale logarithmique discrète
$a q^\Z$.

\paragraph{Ramification.}

Pour tout $r \in \N^*$, on aura besoin de choisir\footnote{Les conséquences de ces
choix ne sont pour le moment absolument pas claires pour moi.} une racine $r$\ieme\
de $q$, notée $q_r$, ces choix étant supposés cohérents: $q_{rs}^s = q_r$. Pour cela,
nous décidons de choisir une fois pour toutes $\tau \in \C$ tel que $q = e^{2 \ii \pi \tau}$,
puis de poser, pour tout $x \in \C$, $q^x := e^{2 \ii \pi \tau x}$. En particulier, nous
noterons, pour tout $r \in \N^*$, $q_r := q^{1/r} = e^{2 \ii \pi \tau/r}$. Par ailleurs,
nous noterons $\zeta_r := e^{2 \ii \pi/r}$.

On choisit de même des extensions cycliques $K_r := \C(\{z_r\})$, $z_r^r = z$, 
avec la même convention $z_{rs}^s = z_r$; en vertu du théorème de Puiseux,
cela revient à fixer une cl\^oture algébrique 
$K_\infty := \bigcup\limits_{r \in \N^*} \C(\{z^{1/r}\})$ de $K = \Ka$.

Pour tout $r \geq 2$, l'extension $K_r= K[z_r]$ de $K$ est cyclique et son groupe
de Galois s'identifie au groupe $\mu_r := \{1,\zeta_r,\ldots,\zeta_r^{r-1}\}$ des
racines $r$\iemes\ de
l'unité dans $\C$: si $j \in \mu_r$, posant $\sigma_j(f(z_r)) := f(j z_r)$,
on voit que $j \mapsto \sigma_j$ est un isomorphisme de $\mu_r$ sur ce groupe
de Galois.

L'action de $\sq$ sur $K$ s'étend à naturellement à $K_r$ en l'automorphisme
$f(z_r) \mapsto f(q_r z_r)$, que l'on notera encore $\sq$ (on a donc une
\og extension de corps aux différences \fg). Un fait important est que
\emph{l'action de $\sq$ sur $K_r$ commute à celle de $\mu_r$}. \\

\emph{Pour alléger l'écriture, les catégories et groupes notés $\EE_q^{(0)}$,
$G^{(0)}$, etc, dans \cite{RSZ,RS3} seront ici notés $\EE_q$, $\Gq$, etc. De plus,
à partir de \ref{subsection:prérequisRamification}, les conditions de restriction
sur les pentes seront dénotées par des indices supérieurs: ainsi $\EE_q^1$
désigne la sous-catégorie pleine de $\EE_q$ formée des objets à pentes entières
(au lieu de $\EE_{q,1}^{(0)}$ dans \emph{loc. cit}).}

\paragraph{Fonctions.}

Nous noterons\footnote{Cette formule, ainsi que l'équation fonctionnelle
$\thq(qz) = z \thq(z)$, sont directement liées au choix de l'hypothèse
$\lmod q \rmod > 1$. Pour le cas de l'hypothèse symétrique $\lmod q \rmod < 1$,
voir par exemple \cite{ORS}.}:
$$
\thq(z) := \sum_{m \in \Z} q^{-m(m+1)/2} z^m.
$$
C'est une fonction holomorphe sur $\C^*$, qui vérifie les équations fonctionnelles:
$$
\thq(qz) = z \thq(z) = \thq(1/z)
$$
et la \emph{formule du triple produit de Jacobi}:
$$
\thq(z) = (-q^{-1};q^{-1})_\infty \; (-q^{-1}z;q^{-1})_\infty \; (-z^{-1};q^{-1})_\infty,
$$
où l'on définit les \emph{symboles de Pochhammer} par les produits infinis:
$$
(a;q^{-1})_\infty := \prod_{m \geq 0} (1 - q^{-m} a).
$$
On pose enfin, pour tout $c \in \C^*$:
$$
\thqc(z) := \thq(z/c),
$$
fonction dont les propriétés se déduisent immédiatement de celles de $\thq$.

\subsection*{Remerciements}

Une partie de ce travail a été effectuée alors que l'étais membre de l'Institut
Mathématique de Toulouse; l'unique support en a été mon salaire de fonctionnaire.
Cependant, je remercie Julien Roques de m'avoir invité il y a trois ans à en
exposer une version préliminaire aux rencontres qu'il a organisées à Grenoble. \\

D'un bout à l'autre j'ai été encouragé par l'intérêt de mon ami, mentor et coauteur
Jean-Pierre Ramis. À l'occasion d'un incident de parcours (rencontre de \og mauvaises
valeurs \fg\ de $q$, voir au \ref{subsubsection:basecanonique} la définition
\ref{defn:bonnevaleur}), j'ai eu la chance d'effleurer des thèmes chers à Ramanujan;
et j'ai alors pu bénéficier de l'aide du Professeur Bruce Berndt et de celle de mon
ami et coauteur Changgui Zhang.

%%%%%%%%%%%%%%%%%%%%%%%%%%%%%%%%%%%%%%%%%%%%%%%%%%%%%%%%%%%%%%%%%%%%%%%%%%%%%

% 2

\section{Structure générale de la catégorie des équations aux $q$-différences}
\label{section:structuregénérale}

% 2.1

\subsection{Propriétés abéliennes et tensorielles des modules aux $q$-différences}
\label{subsection:propabeltens}

% 2.1.1

\subsubsection{La catégorie $\EE_q$}

Un modèle intrinsèque de l'équation \eqref{eqn:eqd} est le \emph{module
aux $q$-différences}, ici défini comme un couple $(V,\Phi)$ formé d'un
$K$-espace vectoriel de dimension finie $V$ et d'un automorphisme $\sq$-
linéaire $\Phi$ de $V$, autrement dit, un automorphisme du groupe $V$
satisfaisant la forme multiplicative de la règle de Leibniz:
$$
\forall a \in K \;,\; \forall x \in V \;,\; \Phi(ax) = \sq(a) \Phi(x).
$$
($\Phi$ est donc en particulier $\C$-linéaire.) Un morphisme de $(V,\Phi)$
dans $(V',\phi')$ est une application $K$-linéaire $F: V \rightarrow V'$
qui entrelace $\Phi$ et $\Phi'$: $\Phi' \circ F = F \circ \Phi$. On
définit ainsi une catégorie $\EE_q$. \\

Notant $\D := K\left<\sigma,\sigma^{-1}\right>$ l'anneau des polynômes
de Öre non commutatifs caractérisés par la relation de commutation tordue:
$$
\forall a \in K \;,\; \forall k \in \Z \;,\; \sigma^k.a = \sq^k(a) \sigma^k,
$$
on voit immédiatement que tout module aux $q$-différences $(V,\Phi)$ définit
une structure de $\D$-module à gauche sur $V$ par la règle:
$$
\forall P := \sum a_k \sigma^k \in \D \;,\; \forall x \in V \;,\;
P.x := \sum a_k \Phi^k(x).
$$
On identifie ainsi $\EE_q$ à la catégorie des $\D$-modules à gauche de longueur
finie. \\

Soit $\B$ une base de $V$. Alors $\Phi(\B)$ est également une base, et il
existe un unique $A \in \GLn(K)$ telle que $\B = A \Phi(\B)$. On obtient ainsi
un isomorphisme de modules aux $q$-différences $X \mapsto \B X$ de $(K^n,\Phi_A)$
dans $(V,\Phi)$, où l'on a posé\footnote{Le choix d'utiliser la matrice
inverse $A^{-1}$ vient de ce que, dans ce modèle, les solutions $X$ de
l'équation \eqref{eqn:eqd} sont les points fixes de $\Phi_A$.}
$\Phi_A(X) := A^{-1} \sq X$. La sous-catégorie pleine de $\EE_q$ dont
les objets sont les $(K^n,\Phi_A)$ est donc essentielle. On en complète la
description en notant qu'un morphisme $F: (K^n,\Phi_A) \rightarrow (K^p,\Phi_B)$
est une matrice $F \in \Mat_{p,n}(K)$ telle que $(\sq F) A = B F$. \\

On retrouve ainsi le vocabulaire de la classification: si $F$ est un
isomorphisme, $n = p$ et $B = (\sq F) A F^{-1} =: F[A]$ (\emph{transformation
de jauge}); on dira également que $A$ et $B$ (ou les équations associées)
sont analytiquement équivalentes. On étend cette notion au cas où $B = F[A]$
avec $A,B \in \GLn(K)$ et $F \in \GLn(\Kc)$: on dit alors que $A$ et $B$ sont
formellement équivalentes\footnote{Il n'est pas utile d'élargir la catégorie
$\EE_q$ en autorisant des objets formels $A \in \GLn(\Kc)$: nous verrons en
effet que toute telle matrice est formellement équivalente à une matrice à
coefficients dans $K$.}.

% 2.1.2

\subsubsection{Propriétés abéliennes et tensorielles de $\EE_q$}
\label{subsubsection:propabeltens}

Le centre de $\D$ est $\C$ et la catégorie des $\D$-modules à gauche est
donc abélienne $\C$-linéaire; il en est donc de même de sa sous-catégorie
$\EE_q$. \\

Si $A$ est triangulaire par blocs:
$A = \begin{pmatrix} A' & \star \\ 0 & A'' \end{pmatrix}$, $A' \in \GL_{n'}(K)$,
$A'' \in \GL_{n''}(K)$, $n' + n'' = n$, la suite exacte évidente
$0 \rightarrow K^{n'} \rightarrow K^n \rightarrow K^{n''} \rightarrow 0$
définit en fait une suite exacte
$0 \rightarrow (K^{n'},\Phi_{A'}) \rightarrow (K^n,\Phi_A) \rightarrow (K^{n''},\Phi_{A''})
\rightarrow 0$.
On démontre que toute suite exacte dans $\EE_q$ est isomorphe à une
suite exacte de cette forme. \\

Soit $(V,\Phi)$ un objet quelconque de $\EE_q$. Son dual $(V^*,\Phi^*)$
est défini en prenant pour $V^*$ le dual du $K$-espace vectoriel $V$ et pour
$\Phi^*$ la contragrédiente ${}^t \Phi^{-1}$. Plus généralement, le \og Hom
interne \fg\ $\underline{\Hom}((V,\Phi),(W,\Psi))$ s'obtient en munissant
le $K$-espace vectoriel $\Lin_K(V,W)$ de l'automorphisme $\sq$-linéaire
$f \mapsto \Psi \circ f \circ \Phi^{-1}$. Si l'on prend pour $(W,\Psi)$
l'objet unité $\underline{1} := (K,\sq)$, on retrouve (à identifications
canoniques près) le dual. \\

Soient $(V_1,\Phi_1)$ et $(V_2,\Phi_2)$ deux objets de $\EE_q$. Il existe
un unique automorphisme $\sq$-linéaire $\Phi$ de $V := V_1 \otimes_K V_2$ tel
que $\Phi(x_1 \otimes x_2) = \Phi_1(x_1) \otimes \Phi_2(x_2)$. On pose alors:
$(V_1,\Phi_1) \otimes (V_2,\Phi_2) := (V,\Phi)$. \\

Ces structures font de $\EE_q$, une catégorie tensorielle rigide. Pour en
faire une catégorie tannakienne neutre, il faut la munir \cite{DM} d'un foncteur
fibre sur $\C$. Cette construction est indirecte. On verra plus loin que toute
matrice $A \in \GLn(K)$ est analytiquement équivalente à une matrice régulière
sur $\C^*$ (et même à coefficients dans $\C[z,z^{-1}]$). On se restreint à la
sous-catégorie pleine correspondante des $(K^n,\Phi_A)$ (qui est donc équivalente
à $\EE_q$), on choisit un point-base arbitraire $z_0 \in \Cs$ et l'on considère
le foncteur $\tilde{\omega}_{z_0}: (K^n,\Phi_A) \leadsto \C^n$ dont l'effet sur les
morphismes est $F \leadsto F(z_0)$. En effet, sous les hypothèses indiquées, $F(z_0)$
est nécessairement bien défini.

\begin{rem}
Voici une version plus géométrique de cette construction. À toute matrice
$A \in \GLn(K)$, on associe le faisceau $\F_A$ sur $\Eq$ défini, pour tout
ouvert $V$ de $\Eq$, par:
$$
\F_A(V) := \{X \text{~holomorphe sur~} \pi^{-1}(V) \text{~près de~} 0
\tq \sq X = A X \text{~près de~} 0\}.
$$
C'est un module localement libre de rang $n$, qui définit donc un fibré vectoriel
holomorphe sur le tore complexe $\Eq$, et, par théorème GAGA, un fibré algébrique
sur la courbe elliptique $\Eq$; nous noterons encore $\F_A$ ces fibrés. Alors
l'opération $(K^n,\Phi_A) \leadsto \F_A$ est un foncteur fibre sur la courbe $\Eq$.
Chaque point de $\Eq$ donne donc lieu \cite{DF} à un foncteur fibre sur $\C$.
La première construction est plus riche que la deuxième en ce qu'elle porte
sur le fibré équivariant $\Cs \times \C^n$ qui relève $\F_A$ sur $\Cs$.
\end{rem}

Comme prévisible dans ce genre d'affaire, le choix du point base est sans importance
pour le résultat final mais peut faire une différence pour les calculs et raisonnements
intermédiaires. Les remarques suivantes sont donc ici appropriées:
\begin{enumerate}
\item Pour tout disque ouvert $D$ de $\Cs$ contenu dans une couronne de la forme
$\{z \in \Cs \tq r \leq \lmod z \rmod < r \lmod q \rmod\}$ (dilatation par un
$r > 0$ de la couronne fondamentale $1 \leq \lmod z \rmod < \lmod q \rmod$)
l'ouvert $\pi(D)$ de $\Eq$ est trivialisant pour le faisceau et le fibré ci-dessus.
Il y a donc des isomorphismes canoniques (des identifications) entre tous les
foncteurs fibres $\tilde{\omega}_{z_0}$, $z_0 \in D$.
\item Il y a également un isomorphisme canonique entre $\tilde{\omega}_{z_0}$ et
$\tilde{\omega}_{q z_0}$, celui qui, à un objet $A$, associe l'isomorphisme:
$$
A(z_0): \C^n = \tilde{\omega}_{z_0}(A) \rightarrow \C^n = \tilde{\omega}_{qz_0}(A).
$$
En effet, le fait que c'est une transformation naturelle vient de ce que, si $F$
est un morphisme de $A$ dans $B$, l'équation fonctionnelle $F(qz) A(z) = B(z) F(z)$
instanciée en $z_0$ traduit exactement la fonctorialité.
\end{enumerate}

En pratique, c'est pour un choix légèrement différent de foncteur fibre que nous
opterons (\cf\ \ref{subsubsection:prérequisGaloislocal}), mais les mêmes remarques
s'y appliqueront de manière évidente.

% 2.2

\subsection{Polygone de Newton, filtration par les pentes et gradué associé}
\label{subsection:filtrationparlespentes}

Dorénavant, nous identifierons $\EE_q$ à la sous-catégorie pleine des
$(K^n,\Phi_A)$; et chaque objet $M := (K^n,\Phi_A)$ à la matrice $A \in \GLn(K)$
lorsque cela simplifiera les formulations. Si nécessaire, nous imposerons de
plus des conditions supplémentaires à $A$, comme être holomorphe sur $\Cs$,
voire être à coefficients dans $\C[z,z^{-1}]$: ceci, du moment que la sous-catégorie
pleine correspondante est essentielle (\ie\ équivalente à $\EE_q$). Les résultats
de toute cette sous-section sont justifiés en détail dans \cite{JSFIL,RSZ}.

% 2.2.1

\subsubsection{Polygone de Newton}
\label{subsubsection:polygonedeNewton}

Le principal invariant formel attaché à l'équation \eqref{eqn:eqd} et
au module aux $q$-différences $M = (K^n,\Phi_A)$ qui la code est son
\emph{polygone de Newton}; nous ne décrirons pas sa construction (voir
\cite{JSFIL,RSZ}) mais indiquons sa nature et ses principales propriétés.
On associe à $A \in \GLn(K)$ des \emph{pentes} $\mu_1 < \cdots < \mu_k$,
qui sont des rationnels, et des multiplicités $r_1,\ldots,r_k$, qui sont des
naturels non nuls tels que $d_1 := r_1 \mu_1,\ldots, d_k :=r_k \mu_k \in \Z$
et $r_1 + \cdots + r_k = n$. Ces données codent une partie convexe $P$ de
$\R^2$ telle que $P + (\{0\} \times \R_+) = P$ dont la frontière est formée
de deux demi-droites verticale infinies vers le haut et des $k$ vecteurs
$(r_i,d_i)$ (de gauche à droite); $P$ est totalement déterminé à translation
près par ces données. Pour exprimer les propriétés du polygone de Newton,
notons $M$ un module aux $q$-différences et $r_M: \Q \rightarrow \N$ la
\emph{fonction de Newton}: elle a pour support $S(M) := \{\mu_1,\ldots,\mu_k\}$
et vérifie $\forall i = 1,\ldots,k \;,\; r_M(\mu_i) = r_i$. Alors:
\begin{enumerate}
\item Deux modules formellement isomorphes ont même fonction de Newton.
\item Si $0 \rightarrow M' \rightarrow M \rightarrow M'' \rightarrow 0$
est une suite exacte, $r_M = r_{M'} + r_{M''}$.
\item Si $M = M' \otimes M''$, on a pour tout $\mu \in \Q$:
$$
r_M(\mu) = \sum_{\mu' + \mu'' = \mu} r_{M'}(\mu') r_{M''}(\mu'').
$$
\end{enumerate}

Nous dirons que $M$ est \emph{pur isocline} si son polygone de Newton ne comporte
qu'une seule pente; et qu'il est \emph{pur} s'il est somme directe de modules purs
isoclines.

\paragraph{Ramification.}

On reprend les notations introduites en \ref{subsection:Notations}. Soit
$r \in \N^*$ et soit $z_r := z^{1/r}$. En posant $A'(z_r) := A(z_r^r)$, on définit
une équation aux $q_r$ différences $X'(q_r z_r) = A'(z_r) X'(z_r)$. On vérifie
alors que, avec les notations vues ci-dessus pour $A$, le polygone de Newton
de $A'$ consiste en les pentes $r \mu_i$ avec les multiplicités (inchangées) $r_i$.

% 2.2.2

\subsubsection{Filtration et graduation par les pentes}
\label{subsub:filgrad}

\begin{prop}
Tout module $M$ de $\EE_q$ admet une unique tour de sous-modules:
$$
0 = M_0 \subset M_1 \subset \cdots \subset M_k = M
$$
telle que chaque quotient $M_i/M_{i-1}$, $1 \leq i \leq k$, soit pur isocline
de pente $\mu_i$ et de rang $\dim_\C M_i = r_i$. (Les $\mu_i$, $r_i$ désignent
les pentes de $M$ et leurs multiplicités.)
\end{prop}

\begin{cor}
Toute matrice de $\GLn(K)$ est analytiquement équivalente à une matrice $A$
triangulaire supérieure par blocs, dont les blocs diagonaux $B_i \in \GL_{r_i}(K)$,
$i = 1,\ldots,k$, sont purs isoclines de pentes les $\mu_i$.
\end{cor}

Autrement dit:
\begin{equation}
\label{eqn:formestandardtriangulaire}
A = A_{U} := \begin{pmatrix}
B_{1}  & \ldots & \ldots & \ldots & \ldots \\
\ldots & \ldots & \ldots  & U_{i,j} & \ldots \\
0      & \ldots & \ldots   & \ldots & \ldots \\
\ldots & 0 & \ldots  & \ldots & \ldots \\
0      & \ldots & 0       & \ldots & B_{k}     
\end{pmatrix} \in \GLn(K),
\end{equation}

\begin{prop}
On définit ainsi une filtration $(M^{\geq \mu})_{\mu \in \Q}$ dont les sauts ont lieu
aux valeurs $\mu = \mu_i$. Le gradué associé est $\gr M := \bigoplus M_i/M_{i+1}$. 
\end{prop}

Les propriétés fonctorielles, abéliennes, tensorielles de la filtration sont
résumées par celle-ci \cite{JSFIL}:

\begin{prop}
\label{prop:foncteurgr}
Le foncteur $M \leadsto \gr M$ de $\EE_q$ dans la sous-catégorie pleine
$\EE_{q,p}$ des modules purs est exact, fidèle et $\otimes$-compatible.
\end{prop}

Soient $A$ de la forme \eqref{eqn:formestandardtriangulaire} décrite ci-dessus
et $M$ le module correspondant. La matrice diagonale par blocs $A_0$ obtenue
en remplaçant les $U_{i,j}$ par $0$ correspond au gradué $\gr M$:
\begin{equation}
\label{eqn:formestandarddiagonale}
A_{0} := \begin{pmatrix}
B_{1}  & \ldots & \ldots & \ldots & \ldots \\
\ldots & \ldots & \ldots  & 0 & \ldots \\
0      & \ldots & \ldots   & \ldots & \ldots \\
\ldots & 0 & \ldots  & \ldots & \ldots \\
0      & \ldots & 0       & \ldots & B_{k}    
\end{pmatrix} \in \GLn(K).
\end{equation}
Par convention, dans ces écritures, on a $\mu_1 < \cdots < \mu_k$ et chaque bloc
diagonal $B_i \in \GL_{r_i}(K)$ est pur isocline de pente $\mu_i$; on démontre de
plus que les blocs rectangulaires $U_{i,j} \in \Mat_{r_i,r_j}(K)$,
$1 \leq i < j \leq k$, peuvent être pris à coefficients dans $\C[z,z^{-1}]$. \\

À l'aide des résultats de \cite{vdPR} (qui seront rappelés plus loin), on
vérifie que les $B_i$ peuvent eux-mêmes être pris à coefficients dans
$\C[z,z^{-1}]$. Les objets répondant à cette description forment une
sous-catégorie tannakienne essentielle de $\EE_q$, et nous nous
restreindrons à cette sous-catégorie pour tout le reste de l'article.

% 2.2.3

\subsubsection{Catégorie formelle}
\label{subsubsection:catégorieformelle}

Dans la catégorie $\hat{\EE}_q$ obtenue en remplaçant $K$ par $\Kc$,
la filtration ci-dessus a encore lieu.

\begin{prop}
(i) Dans la catégorie $\hat{\EE}_q$, la filtration par les pentes est
canoniquement scindée, \ie\ le foncteur $\gr$ correspondant est isomorphe
au foncteur identité. \\
(ii) Deux modules (analytiques) sont formellement isomorphes si, et seulement
si, leurs gradués sont isomorphes (formellement ou analytiquement, cela revient
au même).
\end{prop}

On peut donc identifier $\hat{\EE}_q$ à la catégorie $\EE_{q,p}$ des modules purs (qui
a été introduite par la proposition \ref{prop:foncteurgr} de \ref{subsub:filgrad}) et
la classification formelle se ramène donc à la classification des modules purs (ce point
sera détaillé en \ref{subsubsection:prérequisclassificationlocale}). \\

Pour formuler l'énoncé matriciel correspondant, nous introduisons le sous-groupe
algébrique $\G$ de $\GLn$ formé des matrices triangulaires par blocs et ayant
pour blocs diagonaux les matrices identités $I_{r_1},\ldots,I_{r_k}$, \ie\ défini
par le format ci-dessous:
\begin{equation}
\label{eqn:formestandardautomorphisme}
\begin{pmatrix}
I_{r_1}  & \ldots & \ldots & \ldots & \ldots \\
\ldots & \ldots & \ldots  & \star & \ldots \\
0      & \ldots & \ldots   & \ldots & \ldots \\
\ldots & 0 & \ldots  & \ldots & \ldots \\
0      & \ldots & 0       & \ldots & I_{r_k}    
\end{pmatrix}.
\end{equation}

\begin{cor}
\label{cor:auto=id}
(i) On conserve les notations $A,A_0$ ci-dessus. Il existe un unique
$\hat{F} \in \G(\Kc)$ tel que $\hat{F}[A_0] = A$. \\
(ii) Si $\hat{F} \in \G(\Kc)$ est tel que $\hat{F}[A] = A$, alors 
$\hat{F} = I_n$.
\end{cor}

% 2.3

\subsection{Conséquences pour la théorie analytique locale}

% 2.3.1

\subsubsection{Conséquences pour la classification analytique locale}
\label{subsubsection:prérequisclassificationlocale}

La classification formelle étant admise, il est naturel de chercher à
déterminer le quotient d'une classe formelle par la relation d'équivalence
analytique. Cependant, on constate que ce passage au quotient ne donne pas
lieu à un espace \og raisonnable \fg\ (voir la remarque à la fin de ce
numéro). Inspirés par \cite{BV}, nous avons donc dans \cite{RSZ} rigidifié
les objets à classifier en considérant les couples $(M,u)$ formés d'un
module $M$ et d'un isomorphisme $u$ de son formalisé avec un objet formel
fixé (\emph{classification analytique isoformelle}). \\

Vue l'identification de $\hat{\EE}_q$ à $\EE_{q,p}$ décrite plus haut, cela
revient à fixer un module pur $M_0$ et à définir l'ensemble $\F(M_0)$ quotient
de l'ensemble des couples $(M,u: \gr M \rightarrow M_0)$ (où $u$ est un
isomorphisme) par la relation: $(M,u) \simeq (M',u')$ s'il existe
$f: M \rightarrow M'$ tel que $u = u' \circ \gr f$ (un tel $f$ est d'ailleurs
automatiquement un isomorphisme). En termes des matrices $A_0, A$ associées,
on a:
$$
\F(M_0) = \F(A_0) := \dfrac
   {\{A \in \GLn(K) \tq A \text{~est triangulaire supérieure par blocs et~}
     \gr A = A_0\}}
   {A \simeq B \text{~si, et seulement si~} B = F[A], A \in \G(K)} \cdot
   $$

\begin{thm}[\cite{RSZ}]
\label{thm:structaffineF(M_0)}
$\F(M_0)$ est naturellement muni d'une structure d'espace affine de
dimension $\sum\limits_{1 \leq i < j \leq k} r_i r_j (\mu_j - \mu_i)$.  
\end{thm}

\includegraphics[width=5cm]{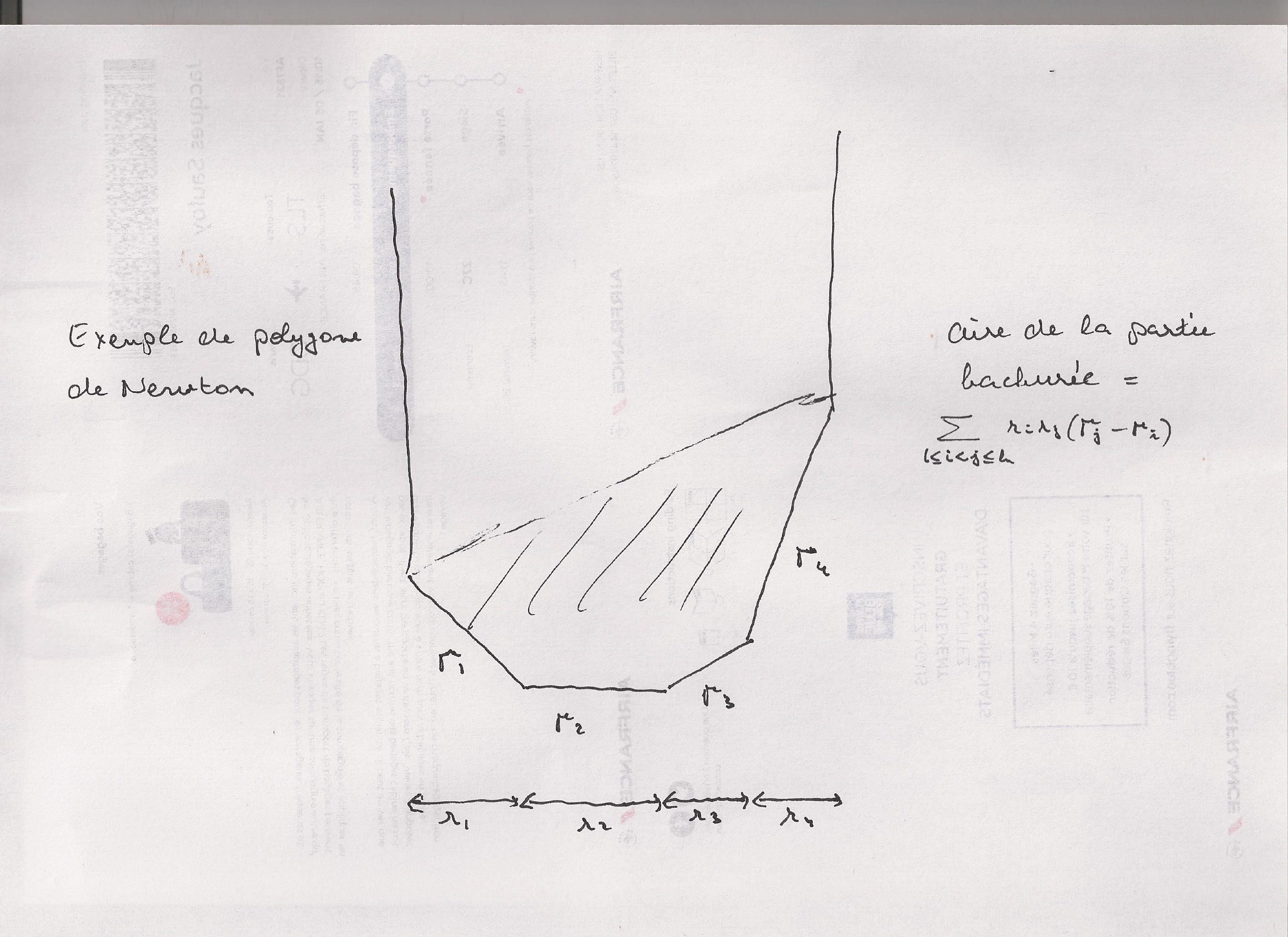}

Dans le cas où les pentes sont entières, on peut munir $\F(M_0)$ de
coordonnées explicites comme suit. Tout d'abord, tout module pur isocline
de pente $\mu \in \Z$ peut être décrit par une matrice de la forme $z^\mu C$
où $C \in \GL_r(\C)$. La matrice $A_0$ associée à $M_0$ peut donc être
choisie de blocs diagonaux $z^{\mu_i} A_i$, $A_i \in \GL_{r_i}(\C)$:
\begin{equation}
\label{eqn:formestandarddiagonaleentiere}
A_{0} := \begin{pmatrix}
z^{\mu_{1}} A_{1}  & \ldots & \ldots & \ldots & \ldots \\
\ldots & \ldots & \ldots  & 0 & \ldots \\
0      & \ldots & \ldots   & \ldots & \ldots \\
\ldots & 0 & \ldots  & \ldots & \ldots \\
0      & \ldots & 0       & \ldots & z^{\mu_{k}} A_{k}    
\end{pmatrix} \in \GLn(K).
\end{equation}
Cette matrice étant fixée, toute matrice $A$ triangulaire supérieure par blocs
et telle que $\gr A = A_0$:
\begin{equation}
\label{eqn:formestandardtriangulaireentière}
A = A_{U} := \begin{pmatrix}
z^{\mu_{1}} A_{1}  & \ldots & \ldots & \ldots & \ldots \\
\ldots & \ldots & \ldots  & U_{i,j} & \ldots \\
0      & \ldots & \ldots   & \ldots & \ldots \\
\ldots & 0 & \ldots  & \ldots & \ldots \\
0      & \ldots & 0       & \ldots & z^{\mu_{k}} A_{k}     
\end{pmatrix} \in \GLn(K),
\end{equation}
est équivalente (par la relation ci-dessus) à une unique matrice de même format
telle que de plus chaque bloc rectangulaire $U_{i,j}$ soit à coefficients dans
l'espace:
\begin{equation}
\label{eqn:polsdeBG}
K_{\mu_i,\mu_j} := \sum\limits_{\mu_i \leq d < \mu_j} \C z^d \subset \C[z,z^{-1}].
\end{equation}
C'est la \emph{forme normale de Birkhoff-Guenther}\footnote{Les matrices de cette
forme ont de plus l'avantage de former une sous-catégorie tannakienne
essentielle de la catégorie des objets de pentes entières.}. Les
coefficients de ces polynômes de Laurent sont au nombre de
$\sum\limits_{1 \leq i < j \leq k} r_i r_j (\mu_j - \mu_i)$ et forment le
système de coordonnées évoqué.

\begin{rem}
La relation d'équivalence analytique sur $\F(M_0)$ se déduit alors de l'action 
par conjugaison du groupe algébrique $\prod \Aut(A_i) \subset \prod \GL_{r_i}(\C)$
sur cet espace affine. L'étude des quotients a été abordée dans la thèse déjà
citée de Anton Eloy.
\end{rem}

% 2.3.2

\subsubsection{Conséquences pour la détermination du groupe de Galois local}
\label{subsubsection:prérequisGaloislocal}

Comme déjà dit, on considérera désormais les catégories $\EE_q$
et $\EE_{q,p}$ comme identifiées à leurs sous-catégories tannakiennes
essentielles dont les objets et les morphismes sont définis en tout point
de $\Cs$. On peut donc définir, pour tout point-base $z_0 \in \Cs$, un foncteur
fibre $\tilde{\omega}_{z_0}$ sur $\C$, selon la construction indiquée à la fin de 
\ref{subsubsection:propabeltens}; et l'on choisit une fois pour toutes un tel
$z_0$ arbitraire. Soit $\hat{\omega}_{z_0}$ la restriction de $\tilde{\omega}_{z_0}$
à $\EE_{q,p}$ (on n'indique donc plus la dépendance en $z_0$). Des propriétés
du foncteur $\gr$ énoncées au numéro \ref{subsection:filtrationparlespentes},
on déduit que $\omega_{z_0} := \hat{\omega}_{z_0} \circ \gr$ est un foncteur-fibre sur
$\EE_q$: c'est ce dernier que l'on utilisera dans la suite (ici et à la section
\ref{section:groupedeGaloislocal}). Les identifications entre les différents
$\tilde{\omega}_{z_0}$ mentionnées à la fin de \ref{subsubsection:propabeltens}
s'appliquent encore aux $\omega_{z_0}$. Pour simplifier, on omettra donc généralement
d'indiquer la dépendance en $z_0$ et l'on écrira donc $\tilde{\omega}$, $\hat{\omega}$
et $\omega$. \\

Soient $\Gq$ et $\Gqp$ les groupes de Galois\footnote{Sous sa forme la plus forte,
la théorie tannakienne définit des schémas en groupes (ici sur $\Cs$ ou même sur
$\Eq$), mais nous n'envisageons dans ce travail que les groupes proalgébriques
des points à valeurs dans $\C$ de ces schémas; ils contiennent en effet toute
l'information nécessaire.} tannakiens respectifs de $\EE_q$ et $\EE_{q,p}$, réalisés
comme groupes des automorphismes de $\omega$ et de $\hat{\omega}$. Du fait que
$\gr$ est une rétraction de l'inclusion $i$ de $\EE_{q,p}$ dans $\EE_q$, on déduit
par dualité tannakienne une suite exacte scindée
$$
\xymatrix{
  1 \ar@<0ex>[rr] & & \Stq \ar@<0ex>[rr] & & \Gq \ar@<0ex>[rr]_{i^*} & & 
  \Gqp \ar@/_1pc/[ll]_{\gr^*} \ar@<0ex>[rr] & & 1
}
$$
et une décomposition en produit semi-direct:
$$
\Gq = \Stq \ltimes \Gqp,
$$
où le \og groupe de Stokes \fg\ $\Stq$ est le noyau du morphisme
$i^*: \Gq \rightarrow \Gqp$. Il résulte des propriétés de la
filtration par les pentes que $\Stq$ est prounipotent. \\

Plus précisément, pour tout objet $A$ de $\EE_q$, de gradué $A_0$,
la dualité tannakienne associe à $A$, resp. $A_0$, une représentation
$\Gq \rightarrow \GLnc$, resp. $\Gqp \rightarrow \GLnc$, dont
l'image est un sous-groupe algébrique $\Gq(A)$, resp. $\Gqp(A_0)$,
de $\GLnc$. L'image $\Stq(A)$ de $\Stq$ par la première de ces représentations
est le noyau de $\Gq(A) \rightarrow \Gqp(A_0)$, un sous-groupe algébrique
de $\G(\C)$ (donc unipotent) tel que $\Gq(A) = \Stq(A) \ltimes \Gqp(A_0)$;
tandis que $\Gqp(A_0)$ est formé de matrices diagonales par blocs (de
tailles $r_1,\ldots,r_k$). L'action de $\Gqp(A_0)$ sur $\Stq(A)$ est la
conjugaison dans $\GLnc$. Enfin, les groupes proalgébriques $\Gq$, $\Stq$
et $\Gqp$ sont respectivement limites projectives des systèmes projectifs
des $(\Gq(A))$, $(\Stq(A))$ et $(\Gqp(A_0))$.

% 2.4

\subsection{Pentes et ramification}
\label{subsection:prérequisRamification}

% 2.4.1

\subsubsection{Restrictions sur les pentes}
\label{subsubsection:restrictionspentes}

Soit $r \in \N^*$. Des propriétés du polygone de Newton indiquées en
\ref{subsubsection:polygonedeNewton}, on déduit que la sous-catégorie
pleine $\EE_q^r$ de $\EE_q$ formée des objets dont toutes les pentes
appartiennent à $\frac{1}{r} \Z$ est tannakienne, de même que
$\EE_{q,p}^r := \EE_{q,p} \cap \EE_q^r$. On en déduit (par restriction
des foncteurs fibres à ces sous-catégories) des groupes tannakiens $\Gqr$,
$\Stqr$ et $\Gqpr$ tels que $\Gqr = \Stqr \ltimes \Gqpr$. Les résultats de
\cite{RSZ} (resp. de \cite{RS3}) concernent $\EE_q^1$ (resp. et $\mathbf{G}_q^1$). \\

Les catégories $\EE_q^r$ forment un système inductif indexé par $\N^*$
ordonné par la relation de divisibilité et la limite inductive est $\EE_q$;
de la même manière, la limite inductive des $\EE_{q,p}^r$ est $\EE_{q,p}$.
Cela entraîne que $\Gq$ est la limite projective des $\Gqr$ et que $\Gqp$ est
celle des $\Gqpr$; et donc (par commutation des noyaux aux limites projectives) que
$\Stq$ est celle des $\Stqr$:
$$
\EE_q = \limind\ \EE_q^r, \quad \EE_{q,p} = \limind\ \EE_{q,p}^r, \quad
\Gq = \limproj\ \Gqr, \quad \Gqp = \limproj\ \Gqpr, \quad
\Stq = \limproj\ \Stqr.
$$
Ces relations seront décrites de manière plus détaillée à la section
\ref{section:groupedeGaloislocal}.

% 2.4.2

\subsubsection{Foncteurs de ramification}
\label{subsubsection:foncteursramif}

Le but de cet article est d'étendre aux catégories $\EE_q^{r}$ une partie des résultats
obtenus dans \cite{RSZ,RS3} pour $\EE_q^{1}$. La méthode consiste à comparer $\EE_q^{r}$
à $\EE_{q_r}^1$ par ramification. \\

Fixons $r$. On définit un foncteur de $\EE_q$ dans $\EE_{q_r}$ en termes matriciels:
$A(z) \leadsto A'(z_r) := A(z_r^r)$ et $F(z) \leadsto F'(z_r) := F(z_r^r)$. On a vu
en \ref{subsubsection:polygonedeNewton} que les pentes sont ainsi multipliées par
$r$. On en déduit donc par restriction des foncteurs $\EE_q^{rs} \leadsto \EE_{q_r}^s$
et en particulier un foncteur de ramification:
$$
\text{Ram}_r: \EE_q^{r} \leadsto \EE_{q_r}^1.
$$

L'apparition de nouvelles \og bases \fg\footnote{Dans la terminologie historique
des \og basic hypergeometric series \fg, la \emph{base} du $q$-calcul est $q$.}
$q_r$ à c\^oté de $q$ nécessite de compliquer un peu les notations et d'appeler
$\Gq$, $\mathbf{G}_{q_r}$ les groupes tannakiens respectifs de $\EE_q$ et $\EE_{q_r}$,
et de même pour les groupes formels et de Stokes. \\

Si l'on note $z_{0,r}$ une racine $r$\ieme\ de $z_0$, on voit de plus que le foncteur
fibre correspondant $\omega'_{z_{0,r}}$ sur $\EE_{q_r}$ vérifie:
$$
\omega'_{z_{0,r}} \circ \text{Ram}_r = \omega_{z_0}.
$$
Par dualité tannakienne, cela donne lieu à des morphismes de groupes
proalgébriques:
$$
\text{Ram}_r^*: \mathbf{G}_{q_r}^1 \rightarrow \Gqr,
$$
et de même pour les groupes formels et de Stokes. Ces morphismes seront décrits
de manière plus détaillée à la section \ref{section:groupedeGaloislocal}.

%%%%%%%%%%%%%%%%%%%%%%%%%%%%%%%%%%%%%%%%%%%%%%%%%%%%%%%%%%%%%%%%%%%%%%%%%%%%%

% 3

\section{Classification analytique locale}
\label{section:classification}

Le problème posé ici est celui de la classification par des \og invariants
transcendants \fg\ selon les termes de Birkhoff dans l'article fondateur
\cite{Birkhoff1}. Pour les équations différentielles fuchsiennes, l'objet
classifiant est la représentation de monodromie, dont le rôle pour les équations
aux $q$-différences est tenu, pour l'essentiel, par la matrice de connexion de
Birkhoff \cite{Birkhoff1,JSAIF,ORS}. Pour les équations différentielles irrégulières,
les invariants locaux sont (outre la monodromie locale) les opérateurs de Stokes,
le modèle est celui du théorème de Birkhoff-Malgrange-Sibuya, qui a été transposé
au cas des $q$-différences dans \cite{RSZ} sous l'hypothèse que les pentes sont
entières. Le principal résultat de cette section est le théorème \ref{thm:qBMSPA},
qui étend au cas de pentes arbitraires le résulata antérieur. Signalons cependant
que le \ref{subsubsection:structaff} complète et corrige des affirmations insuffisamment
justifiées de \cite{RSZ,RS3}.

% 3.1

\subsection{Résumé des résultats obtenus pour les pentes entières}
\label{subsection:classifpentesentières}

% 3.1.1

\subsubsection{Résultats antérieurs}
\label{subsubsection:résultatsantérieurs}

Nous commençons par un résultat pour lequel nous ne sommes pas capables
d'énoncer une généralisation au cas de pentes arbitraires. Une classe
formelle étant fixée par la donnée de la matrice pure $A_0$, on définit
un sous-ensemble fini $\Sigma_{A_0} \subset \Eq$ de \og directions de
sommation interdites \fg\ par des conditions\footnote{Par rapport à
divers travaux antérieurs, on permute ici les rôles de $c$ et de $-c$;
c'est en effet le $c$ choisi ici qui interviendra le plus fréquemment.}  
de \emph{résonance}:
$$
\overline{c} \in \Sigma_{A_0} \underset{def}{\Longleftrightarrow}
q^\Z c^{\mu_i} \Sp(A_i) \cap q^\Z c^{\mu_j} \Sp(A_j) \neq \emptyset.
$$
Le théorème suivant prend la forme platonique d'un énoncé d'existence, mais
des calculs explicites de telles \og sommations algébriques \fg\ seront décrits
en \ref{subsection:calculs explicites}.

\begin{thm}[Sommation algébrique, \cite{JSStokes}]
\label{thm:sommationalgébrique}
Quelle que soit la \og direction de sommation autorisée \fg\
$\overline{c} \in \Eq \setminus \Sigma_{A_0}$, il existe une unique
transformation de jauge $F_{\overline{c}} \in \G(\M(\Cs))$ (autrement dit
méromorphe sur $\Cs$ et de la forme qui définit $\G$, voir l'équation
\eqref{eqn:formestandardautomorphisme}) telle que $F_{\overline{c}}[A_0] = A$,
et dont les blocs rectangulaires $F_{i,j}$, vérifient:
\begin{equation}
\label{eqn:polaritésommationalgébrique}
\forall i,j \;,\; 1 \leq i < j \leq k \;,\;
\div_\Eq(F_{i,j}) \geq - (\mu_j - \mu_i) [\overline{-c}].
\end{equation}
\end{thm}

Expliquons cette dernière condition. La relation $F_{\overline{c}}[A_0] = A$
se réécrit $\sq F_{\overline{c}} = A F_{\overline{c}} A_0^{-1}$, ce qui entraîne
que les pôles des coefficients de $F_{\overline{c}}$ forment des $q$-spirales,
donc des images réciproques par $\pi: \Cs \rightarrow \Eq$ de points de $\Eq$.
L'inégalité ci-dessus dit que tous ces pôles sont sur la $q$-spirale 
$\pi^{-1}(\overline{-c}) = [-c;q]$ et que ceux des $r_i r_j$ coefficients 
de $F_{i,j}$ sont de multiplicité $\leq \mu_j - \mu_i$. La présence un peu désagréable
de signes $-$ dans la définition de $\Sigma_{A_0}$ est due au fait incournable suivant:
la fonction theta qui vérifie $\sq f = (z/a) f$ a ses zéros sur la spirale logarithmique
discrète $[-a;q]$. \\

Mentionnons pour usage ultérieur deux lemmes.

\begin{lem}
\label{lem:groupepolaire}
(i) Les matrices $F \in \G(\M(\Cs))$ dont les blocs rectangulaires $F_{i,j}$,
vérifient les conditions de polarité \eqref{eqn:polaritésommationalgébrique}
forment un groupe, que nous noterons $\G[\overline{c}]$. \\
(ii) Soient $A_U$ et $A_V$ de gradué $A_0$ et notons $F_U$, $F_V$ les
isomorphismes méromorphes obtenus par sommation algébrique dans la direction
autorisée $\overline{c}$. Alors $F_{U,V} := F_V F_U^{-1}$ est l'unique isomorphisme
méromorphe $A_U \rightarrow A_V$ appartenant à $\G[\overline{c}]$.
\end{lem}
\begin{proof}
Un calcul direct établit immédiatement (i), et (ii) découle alors du théorème.
\end{proof}

Le cas particulier suivant est important dans les questions de \og filtration
$q$-Gevrey \fg.

\begin{lem}
\label{lem:F_{U,V}etniveaudelta}
Supposons, avec les notations du lemme précédent, que $U_{i,j} = V_{i,j}$ aux
niveaux $\mu_j - \mu_i < \delta$ (cette relation sera notée plus loin
$A_U \equivd A_V$, \cf\ équation \eqref{eqn:defequivd} de
\ref{subsubsection:structaff}). Alors: \\
(i) Les blocs de niveaux $\mu_j - \mu_i < \delta$ de $F_U$ et $F_V$ coïncident
(c'est encore la relation $F_U \equivd F_V$). \\
(ii) Les blocs de niveaux $\mu_j - \mu_i < \delta$ de $F_{U,V}$ sont nuls et
ses blocs $F_{i,j}$ de niveau $\mu_j - \mu_i = \delta$ sont déterminés par les
équations:
$$
(\sq F_{i,j}) (z^{\mu_j} A_j) - (z^{\mu_i} A_i) F_{i,j} = V_{i,j} - U_{i,j},
$$
avec la condition $\div_\Eq(F_{i,j}) \geq - (\mu_j - \mu_i) [\overline{-c}]$.
\end{lem}
\begin{proof}
Soit $\Gd$ le sous-groupe distingué de $\G$ formé des éléments dont les blocs
de niveaux $\mu_j - \mu_i < \delta$ de $F_{U,V}$ sont nuls (ce sous-groupe sera
étudié de plus près en \ref{subsubsection:structaff}). On vérifie facilement
que $F_U \equivd F_V \Leftrightarrow F_{U,V} \in \Gd$ (congruence modulo $\Gd$).
Il suffit donc de prouver (ii). \\
La propriété $F_{U,V}[A_U] = A_V$ se traduit par le système d'équations en les
blocs $F_{i,j}$:
$$
\forall i,j \;,\; 1 \leq i < j \leq k \;,\;
U_{i,j} + \sum_{i < l < j} (\sq F_{i,l}) U_{l,j} + (\sq F_{i,j}) (z^{\mu_j} A_j) =
(z^{\mu_i} A_i) F_{i,j} + \sum_{i < l < j} V_{i,l} F_{l,j} + V_{i,j}.
$$
Pour $\mu_j - \mu_i < \delta$, on a $U_{i,j} = V_{i,j}$ et on voit immédiatement
que les solutions $F_{i,j} = 0$ conviennent: par unicité, ce sont donc les seules.
Les équations au niveau $\mu_j - \mu_i = \delta$ prennent alors bien la forme
indiquée dans le lemme.
\end{proof}

Rappelons (\cf\ \ref{subsection:filtrationparlespentes}) qu'il existe une
unique transformation de jauge formelle $\hat{F} \in \G(\Kc)$ telle que
$\hat{F}[A_0] = A$; elle est en générale divergente et l'on considère
$F_{\overline{c}}$ comme \og somme de $\hat{F}$ dans la direction
$\overline{c}$ \fg\ (au sens de la sommation des séries divergentes).

\begin{cor}
\label{cor:sommationalgébrique}
Soient $\overline{c}, \overline{d} \in \Eq \setminus \Sigma_{A_0}$.
Alors $F_{\overline{c},\overline{d}} := F_{\overline{c}}^{-1} F_{\overline{d}}$
vérifie: $F_{\overline{c},\overline{d}}[A_0] = A_0$; autrement dit, si l'on
note $F_{i,j}$ les blocs rectangulaires de $F_{\overline{c},\overline{d}}$,
on a: $\sq F_{i,j} = z^{-(\mu_j - \mu_i)} A_i F_{i,j} A_j^{-1}$. De plus,
$\div_\Eq(F_{i,j}) \geq - (\mu_j - \mu_i) ([\overline{-c}] + [\overline{-d}])$,
autrement dit, $F_{i,j}$ est méromorphe sur $\Cs$, tous ses pôles sont sur
$[-c,-d;q] := [-c;q] \cup [-d;q]$ et leurs multiplicités sont $\leq \mu_j - \mu_i$.
\end{cor}

De la relation évidente
$F_{\overline{c},\overline{d}} F_{\overline{d},\overline{e}} = F_{\overline{c},\overline{e}}$,
on déduit que les $F_{\overline{c},\overline{d}}$ forment un cocycle pour un
certain faisceau de groupes (non abéliens) $\Lambda_I(A_0)$ sur $\Eq$
défini comme suit: les sections de $\Lambda_I(A_0)$ sur l'ouvert $U$
de $\Eq$ sont les $F \in \G(\O(\pi^{-1}(U)))$ telles que $F[A_0] = A_0$.
Ce cocycle est attaché au recouvrement $\Ua$ de $\Eq$ formé des ouverts
de Zariski $U_{\overline{c}} := \Eq \setminus \{\overline{-c}\}$,
$\overline{c} \in \Eq \setminus \Sigma_{A_0}$. On traduit les relations
$\div_\Eq(F_{i,j}) \geq - (\mu_j - \mu_i) ([\overline{-c}] + [\overline{-d}])$
en disant que c'est un cocycle \emph{privilégié}. L'ensemble des cocycles
privilégiés est noté $Z^1_{pr}(\Ua,\Lambda_I(A_0))$. Le théorème suivant est
un $q$-analogue du théorème de Birkhoff-Malgrange-Sibuya:

\begin{thm} \cite{RSZ}
\label{thm:qBMS}
Le cocycle $(F_{\overline{c},\overline{d}})$ ne dépend que de la classe de $A$
dans $\F(A_0)$. Les applications:  
$$
\F(A_0) \rightarrow Z^1_{pr}(\Ua,\Lambda_I(A_0)) 
\rightarrow H^1(\Eq,\Lambda_I(A_0))
$$
ainsi définies sont bijectives.
\end{thm}

Il s'agit bien entendu de cohomologie non abélienne \cite{Frenkel} (mais,
les groupes impliqués étant unipotents, la non-abélianité est modérée !) et
les bijections ci-dessus sont en fait des isomorphismes d'ensembles pointés:
la classe de $(A_0, \Id)$ dans $\F(A_0)$ correspond au cocycle trivial dans
$Z^1_{pr}(\Ua,\Lambda_I(A_0))$ et à la classe de cohomologie triviale dans
$H^1(\Eq,\Lambda_I(A_0))$.

% 3.1.2

\subsubsection{Compléments}
\label{subsubsection:compléments}

Le cas général qui sera abordé en \ref{subsection:classifpentesarbitraires}
se ramènera au cas des pentes entières par ramification. Pour garantir que
les constructions sont indépendantes des choix arbitraires, nous aurons besoin
d'une nouvelle notion et de deux lemmes.

\begin{defn}
Soit $A$ de la forme \eqref{eqn:formestandardtriangulaireentière} 
(donc triangulaire supérieure par blocs) dans $\EE_q$ et soit
$A_0 := \gr A$. On ne fait ici aucune hypothèse sur les pentes. Une
\emph{famille trivialisante adaptée à $A$} est la donnée:
\begin{itemize}
\item d'un recouvrement $\U := (U_\alpha)$ de $\Eq$ par des ouverts de
Zariski;
\item d'isomorphismes méromorphes $F_\alpha: A_0 \rightarrow A$, dans $\G$,
 chaque $F_\alpha$ étant holomorphe sur $U_\alpha$.
\end{itemize}
\end{defn}

On voit alors que la famille des $F_{\alpha,\beta} := F_\alpha^{-1} F_\beta$
est un cocycle de $Z^1(\U,\Lambda_I(A_0))$.

\begin{lem}
Tous les cocycles provenant de \FTAs\ ont même classe dans $H^1(\Eq,\Lambda_I(A_0))$.
\end{lem}
\begin{proof}
Soient deux \FTAs\ $((U_\alpha),(F_\alpha))$ et $((V_\beta),(G_\beta))$. Soit
$(W_\gamma)$ le recouvrement des $U_\alpha \cap V_\beta$, qui raffine à la fois
$(U_\alpha)$ et $(V_\beta)$ et soient $F_\gamma$, $G_\gamma$ les restrictions
correspondantes des $F_\alpha$ et des $G_\beta$. Posons
$H_\gamma := G_\gamma^{-1} F_\gamma \in \Gamma(W_\gamma,\Lambda_I(A_0))$. Alors
les $F_{\gamma,\delta} := F_{\gamma}^{-1} F_{\delta}$ et les
$G_{\gamma,\delta} := G_{\gamma}^{-1} G_{\delta}$ vérifient
$F_{\gamma,\delta} = H_\gamma^{-1} G_{\gamma,\delta} H_{\delta}$, donc définissent
la même classe.
\end{proof}

On voit donc que la \og sommation algébrique \fg\ n'est en somme qu'un moyen
explicite de faire apparaître une \FTA. \\

Soient maintenant $A_0$ et $A'_0$ de la forme 
\eqref{eqn:formestandarddiagonaleentiere} (donc diagonales par blocs) et
$\Phi: A_0 \rightarrow A'_0$ un isomorphisme: $A_0$ et $A'_0$ ont donc même
polygone de Newton. Alors la $\Phi$ est diagonale par blocs, de blocs
$\Phi_i \in \GL_{r_i}(\C)$, $i = 1,\ldots,k$. Si $A$ est triangulaire
supérieure par blocs de gradué $\gr A = A_0$, on vérifie immédiatement que
$A' := \Phi[A]$  est triangulaire supérieure par blocs de gradué
$\gr A' = A'_0$. De plus, si $B$ est triangulaire supérieure par blocs
de gradué $\gr B = A_0$, notant $B' := \Phi[B]$, toute équivalence
$F: A \rightarrow B$ dans $\G$ donne lieu à une équivalence
$\Phi F \Phi^{-1}: A' \rightarrow B'$ qui est dans $\mathfrak{G}_{A'_{0}} = \G$.
On définit ainsi une application bijective $\F(A_0) \rightarrow \F(A'_0)$. \\

Notant $A_0 = \Diag(z^{\mu_1} A_1,\ldots,z^{\mu_k} A_k)$, on a
$A'_0 = \Diag(z^{\mu_1} A'_1,\ldots,z^{\mu_k} A'_k)$ où $A'_i := \Phi_i A_i \Phi_i^{-1}$
d'où l'on tire que $\Sigma_{A_0} = \Sigma_{A'_0}$ et aussi que
$\Ua = {\mathfrak U}_{A'_{0}}$. Un cocycle privilégié
$(F_{\overline{c},\overline{d}}) \in Z^1_{pr}(\Ua,\Lambda_I(A_0))$ donne lieu à un cocycle
privilégié $(F'_{\overline{c},\overline{d}}) \in Z^1_{pr}({\mathfrak U}_{A'_{0}},\Lambda_I(A'_0))$
par la formule
$F'_{\overline{c},\overline{d}} := \Phi F'_{\overline{c},\overline{d}} \Phi^{-1}$,
et on obtient ainsi un isomorphisme de $Z^1_{pr}(\Ua,\Lambda_I(A_0))$
sur $Z^1_{pr}({\mathfrak U}_{A'_{0}},\Lambda_I(A'_0))$. De même, les $F_{\overline{c}}$
obtenus par sommation algébrique pour $A$ (théorème \ref{thm:sommationalgébrique})
donnent lieu aux $F'_{\overline{c}}$ obtenus par sommation algébrique pour $A'$
par la formule $F'_{\overline{c}} := \Phi F'_{\overline{c}} \Phi^{-1}$. En résumé:

\begin{lem}
Ces constructions donnent lieu à un diagramme commutatif, dans lequel toutes
les flèches sont des isomorphismes d'ensembles pointés:
$$
\xymatrix{
\F(A_0) \ar@<0ex>[d] \ar@<0ex>[rr] & &
Z^1_{pr}(\Ua,\Lambda_I(A_0)) \ar@<0ex>[d] \ar@<0ex>[rr] & &
H^1(\Eq,\Lambda_I(A_0)) \ar@<0ex>[d] \\
\F(A'_0) \ar@<0ex>[rr] & &
Z^1_{pr}({\mathfrak U}_{A'_{0}},\Lambda_I(A'_0)) \ar@<0ex>[rr] & &
H^1(\Eq,\Lambda_I(A'_0))
}
$$
\end{lem}

% 3.1.3

\subsubsection{Structure affine sur $\F(A_0)$ et $H^1(\Eq,\Lambda_I(A_0))$}
\label{subsubsection:structaff}

La structure affine de $\F(A_0) = \F(M_0)$ (théorème \ref{thm:structaffineF(M_0)})
a été décrite dans \cite{RSZ}. Dans ce même travail, il est affirmé sans
justification suffisante\footnote{Le c\oe ur de l'argument fourni est le suivant:
si l'espace vectoriel de dimension finie $V$ opère librement sur l'ensemble pointé
$X$ et si le quotient $X'$ est muni d'une structure affine, alors $X$ admet une
structure affine canonique compatible avec ces données (et qui permet donc
d'identifier $X$ à $V \times X'$). C'est faux: il faut encore pour cela disposer
d'une \emph{section} $X' \rightarrow X$.} que les suites exactes de cohomologie
(non abélienne) permettent de munir $H^1(\Eq,\Lambda_I(A_0))$ d'une structure
affine telle que les bijections du théorème \ref{thm:qBMS} soient des isomorphismes
affines (\cf\ \cite[cor. 6.2.2 p. 93]{RSZ}). Cette structure est précisée, encore
sans justification suffisante, dans \cite{RS3} (voir p. 188 après l'explicitation
de la suite exacte en haut de la page). Nous complétons ici ces arguments en vue
d'une généralisation à la section \ref{subsection:classifpentesarbitraires}.

\paragraph{Divers groupes et algèbres de Lie.}

L'algèbre de Lie $\g := \Lie(\G)$ du groupe unipotent $\G$ décrit par l'équation
\eqref{eqn:formestandardautomorphisme} à la fin de la sous-section
\ref{subsubsection:catégorieformelle} est définie par le format:
\begin{equation}
\label{eqn:formestandardendomorphismenilpotent}
\begin{pmatrix}
0_{r_1}  & \ldots & \ldots & \ldots & \ldots \\
\ldots & \ldots & \ldots  & \star & \ldots \\
0      & \ldots & \ldots   & \ldots & \ldots \\
\ldots & 0 & \ldots  & \ldots & \ldots \\
0      & \ldots & 0       & \ldots & 0_{r_k}    
\end{pmatrix}.
\end{equation}
C'est une algèbre de Lie nilpotente. Quelque soit la $\C$-algèbre, les applications
$x \mapsto I_n + x$ et $x \mapsto \exp x$ réalisent des isomorphismes (au sens de la
géométrie algébrique et différentielle) de $\g(R)$ dans $\G(R)$. \\

Nous allons munir $\G$, resp. $\g$, d'une filtration par des sous-groupes distingués,
resp. par des idéaux\footnote{Ces filtrations sont ici décrites de manière purement
algébrique, comme plus bas celles des faisceaux $\Lambda_I(A_0)$ et $\lambda_I(A_0)$.
Leur interprétation $q$-Gevrey est formulée dans \cite[\S 3.4, p. 36]{RSZ} et dans
\cite[p. 188]{RS3}.}. Pour tout $\delta > 0$, on pose:
\begin{equation}
\label{eqn:filtrationsousgroupes}
\Gd := \left\{
\begin{pmatrix}
I_{r_1}  & \ldots & \ldots & \ldots & \ldots \\
\ldots & \ldots & \ldots  & F_{i,j} & \ldots \\
0      & \ldots & \ldots   & \ldots & \ldots \\
\ldots & 0 & \ldots  & \ldots & \ldots \\
0      & \ldots & 0       & \ldots & I_{r_k}    
\end{pmatrix}
\tq 0 < \mu_j - \mu_i < \delta \Rightarrow F_{i,j} = 0 \right\}
\end{equation}
et:
\begin{equation}
\label{eqn:filtrationidéaux}
\gd := \left\{
\begin{pmatrix}
0_{r_1}  & \ldots & \ldots & \ldots & \ldots \\
\ldots & \ldots & \ldots  & F_{i,j} & \ldots \\
0      & \ldots & \ldots   & \ldots & \ldots \\
\ldots & 0 & \ldots  & \ldots & \ldots \\
0      & \ldots & 0       & \ldots & 0_{r_k}    
\end{pmatrix}
\tq 0 < \mu_j - \mu_i < \delta \Rightarrow F_{i,j} = 0 \right\},
\end{equation}
autrement dit, tous les \emph{niveaux $q$-Gevrey} $\mu_j - \mu_i < \delta$,
$1 \leq i < j \leq k$, sont nuls. Ainsi, $\Lie(\Gd) = \gd$ et les applications
$x \mapsto I_n + x$ et $x \mapsto \exp x$ réalisent des isomorphismes de $\gd$
dans $\Gd$. \\

Ces filtrations sont exhaustives (elles démarrent pour $\delta$ petit à l'espace
entier) et séparées (elles stationnent pour $\delta$ grand au sous-groupe trivial).
Elles ont un nombre fini de crans rationnels (les valeurs possibles de $\mu_j - \mu_i$).
On note selon l'usage:
$$
\Gds := \bigcup_{\delta' > \delta} \mathfrak{G}^{\geq \delta'}_{A_{0}}
\quad \text{~et~} \quad
\gds := \bigcup_{\delta' > \delta} \mathfrak{g}^{\geq \delta'}_{A_{0}}.
$$
En particulier, si $A_0$ est à pentes entières,
$\Gds = \mathfrak{G}^{\geq \delta+1}_{A_{0}}$ et
$\gds = \mathfrak{g}^{\geq \delta+1}_{A_{0}}$. \\

Soit $\gde$le sous espace de $\g$ défini par l'annulation de tous les niveaux
$q$-Gevrey $\mu_j - \mu_i \neq \delta$. On a donc $\gd = \gde \oplus \gds$.
De plus, les applications $x \mapsto I_n + x$ et $x \mapsto \exp x$ réalisent
un isomorphisme (le même) de $\gde$ dans $\Gd/\Gds$ et l'on a une suite exacte
d'extension centrale\footnote{Nous notons $0$ le groupe à gauche en vertu de la
notation additive sur $\g$ et $1$ le groupe à droite en vertu de la notation
multiplicative sur $\G$.}:
$$
0 \longrightarrow \gde \longrightarrow \G/\Gds \longrightarrow \Gd/\Gds \longrightarrow 1.
$$

\paragraph{Divers faisceaux et suites exactes.}

Le faisceau sur $\Eq$ de groupes unipotents $\Lambda_I(A_0)$ donne lieu à un faisceau
$\lambda_I(A_0) := \Lie\left(\Lambda_I(A_0)\right)$ d'algèbres de Lie nilpotentes; pour
tout ouvert $U$ de $\Eq$:
$$
\lambda_I(A_0)(U) = \{F \in \g\left(\O(\pi^{-1}(U))\right) \tq (\sq F) A_0 = A_0 F\}.
$$
Les filtrations de $\G$ et $\g$ introduites plus haut donnent lieu à des filtrations
des faisceaux $\Lambda_I(A_0)$ et $\lambda_I(A_0)$. Par exemple:
$$
\Lambda^{\geq \delta}_I(A_0)(U) := 
\Lambda_I(A_0)(U) \cap \Gd\left(\O(\pi^{-1}(U))\right) =
\{F \in \Gd\left(\O(\pi^{-1}(U))\right) \tq (\sq F) A_0 = A_0 F\},
$$
etc; on définit de même $\Lambda^{> \delta}_I(A_0)$, ainsi que
$\lambda^{\geq \delta}_I(A_0)$, $\lambda^{> \delta}_I(A_0)$ et
$\lambda^{(\delta)}_I(A_0)$. On a encore une suite exacte d'extension centrale:
\begin{equation}
\label{eqn:extcentralefaisceaux}
0 \longrightarrow \lambda^{(\delta)}_I(A_0) \longrightarrow 
\Lambda_I(A_0)/\Lambda^{> \delta}_I(A_0) \longrightarrow 
\Lambda_I(A_0)/\Lambda^{\geq \delta}_I(A_0) \longrightarrow 1.
\end{equation}

Pour $1 \leq i < j \leq k$, notons $\mathfrak{g}^{(i,j)}_{A_0} \subset \g$
le sous-espace formé des matrices dont tous les blocs autres que $F_{i,j}$
sont nuls et $\lambda^{(i,j)}_I(A_0) \subset \lambda_I(A_0)$ le sous-faisceau
correspondant. On a donc:
$$
\gd = \bigoplus_{\mu_j - \mu_i = \delta} \mathfrak{g}^{(i,j)}_{A_0}
\quad \text{~et~} \quad
\lambda^{(\delta)}_I(A_0) = \bigoplus_{\mu_j - \mu_i = \delta} \lambda^{(i,j)}_I(A_0).
$$
Identifiant une matrice de $\mathfrak{g}^{(i,j)}_{A_0}(R)$ à son bloc
$F_{i,j} \in \Mat_{r_i,r_j}(R)$, on peut écrire pour tout ouvert $U$ de $\Eq$:
$$
\lambda^{(i,j)}_I(A_0)(U) =
\left\{F \in \Mat_{r_i,r_j}\left(\O(\pi^{-1}(U))\right) \tq
(\sq F) (z^{\mu_j} A_j) = (z^{\mu_i} A_i) F \right\}.
$$
Ce faisceau est un fibré vectoriel sur $\Eq$ de rang $r_i r_j$ et de degré
$\mu_i - \mu_j$ (\cf\ \cite[6.2 et 6.3]{RSZ}). Le faisceau $\lambda^{(\delta)}_I(A_0)$
est donc un fibré vectoriel sur $\Eq$ de rang
$\sum\limits_{\mu_j - \mu_i = \delta} r_i r_j$ et de degré $-\delta$. En conséquence:
$$
V^{(i,j)} := H^1\left(\Eq,\lambda^{(i,j)}_I(A_0)\right)
\quad \text{~et~} \quad
V^{(\delta)} := H^1\left(\Eq,\lambda^{(\delta)}_I(A_0)\right) =
\bigoplus_{\mu_j - \mu_i = \delta} V^{(i,j)}
$$
sont des $\C$-espaces vectoriels de dimensions respectives
$r_i r_j (\mu_j - \mu_i)$ et $\delta \sum\limits_{\mu_j - \mu_i = \delta} r_i r_j$. \\

De l'extension centrale \eqref{eqn:extcentralefaisceaux} on déduit, selon Frenkel
\cite{Frenkel} la suite exacte de cohomologie non abélienne:
\begin{equation}
\label{eqn:secna}
0 \longrightarrow V^{(\delta)} \longrightarrow
H^1\left(\Eq,\Lambda_I(A_0)/\Lambda^{> \delta}_I(A_0)\right) \longrightarrow
H^1\left(\Eq,\Lambda_I(A_0)/\Lambda^{\geq \delta}_I(A_0)\right) 
\longrightarrow 1,
\end{equation}
qui signifie que $V^{(\delta)}$ opère librement sur l'ensemble pointé
$H^1\left(\Eq,\Lambda_I(A_0)/\Lambda^{> \delta}_I(A_0)\right)$
avec quotient $H^1\left(\Eq,\Lambda_I(A_0)/\Lambda^{\geq \delta}_I(A_0)\right)$.

\paragraph{Filtration $q$-Gevrey.}

À partir d'ici et pour toute la suite de \ref{subsubsection:structaff},
on suppose les pentes de $A_0$ entières. Avec la notation introduite en
\ref{subsubsection:prérequisclassificationlocale}, équation
\eqref{eqn:formestandardtriangulaireentière}, on note temporairement:
$$
\F_{A_0} := \{A_U \text{~en forme normale de Birkhoff-Guenther}\},
$$
autrement dit (\ref{subsubsection:prérequisclassificationlocale}, équation
\eqref{eqn:polsdeBG}), l'ensemble des $A_U$ telles que
$U_{i,j} \in \Mat_{r_i,r_j}(K_{\mu_i,\mu_j})$, $1 \leq i < j \leq k$. C'est donc
d'après \ref{subsubsection:prérequisclassificationlocale}
un ensemble de représentants pour la relation d'équivalence qui définit
$\F(A_0)$, \ie\ l'application naturelle $\F_{A_0} \rightarrow \F(A_0)$
est bijective. Plus précisément, l'application de
$\prod\limits_{1 \leq i < j \leq k} \Mat_{r_i,r_j}(K_{\mu_i,\mu_j})$ dans $\F(A_0)$
qui, à $(U_{i,j})_{1 \leq i < j \leq k}$ associe la classe de $A_U$,
est un isomorphisme d'espaces affines. \\

Pour tout $\delta \geq 0$, on définit une relation d'équivalence $\equivd$ sur
$\F_{A_0}$ par:
\begin{equation}
\label{eqn:defequivd}
A_U \equivd A_V \underset{def}{\Longleftrightarrow}
U_{i,j} = V_{i,j} \text{~pour~} 0 < \mu_j - \mu_i \leq \delta.
\end{equation}
Un ensemble de représentants pour le quotient $\frac{\F_{A_0}}{\equivd}$ est
alors donné par le sous-espace affine:
$$
\F_{A_0}^{\leq \delta} :=
\{A_U \in \F_{A_0} \tq U_{i,j} = 0 \text{~pour~} \mu_j - \mu_i > \delta\},
$$
autrement dit la flèche composée
$\F_{A_0}^{\leq \delta} \rightarrow \F_{A_0} \rightarrow \frac{\F_{A_0}}{\equivd}$
est bijective. Pour $0 \leq \delta' < \delta$, on a $\equivd \subset \equivdp$
et la surjection naturelle $\frac{\F_{A_0}}{\equivd} \to \frac{\F_{A_0}}{\equivdp}$
s'identifie à la troncature $\F_{A_0}^{\leq \delta} \to \F_{A_0}^{\leq \delta'}$, qui
remplace par $0$ tous les blocs $U_{i,j}$ tels que $\mu_j - \mu_i > \delta'$. \\

Notant, pour tout $\delta \geq 0$:
$$
W^{(\delta)} := \bigoplus_{\mu_j - \mu_i = \delta} \Mat_{r_i,r_j}(K_{\mu_i,\mu_j})
\quad \text{~et~}
W^{\leq \delta} := \bigoplus_{\delta' \leq \delta} W^{(\delta')},
$$
on a donc un isomorphisme d'espaces affines:
$$
W^{\leq \delta} \isom \F_{A_0}^{\leq \delta}.
$$

L'application de classification $\F_{A_0} \isom H^1\left(\Eq,\Lambda_I(A_0)\right)$
induit alors par passage au quotient une \emph{bijection}:
$$
\F_{A_0}^{\leq \delta} \isom H^1\left(\Eq,\dfrac{\Lambda_I(A_0)}{\Lambda^{> \delta}_I(A_0)}\right)
$$
et l'on a des diagrammes commutatifs à flèches verticales bijectives:
$$
\xymatrix{
\F_{A_0} \ar@<0ex>[d] \ar@<0ex>[r] & 
\F_{A_0}^{\leq \delta} \ar@<0ex>[d] \\
H^1\left(\Eq,\Lambda_I(A_0)\right) \ar@<0ex>[r] & 
H^1\left(\Eq,\dfrac{\Lambda_I(A_0)}{\Lambda^{\geq \delta+1}_I(A_0)}\right)
}
$$
et
$$
\xymatrix{
\F_{A_0}^{\leq \delta} \ar@<0ex>[d] \ar@<0ex>[r] & 
\F_{A_0}^{\leq \delta-1} \ar@<0ex>[d] \\
H^1\left(\Eq,\dfrac{\Lambda_I(A_0)}{\Lambda^{\geq \delta+1}_I(A_0)}\right)
\ar@<0ex>[r] & 
H^1\left(\Eq,\dfrac{\Lambda_I(A_0)}{\Lambda^{\geq \delta}_I(A_0)}\right)
}
$$
On a vu plus haut que la flèche horizontale basse de ce dernier diagramme
est un passage au quotient par l'action libre de $V^{(\delta)}$. Il est bien
évident que la flèche horizontale haute est un passage au quotient par l'action
libre de $W^{(\delta)}$ et que cette dernière est affine. On a décrit dans
\cite[\S 6.2.3 p. 93]{RSZ} et dans \cite[fin de \S 3.1 p. 189]{RS3} un
isomorphisme $W^{(\delta)} \isom V^{(\delta)}$ compatible avec ces données (il est
rappelé ci-dessous dans \ref{subsubsection:CCAC}, \og Calculs de cocycles \fg).
La section (inclusion naturelle)
$\F_{A_0}^{\leq \delta-1} \rightarrow \F_{A_0}^{\leq \delta}$ de la flèche
horizontale haute induit par transport une section
$H^1\left(\Eq,\Lambda_I(A_0)/\Lambda^{\geq \delta}_I(A_0)\right) \rightarrow
H^1\left(\Eq,\Lambda_I(A_0)/\Lambda^{\geq \delta+1}_I(A_0)\right)$
de la flèche horizontale basse compatible avec les données. En itérant, cela
permet de munir de proche en proche chaque ensemble
$H^1\left(\Eq,\Lambda_I(A_0)/\Lambda^{\geq \delta}_I(A_0)\right)$ d'une structure
affine, de telle sorte que:
\begin{enumerate}
\item les suites exactes de cohomologie \eqref{eqn:secna} soient scindées;
\item il s'en déduise un isomorphisme affine
$H^1\left(\Eq,\Lambda_I(A_0)\right) \isom \bigoplus\limits_{\delta > 0} V^{(\delta)}$;
\item la bijection $\F_{A_0} \isom H^1\left(\Eq,\Lambda_I(A_0)\right)$ soit un
  isomorphisme affine.
\end{enumerate}
Cela complète la justification annoncée.

% 3.1.4

\subsubsection{Calculs de cocycles et action de $\C^*$}
\label{subsubsection:CCAC}

\paragraph{Calculs de cocycles.}

Rappelons pour mémoire, et pour son utilité quand on traite d'exemples, le calcul
de l'application $W^{(\delta)} \isom V^{(\delta)}$ invoquée plus haut. Soient d'abord
$A_U$ et $A_V$ quelconques de gradué $A_0$ (pas nécessairement en forme normale de
Birkhoff-Guenther). L'équation $F[A_U] = A_V$ avec $F \in \G$ (abréviation pour dire
que $F$ est au format décrit dans \eqref{eqn:formestandardautomorphisme}) équivaut
au système:
$$
\forall i,j \;,\; 1 \leq i < j \leq k \;,\;
U_{i,j} + \sum_{i < l < j} (\sq F_{i,l}) U_{l,j} + (\sq F_{i,j}) (z^{\mu_j} A_j) =
(z^{\mu_i} A_i) F_{i,j} + \sum_{i < l < j} V_{i,l} F_{l,j} + V_{i,j}.
$$
Supposons maintenant que $A_U \equivd A_V$, et donc que leurs classes dans
$H^1\left(\Eq,\Lambda_I(A_0)/\Lambda^{\geq \delta}_I(A_0)\right)$ soient les
mêmes. Alors les équations ci-dessus admettent, aux niveaux
$\mu_j - \mu_i < \delta$, les solutions $F_{i,j} = 0$; et, si ces
solutions sont choisies, les équations au niveau $\delta$ deviennent:
$$
\forall i,j \;,\; 1 \leq i < j \leq k \;,\; \mu_j - \mu_i = \delta \;,\;
(\sq F_{i,j}) (z^{\mu_j} A_j) - (z^{\mu_i} A_i) F_{i,j} = V_{i,j} - U_{i,j}.
$$
On se place dans les conditions du théorème \ref{thm:sommationalgébrique}
(sommation algébrique dans une direction autorisée $\overline{c} \in \Eq$),
ce qui garantit l'unicité de la solution, notons la $F_{\overline{c}}$. Pour
$i,j$ fixés tels que $\mu_j - \mu_i = \delta$, les blocs $(i,j)$ des différences
$(F_{\overline{d}} - F_{\overline{c}})$ définissent un cocycle privilégié du faisceau
$\lambda_I^{(i,j)}(A_0)$ et la collection de ces cocycles est elle-même un cocycle
privilégié du faisceau $\lambda_I^{(\delta)}(A_0)$. La classe de ce cocycle est
alors l'unique élément $c_{U,V} \in V^{(\delta)}$ dont l'action amène la classe
de $A_U$ dans $H^1\left(\Eq,\Lambda_I(A_0)/\Lambda^{> \delta}_I(A_0)\right)$ sur
la classe de $A_V$ (c'est la suite exacte \eqref{eqn:secna} qui garantit \emph{a
priori} l'existence et l'unicité de cet élément). \\

Prenant $U = 0$ et $V \in W^{(\delta)}$ (étendu par $0$ aux niveaux $\neq \delta$)
on définit ainsi une application $V \mapsto c_{0,V}$ qui est l'isomorphisme
recherché $W^{(\delta)} \isom V^{(\delta)}$.

\paragraph{Action de $\C^*$.}

Nous en parlons ici pour usage ultérieur au \ref{subsubsection:classification}.
Soit $\lambda \in \C^*$ et notons $f^\lambda(z) := f(\lambda z)$,
$A_0^\lambda(z) := A_0(\lambda z)$, etc (donc $f^q = \sq f$). Du fait que
$f \mapsto f^\lambda$ commute avec $\sq$, on déduit que $A_U \mapsto A_U^\lambda$
induit une application $\F(A_0) \rightarrow \F(A_0^\lambda)$ et un diagramme
commutatif dont toutes les flèches sont des isomorphismes affines:
$$
\xymatrix{
\F_{A_0} \ar@<0ex>[d] \ar@<0ex>[rr] & &
\F(A_0) \ar@<0ex>[d] \ar@<0ex>[rr] & &
H^1(\Eq,\Lambda_I(A_0)) \ar@<0ex>[d] \\
\F_{A_0^\lambda} \ar@<0ex>[rr] & &
\F(A_0^\lambda) \ar@<0ex>[rr] & &
H^1(\Eq,\Lambda_I(A_0^\lambda))
}
$$
Nous appliquerons ce diagramme au cas où $A_0 = A_0^\lambda$.

% 3.2

\subsection{Extension au cas de pentes arbitraires}
\label{subsection:classifpentesarbitraires}

% 3.2.1

\subsubsection{Ramification}
\label{subsubsection:ramification}

Soit $r \geq 2$ un entier. On reprend les notations de \ref{subsection:Notations}
et de \ref{subsection:prérequisRamification} concernant la ramification; en
particulier, on notera $A'(z_r) := A(z_r^r) = A(z)$, etc. Si $A$ est triangulaire
supérieure par blocs de partie diagonale $A_0$, alors $A'$ est triangulaire
supérieure par blocs de partie diagonale $A'_0(z_r) = A_0(z_r^r) = A_0(z)$. Si
$F: A \rightarrow B$ est une transformation de jauge dans $\G(K)$, alors
$F': A' \rightarrow B'$ est une tranformation de jauge dans $\G(K_r)$. \\

Selon \ref{subsection:Notations}, le groupe $\mu_r$ des racines $r$\iemes\
de l'unité dans $\Cs$ opère sur $K_r$, donc sur $\GLn(K_r)$, et 
en particulier sur le sous-ensemble des matrices triangulaires supérieures 
par blocs de partie diagonale $A'_0$. Il est clair que l'application de
ramification $A(z) \mapsto A'(z_r)$ envoie $\F(A_0)$ dans le sous-ensemble
$\F(A'_0)^{\mu_r}$ des points fixes de $\F(A'_0)$.

\begin{prop}
Cette application est injective. (On verra en \ref{subsubsection:classification},
théorème \ref{thm:qBMSPA}, qu'elle est bijective.)
\end{prop}
\begin{proof}
Soient $A,B$ dans la classe de $A_0$ et soit $G: A' \rightarrow B'$ une tranformation
de jauge dans $\G(K_r)$. La partie $\mu_r$-fixe de $G(z_r)$, obtenue par le projecteur
habituel $\dfrac{1}{r} \sum\limits_{j \in \mu_r} \sigma_j$ est:
$$
F(z) := \dfrac{1}{r} \sum_{j \in \mu_r} G(j z_r).
$$
Or $F(z) \in \G(K)$: c'est une fonction de $z$ car $K_r^{\mu_r} = K$; la matrice
$F$ est évidemment triangulaire supérieure par blocs; et sa partie diagonale est
la partie fixe de la partie diagonale de $G$, qui est $I_n$: c'est donc encore
$I_n$. De plus, en appliquant le même projecteur à la relation
$G(q_r z_r) A'(z_r) = B'(z_r) G(z_r)$, on obtient, puisque $A'$ et $B'$ sont
$\mu_r$-invariants, $F(q z) A(z) = B(z) F(z)$, \ie\ $A$ et $B$ représentent
le même élément de $\F(A_0)$. On a donc défini une injection:
$$
\F(A_0)\rightarrow \F(A'_0)^{\mu_r}.
$$
\end{proof}

On va maintenant définir une injection analogue pour les $H^1$.
L'application $x \mapsto x^r$ de $\Cs$ dans lui-même envoie ${q_r}^\Z$
dans $q^\Z$ et passe donc au quotient en $p: \mathbf{E}_{q_r} \rightarrow \Eq$,
qui est surjective de noyau $\mu_r {q_r}^\Z/{q_r}^\Z \simeq \mu_r$, donc une
isogénie de degré $r$. Soit $\U = (U_\alpha)$ un recouvrement ouvert de $\Eq$.
Les $U'_\alpha := p^{-1}(U_\alpha)$ forment un recouvrement ouvert $\U'$ de
$\mathbf{E}_{q_r}$. Si $(F_{\alpha,\beta})$ est un cocycle de dans
$Z^1(\U,\Lambda_I(A_0))$, alors $(F'_{\alpha,\beta})$ est un cocycle de dans
$Z^1(\U',\Lambda_I(A'_0))$. Si 
$(G_{\alpha,\beta}) = (H_\alpha F_{\alpha,\beta} H_\beta^{-1})$ est un cocycle
équivalent à $(F_{\alpha,\beta})$, alors 
$(G'_{\alpha,\beta}) = (H'_\alpha F'_{\alpha,\beta} {H'_\beta}^{-1})$ est un cocycle
équivalent à $(F'_{\alpha,\beta})$. On a donc défini une application
$H^1(\Eq,\Lambda_I(A_0)) \rightarrow H^1(\mathbf{E}_{q_r},\Lambda_I(A'_0))$. \\

Conservant les mêmes notations, on a une action évidente de $\mu_r$ sur
$Z^1(\U',\Lambda_I(A'_0))$ (toujours par $f(z_r) \mapsto f(j z_r)$), et il est
clair que l'application vue plus haut envoie $Z^1(\U,\Lambda_I(A_0))$ dans
$Z^1(\U',\Lambda_I(A'_0))^{\mu_r}$. De plus, l'opération de $\mu_r$ est 
compatible avec l'équivalence des cocycles, et l'on a pour finir une
application de $H^1(\Eq,\Lambda_I(A_0))$ dans
$H^1(\mathbf{E}_{q_r},\Lambda_I(A'_0))^{\mu_r}$.

\begin{prop}
Cette application est injective. (On verra en \ref{subsubsection:classification},
théorème \ref{thm:qBMSPA}, qu'elle est bijective.)
\end{prop}
\begin{proof}
Une relation de cohomologie
$(G'_{\alpha,\beta}) = (H'_\alpha F'_{\alpha,\beta} {H'_\beta}^{-1})$ 
entre cocycles $(G'_{\alpha,\beta})$ et $(F'_{\alpha,\beta})$ provenant
respectivement par ramification de cocycles $(G_{\alpha,\beta})$ et 
$(F_{\alpha,\beta})$ s'écrit
$$
\forall \alpha,\beta \:;\; 
G'_{\alpha,\beta} H'_\beta = H'_\alpha F'_{\alpha,\beta},
$$
et, en appliquant le même opérateur de projection que ci-dessus, on obtient
$$
\forall \alpha,\beta \:;\; 
G_{\alpha,\beta} H_\beta = H_\alpha F_{\alpha,\beta},
$$
où les $H_\alpha$ s'obtiennent par application aux $H'_\alpha$ du projecteur,
et sont dans $\G$ comme précédemment, d'où une relation de cohomologie 
$(G_{\alpha,\beta}) = (H_\alpha F_{\alpha,\beta} H_\beta^{-1})$. On a donc défini
une injection: 
$$
H^1(\Eq,\Lambda_I(A_0)) \rightarrow H^1(\mathbf{E}_{q_r},\Lambda_I(A'_0))^{\mu_r}.
$$
\end{proof}

% 3.2.2

\subsubsection{Classification}
\label{subsubsection:classification}

Supposons maintenant que $r$ est un dénominateur commun des pentes de $A_0$,
de sorte que les pentes de $A'_0$ sont entières. Il est immédiat, par construction, 
que l'application $\F(A'_0) \rightarrow H^1(\mathbf{E}_{q_r},\Lambda_I(A'_0))$
définie au théorème \ref{thm:qBMS} est compatible avec l'opération de $\mu_r$
et que l'on a donc une bijection
$\F(A'_0)^{\mu_r} \rightarrow H^1(\mathbf{E}_{q_r},\Lambda_I(A'_0))^{\mu_r}$.

\begin{prop}
L'image de $\F(A_0)$ dans $\F(A'_0)$ est égale à $\F(A'_0)^{\mu_r}$.
\end{prop}
\begin{proof}
Notons $j$ un générateur de $\mu_r$ et $B^j$ pour $\sigma_j B$. Une classe 
de $\F(A'_0)$ fixée par $\mu_r$ est représentée par une matrice $B(z_r)$ 
triangulaire supérieure par blocs de gradué $\gr B = A'_0$ et telle que
$B \sim B^j$ (équivalence de jauge). On veut montrer que $B \sim C$, où
$\gr C = A'_0$ et $C = C^j$, \ie\ $C$ est fonction de $z_r^r = z$. \\
Soit donc $G: B \rightarrow B^j$ un isomorphisme $G \in \G(K_r)$. Comme
l'action de $\mu_r$ commute à celle de $\sq$, on voit que $G^j \in \G(K_r)$
est un isomorphisme $G^j: B^j \rightarrow B^{j^2}$, et l'on peut itérer.
En composant les $r$ isomorphismes ainsi obtenus, on trouve que
$G^{j^{r-1}} \cdots G^j G \in \G(K_r)$ est un automorphisme de $B$, donc
l'identité d'après le corollaire \ref{cor:auto=id}:
$$
G^{j^{r-1}} \cdots G^j G = I_n.
$$
On pose alors:
$$
H := \dfrac{1}{r} (I_r + G + G^j G + \cdots + G^{j^{r-1}} \cdots G^j G) \in \G(K_r),
$$
qui vérifie $H = H^j G$. Posant $C := H[B]$, on calcule:
$$
C^j = H^j[B^j] = H[B] = C,
$$
d'où la conclusion voulue.
\end{proof}

\begin{rem}
La parenté de cette démonstration avec celle du \og Théorème 90 de Hilbert \fg\ suggère 
qu'une bonne partie de nos arguments pourrait être simplifiée par des considérations
de cohomologie des groupes (voire de cohomologie galoisienne).
\end{rem}

On a actuellement un diagramme incomplet
$$
\xymatrix{
\F(A_0) \ar@<0ex>[rr] \ar@{.>}[d]_? & &
\F(A'_0)^{\mu_r} \ar@<0ex>[d] \\
H^1(\Eq,\Lambda_I(A_0) \ar@<0ex>[rr] & &
H^1(\mathbf{E}_{q_r},\Lambda_I(A'_0))^{\mu_r}
}
$$
Pour le compléter, partons de la classe de $A$ dans $\F(A_0)$ et posons
$A'(z_r) := A(z)$; soit $(G_\alpha)$ une famille trivialisante adaptée à $A'$,
par exemple celle obtenue par \og sommation algébrique \fg. Les parties
$\mu_r$-fixes $F_\alpha$ des $G_\alpha$, obtenues par le même opérateur de
projection que précédemment, forment une famille trivialisante adaptée à $A$
et définissent donc un cocycle  de $Z^1(\U,\Lambda_I(A_0))$, donc une classe
de $H^1(\Eq,\Lambda_I(A_0))$. Ceci définit sans ambiguïté une flèche
verticale gauche qui rend le diagramme ci-dessus commutatif. Du fait que
les flèches horizontale haute et verticale droite sont bijectives et que
la flèche horizontale basse est injective, il s'ensuit que toutes les flèches
sont bijectives:

\begin{thm}
\label{thm:qBMSPA}
On a un diagramme commutatif à flèches bijectives:
$$
\xymatrix{
\F(A_0) \ar@<0ex>[rr] \ar@<0ex>[d] & &
\F(A'_0)^{\mu_r} \ar@<0ex>[d] \\
H^1(\Eq,\Lambda_I(A_0) \ar@<0ex>[rr] & &
H^1(\mathbf{E}_{q_r},\Lambda_I(A'_0))^{\mu_r}
}
$$
\end{thm}

Il résulte de plus des arguments de \ref{subsubsection:compléments}
que la bijection $\F(A_0) \rightarrow H^1(\Eq,\Lambda_I(A_0)$ est 
indépendante du degré de ramification $r$.

%%%%%%%%%%%%%%%%%%%%%%%%%%%%%%%%%%%%%%%%%%%%%%%%%%%%%%%%%%%%%%%%%%%%%%%%%%%%%

% 4

\section{Groupe de Galois local}
\label{section:groupedeGaloislocal}

Rappelons que les notations concernant les catégories $\EE_q^r$, $\EE_{q,p}^r$,
etc, et les groupes $\Gqr$, $\Gqpr$, etc, ont été expliquées à la section
\ref{subsection:prérequisRamification}.

% 4.1

\subsection{Résumé des résultats obtenus pour les pentes entières \cite{RS3}}
\label{subsection:groupedeGaloispentesentières}

La catégorie des équations à pentes entières sur $K := \Ka = \O[1/z]$ (selon 
les notations indiquées en \ref{subsection:Notations}), resp. sa sous-catégorie
pleine des équations pures, est notée $\EE_q^1$, resp. $\EE_{q,p}^1$. Pour simplifier,
nous identifierons $\EE_q^1$ à sa sous-catégorie pleine essentielle des systèmes de
la forme \eqref{eqn:formestandardtriangulaireentière} et de même $\EE_{q,p}^1$ à sa
sous-catégorie pleine essentielle des systèmes de la forme
\eqref{eqn:formestandarddiagonaleentiere} (voir \ref{subsub:filgrad}). \\

Le foncteur \og gradué associé \fg\ $\gr: \EE_q^1 \leadsto \EE_{q,p}^1$,
qui est défini dans ce modèle par $\gr A_{U} = A_0$, est exact, $\C$-linéaire,
fidèle, $\otimes$-compatible, et c'est une rétraction de l'inclusion. L'effet
sur un morphisme $F: A \rightarrow B$ est le suivant. Soient $\mu_i,r_i$ les
données associées à $A \in \GLn(K)$ (polygone de Newton) et $\nu_j,s_j$ celles
associées à $B \in \GL_p(K)$. On a donc $F \in \Mat_{p,n}(K)$, que l'on décompose
en blocs $F_{j,i}$ de tailles $s_j \times r_i$. Alors $\gr F \in \Mat_{p,n}(K)$
s'obtient en remplaçant par $0$ tous les blocs $F_{j,i}$ tels que $\nu_j \neq \mu_i$. \\

Nous allons dans cette section décrire les groupes tannakiens respectifs $\mathbf{G}_q^1$
et $\mathbf{G}_{q,p}^1$ de $\EE_q^1$, resp. $\EE_{q,p}^1$, qui sont des groupes
proalgébriques, et la correspondance entre représentations rationnelles de ces
groupes et objets de ces catégories. Avec ces notations, notre but est donc de
décrire les équivalences de catégories:
$$
\EE_q^1 \longleftrightarrow \Rep(\mathbf{G}_q^1) \quad \text{et} \quad
\EE_{q,p}^1 \longleftrightarrow \Rep(\mathbf{G}_{q,p}^1).
$$
L'inclusion $i: \EE_{q,p}^1 \leadsto \EE_q^1$ et sa rétraction $\gr: \EE_q^1 \leadsto \EE_{q,p}^1$
donnent lieu par dualité tannakienne à des morphismes
$i^*: \mathbf{G}_q^1 \rightarrow \mathbf{G}_{q,p}^1$
et $\gr^*: \mathbf{G}_{q,p}^1 \rightarrow \mathbf{G}_q^1$ tels que
$i^* \circ \gr^* = \Id_{\mathbf{G}_{q,p}^1}$, donc à une
décomposition en produit semi-direct:
$$
\mathbf{G}_q^1 = \Stqu \ltimes \mathbf{G}_{q,p}^1,
$$
où le \emph{groupe de Stokes} $\Stqu := \Ker i^*$ est (pro)unipotent. \\

Le \emph{groupe formel $\mathbf{G}_{q,p}^1$} (rappelons qu'il s'agit en réalité
d'un groupe proalgébrique !) est abélien et donc produit de sa composante unipotente
et de sa composante semi-simple:
$$
\mathbf{G}_{q,p}^1 = \mathbf{G}_{q,p,u}^1 \times \mathbf{G}_{q,p,s}^1.
$$
La composante unipotente est $\mathbf{G}_{q,p,u}^1 = \C$ (elle agit sur les $q$-logarithmes)
et la composante semi-simple est $\mathbf{G}_{q,p,s}^1 = \Cs \times \Hom_{gr}(\Eq,\Cs)$, où
le premier facteur (\og tore theta \fg) agit sur les fonctions theta associées aux pentes
et le deuxième surles $q$-caractères associés aux exposants. \\

Nous reprendrons la terminologie classique des groupes diagonalisables:
\begin{itemize}
\item À tout groupe abélien $\Gamma$, on associe le groupe proalgébrique
$\Gamma^\vee := \Hom_{gr}(\Gamma,\Cs)$ des morphismes de groupe de $\Gamma$
dans $\Cs$; on voit que $\Gamma^\vee$ est proalgébrique car c'est la limite
projective des $\Gamma_i^\vee$, où les $\Gamma_i$ sont les sous-groupes de
type fini de $\Gamma$.
\item À tout groupe proalgébrique $G$, on associe le groupe $X(G)$ de ses
caractères, \ie\ de ses représentations rationnelles dans $\Cs$.
\end{itemize}
On a alors des identifications naturelles $X(\Gamma^\vee) = \Gamma$ et, pour
les groupes proalgébriques dits \og multiplicatifs \fg, $X(G)^\vee = G$. \\

Avec ces notations, on peut écrire:
$$
\mathbf{G}_{q,p,s}^1 = \Z^\vee \times \Eqvee,
$$
et donc:
$$
X\left(\mathbf{G}_{q,p,s}^1\right) = X\left(\Z^\vee\right) \times X\left(\Eqvee\right) =
\Z \times \Eq.
$$
Notons d'ailleurs que
$\Eqvee = \Hom_{gr}(\Eq,\Cs) = \Hom_{gr}(\Cs/q^\Z,\Cs)$ s'identifie
naturellement, via le morphisme $\gamma \mapsto \gamma \circ \pi$, 
au sous-groupe de $\Hom_{gr}(\Cs,\Cs)$ formé des $\gamma$
tels que $\gamma(q) = 1$. Dorénavant, nous opérerons sans commentaire
l'identification d'un tel $\gamma \in \Eqvee = \Hom_{gr}(\Eq,\Cs)$ avec
l'élément correspondant $\gamma \circ \pi \in \Hom_{gr}(\Cs,\Cs)$
tel que $q \mapsto 1$.

% 4.1.1

\subsubsection{Correspondance $\EE_{q,p}^1 \longleftrightarrow \Rep(\mathbf{G}_{q,p}^1)$}
\label{subsubsection:correspondanceEp-Rep(Gp)}

Fixons $A_0$ de la forme \eqref{eqn:formestandarddiagonaleentiere}. Soit
$A_i = A_{i,u} A_{i,s}$ la décomposition de Dunford multiplicative de la partie
constante $A_i$ du $i$\ieme\ bloc. La représentation
$\rho_{A_0}: \mathbf{G}_{q,p}^1 \rightarrow \GLnc$ qui lui est associée par
dualité tannakienne est définie, modulo l'identification de $\mathbf{G}_{q,p}^1$ à
$\C \times \Cs \times \Hom_{gr}(\Eq,\Cs)$, par:
$$
(\lambda,t,\gamma) \mapsto
\begin{pmatrix}
A_{1,u}^\lambda  & \ldots & \ldots & \ldots & \ldots \\
\ldots & \ldots & \ldots  & 0 & \ldots \\
0      & \ldots & \ldots   & \ldots & \ldots \\
\ldots & 0 & \ldots  & \ldots & \ldots \\
0      & \ldots & 0       & \ldots & A_{k,u}^\lambda    
\end{pmatrix}
\begin{pmatrix}
t^{\mu_{1}} I_{r_1} & \ldots & \ldots & \ldots & \ldots \\
\ldots & \ldots & \ldots  & 0 & \ldots \\
0      & \ldots & \ldots   & \ldots & \ldots \\
\ldots & 0 & \ldots  & \ldots & \ldots \\
0      & \ldots & 0       & \ldots & t^{\mu_{k}} I_{r_k}   
\end{pmatrix}
\begin{pmatrix}
\gamma(A_{1,s})  & \ldots & \ldots & \ldots & \ldots \\
\ldots & \ldots & \ldots  & 0 & \ldots \\
0      & \ldots & \ldots   & \ldots & \ldots \\
\ldots & 0 & \ldots  & \ldots & \ldots \\
0      & \ldots & 0       & \ldots & \gamma(A_{k,s})    
\end{pmatrix}.
$$
Noter que ce produit est commutatif. Dans le troisième facteur, il est entendu
que $\gamma$ est appliqué aux classes dans $\Eq$ des valeurs propres des $A_i$
(ou que $\gamma \circ \pi$ est appliqué aux valeurs propres elles-m\^emes),
\cf\ \cite{JSGAL}. \\

Dans l'identification du facteur $\Cs$ de $\mathbf{G}_{q,p,s}^1$ à $\Z^\vee$, le complexe
$t \in \Cs$ correspond au morphisme $h: m \mapsto t^m, \Z \rightarrow \Cs$.
Dans le deuxième facteur ci-dessus, on pourrait donc écrire $h(\mu_i)$
au lieu de $t^{\mu_i}$, et c'est sous cette forme que la description de
$\rho_{A_0}$ sera généralisée dans le cas de pentes non entières: le deuxième
facteur est donc la matrice diagonale par blocs de blocs les $h(\mu_i) \, I_{r_i}$,
$i = 1,\cdots,k$.

% 4.1.2

\subsubsection{Correspondance $\EE_q^1 \longleftrightarrow \Rep(\mathbf{G}_q^1)$}
\label{subsubsection:correspondanceEE1<->Rep(Gq1)}

Fixons $A = A_U$ de la forme \eqref{eqn:formestandardtriangulaireentière}, de sorte
que $A_0 := \gr A$ est de la forme \eqref{eqn:formestandarddiagonaleentiere}. La donnée
de la représentation
$\rho_A: \mathbf{G}_q^1 = \Stqu \ltimes \mathbf{G}_{q,p}^1 \rightarrow \GLnc$
associée à $A$ est équivalente à la donnée d'un couple $(\rho_A',\rho_A'')$ formé de
représentations $\rho_A': \Stqu \rightarrow \GLnc$ et
$\rho_A'': \mathbf{G}_{q,p}^1 \rightarrow \GLnc$ qui respectent les relations de
conjugaison. \\

La composante $\rho_A'': \mathbf{G}_{q,p}^1 \rightarrow \GLnc$ est la représentation
$\rho_{A_0}$ décrite plus haut. Son image est un sous-groupe algébrique du sous-groupe
$\GL_{r_1,\ldots,r_k}(\C) := \prod \GL_{r_i}(\C)$ de $\GLnc$ formé des matrices diagonales
par blocs de tailles $r_1,\ldots,r_k$. \\

L'image de $\rho_A'$ est un sous-groupe de $\G(\C)$. L'action de conjugaison est
induite par celle de $\GL_{r_1,\ldots,r_k}(\C)$, sous laquelle $\G(\C)$ est stable.
Ainsi, si $\Delta \in \Stq$ et si
$$
\rho_A'(\Delta) = \begin{pmatrix}
I_{r_1} & \ldots & \ldots & \ldots & \ldots \\
\ldots & \ldots & \ldots  & \Phi_{i,j} & \ldots \\
0      & \ldots & \ldots   & \ldots & \ldots \\
\ldots & 0 & \ldots  & \ldots & \ldots \\
0      & \ldots & 0       & \ldots & I_{r_k}   
\end{pmatrix},
$$
on peut décrire le conjugué $g^{-1} \Delta g$ de $\Delta$ par $g \in \mathbf{G}_{q,p}^1$,
identifié à $(\lambda,t,\gamma) \in \C \times \Cs \times \Hom_{gr}(\Eq,\Cs)$,
comme\footnote{Dans \cite{RS3} on décrit plutôt l'action à gauche $g \Delta g^{-1}$,
ce qui naturellement ne change rien. Cependant, il y figure une erreur d'interprétation
du calcul (bas p. 182) qui aboutit (haut p. 183) à une confusion entre l'action associée
aux exposants et celle associée aux pentes. Plusieurs formules erronnées de \cite{RS3},
qui par chance n'entraînent d'ailleurs aucune remise en cause des conclusions, seront
ici rectifiées, d'où par endroits de légères discordances; voir en particulier plus loin
le lemme \ref{lem:conditionrectifiée} de \ref{subsubsection:qdérivéesétrangères}.}:
\begin{equation}
\label{eqn:conjuguédeDelta}
\rho_A'(g^{-1} \Delta g) = \begin{pmatrix}
I_{r_1} & \ldots & \ldots & \ldots & \ldots \\
\ldots & \ldots & \ldots  & \Phi^g_{i,j} & \ldots \\
0      & \ldots & \ldots   & \ldots & \ldots \\
\ldots & 0 & \ldots  & \ldots & \ldots \\
0      & \ldots & 0       & \ldots & I_{r_k}   
\end{pmatrix},
\quad \text{où} \quad
\Phi^g_{i,j} = t^{\mu_j - \mu_i} \left(A_{i,u}^\lambda \gamma(A_{i,s})\right)^{-1}
\Phi_{i,j} \left(A_{j,u}^\lambda \gamma(A_{j,s})\right).
\end{equation}
Selon les termes du dernier alinéa de
\ref{subsubsection:correspondanceEp-Rep(Gp)}, on peut d'ailleurs
remplacer le facteur $t^{\mu_j - \mu_i}$ par $h(\mu_j - \mu_i)$. \\

Pour une description commode, il est utile de remplacer les groupes unipotents
$\Stqu$ et $\mathbf{G}_{q,p,u}^1$ par leurs algèbres de Lie 
$$
\stqu := \Lie(\Stqu) \text{~et~} \C \tau := \Lie(\mathbf{G}_{q,p,u}^1).
$$
L'algèbre de Lie $\stqu$ est pronilpotente. On vérifie que le générateur $\tau$
est envoyé par $\Lie(\rho_{A_0})$ sur
$$
\begin{pmatrix}
\log A_{1,u}  & \ldots & \ldots & \ldots & \ldots \\
\ldots & \ldots & \ldots  & 0 & \ldots \\
0      & \ldots & \ldots   & \ldots & \ldots \\
\ldots & 0 & \ldots  & \ldots & \ldots \\
0      & \ldots & 0       & \ldots & \log A_{k,u}   
\end{pmatrix}
\in \gl_{r_1,\ldots,r_k}(\C) := \bigoplus \gl_{r_i}(\C) =
\Lie\left(\GL_{r_1,\ldots,r_k}(\C)\right).
$$

Il sera également utile de regrouper\footnote{Cela revient à considérer
les $q$-logarithmes (sur lesquels agit $\mathbf{G}_{q,p,u}^1$) comme le \og degré $0$\fg\
de l'échelle des opérateurs de Stokes (sur lesquels agit $\Stqu$).} les parties
unipotentes et d'introduire 
$$
\Sttqu := \Stqu \times \mathbf{G}_{q,p,u}^1, \text{~d'où~}
\sttqu := \Lie(\Sttqu) = \stqu \oplus \C \tau \quad \text{et} \quad
\mathbf{G}_q^1 = \Sttqu \ltimes \mathbf{G}_{q,p,s}^1.
$$
Le produit \emph{direct} $\Stqu \times \mathbf{G}_{q,p,u}^1$ est justifié par le fait que
$\Stqu$ et $\mathbf{G}_{q,p,u}^1$ commutent. Dans l'algèbre de Lie pronilpotente $\sttqu$,
le crochet est caractérisé celui de $\stqu$ et par la relation:
$$
[\stqu,\C \tau] = 0.
$$
On voit alors que $\Lie(\rho_A)$ envoie $\C \tau$
dans $\gl_{r_1,\ldots,r_k}(\C)$ (et même dans la sous-algèbre des matrices
dont les blocs diagonaux sont triangulaires supérieurs stricts) et $\sttqu$
dans $\Lie(\G(\C))$ (matrices triangulaire supérieures par blocs dont les
blocs diagonaux sont nuls). L'action adjointe de $\gl_{r_1,\ldots,r_k}(\C)$
sur ces algèbres est l'action de conjugaison sur $\glnc$, sous laquelle
elles sont en effet stables. \\

Une fois ces définitions posées, on reprend l'idée de Jean-Pierre Ramis
dans sa définition du \emph{groupe de monodromie sauvage}: celui-ci est
vu comme un objet hybride $\sttqu \ltimes \mathbf{G}_{q,p,s}^1$ \og codant \fg\ le
groupe qui nous intéresse:
$$
\mathbf{G}_q^1 = \Sttqu \ltimes \mathbf{G}_{q,p,s}^1 =
\exp(\sttqu) \ltimes \mathbf{G}_{q,p,s}^1.
$$
Ce sont les représentations de $\sttqu \ltimes \mathbf{G}_{q,p,s}^1$ que l'on veut 
décrire, autrement dit, les couples formés d'une représentation de $\sttqu$
dans $\glnc$ et d'une représentation de $\mathbf{G}_{q,p,s}^1$ dans $\GLnc$, soumis
à une contrainte de compatibilité avec l'action adjointe de $\GLnc$ sur $\glnc$. \\

Tout cela ne serait \og que \fg\ de la \emph{théorie de Galois algébrique}, mais
on souhaite de plus, afin d'obtenir une \emph{correspondance de Riemann-Hilbert},
décrire explicitement une sous-algèbre de Lie $L_q^1$ de $\sttqu$ de manière à ce
que les représentations ci-dessus soient en bijection avec celles de
$L_q^1 \ltimes \mathbf{G}_{q,p,s}^1$. La nuance est que $\sttqu$ est proalgébrique
et ses représentations rationnelles; alors que $L_q^1$ est discrète, décrite par
générateurs et relations, et ses représentations sont arbitraires. On trouve en fait
\cite{RS3} que $L_q^1$ est libre graduée, ses représentations admettent donc un codage
combinatoire comme on l'espère pour une correspondance de Riemann-Hilbert. \\

L'action par conjugaison du groupe semi-simple $\mathbf{G}_{q,p,s}^1 = \Eqvee \times \Z^\vee$
sur $\sttqu$ donne lieu à une \emph{décomposition de Fourier}:
$$
\sttqu = \bigoplus_{\delta \in \Z \atop \beta \in \Eq} \sttq^{(\delta,\beta)},
$$
où l'action de $(\gamma,h) \in \Eqvee \times \Z^\vee$ sur le facteur
$\sttq^{(\delta,\beta)}$ est l'homothétie de rapport $\gamma(\beta) h(\delta)$
(rappelons que si $\beta = \overline{c}$, $c \in \Cs$, on identifie
$\gamma(\beta) = \gamma(c)$). En fait, à cause de la filtration par
les pentes et de la définition de $\sttqu$ à partir de $\stqu$, on voit que
$\sttq^{(\delta,\beta)} = 0$ pour $\delta < 0$ ainsi que pour $\delta = 0$
et $\beta \neq \overline{1}$; et que $\sttq^{(0,\overline{1})} = \C \tau$.
Les composantes\footnote{Elles étaient notées $\Delta_\alpha^{(\delta,\overline{c})}$
dans \cite{RS3}.} $\Delta_\alpha^{(\delta,\beta)} \in \sttq^{(\delta,\beta)}$ de
$\Delta_\alpha \in \sttqu$ seront précisées ci-dessous.

% 4.1.3

\subsubsection{Le groupe de Stokes et son algèbre de Lie}
\label{subsubsection:groupedeStokes}

Suivant \cite{JSStokes,RS3}, on va maintenant construire suffisamment d'éléments
explicites de $\stqu$ qui, avec $\tau$, engendreront $L_q^1$. On les obtiendra à
partir d'éléments de $\Stqu$, considéré comme automorphismes d'un certain foncteur
fibre sur $\EE_q^1$. Comme le groupe tannakien n'est pas commutatif, le choix préalable
d'un tel foncteur fibre (qui s'apparente au choix d'un point-base) est nécessaire.
On choisit donc $z_0 \in \Cs$ arbitraire et l'on identifie $\mathbf{G}_q^1$ au groupe
des automorphismes de $\omega := \hat{\omega} \circ \gr$ qui a été défini au numéro
\ref{subsubsection:prérequisGaloislocal} et dépend de manière inessentielle de $z_0$.
Des précautions liées au choix du point base seront nécessaires à partir du 
\ref{subsection:calculs explicites}, on les discutera à ce moment là. \\

Soit $A$ un objet de $\EE_q^1$, sous la forme \eqref{eqn:formestandardtriangulaireentière}
et soit $A_0 := \gr A$, qui est donc de la forme \eqref{eqn:formestandarddiagonaleentiere}.
Notant $\mathbf{G}_q^1(A)$ l'image de $\rho_A$ on a, avec les notations évidentes:
$$
\mathbf{G}_q^1(A) = \Stqu(A) \ltimes \mathbf{G}_{q,p}^1(A) =
\Sttqu(A) \ltimes \mathbf{G}_{q,p,s}^1(A) =
\sttqu(A) \ltimes \mathbf{G}_{q,p,s}^1(A),
$$
cette dernière égalité ayant le sens indirect expliqué précédemment en
\ref{subsubsection:correspondanceEE1<->Rep(Gq1)}. \\

On va maintenant construire des éléments de $\Stqu(A)$ et de $\stqu(A)$.
Pour cela, commençant par le groupe $\Stqu(A)$, on reprend les notations de
\ref{subsection:classifpentesentières}. Les $F_{\alpha,\alpha'}$,
$\alpha,\alpha' \in \Eq \setminus \Sigma_{A_0}$, 
construits en \ref{subsection:classifpentesentières} (\cf\ le théorème
\ref{thm:sommationalgébrique} et son corollaire \ref{cor:sommationalgébrique})
sont des automorphismes méromorphes de $A_0$.

\begin{prop}
Si $z_0$ n'est pas un pôle de $F_{\alpha,\alpha'}$, alors
$F_{\alpha,\alpha'}(z_0) \in \Stqu(A)$.
\end{prop}

On peut démontrer \cite{JSStokes,RS3} que ces éléments, avec leurs conjugués
par le groupe formel, engendrent topologiquement $\Stqu(A)$; mais le recours
à l'algèbre de Lie sera plus fructueux.

% 4.1.4

\subsubsection{$q$-dérivées étrangères}
\label{subsubsection:qdérivéesétrangères}

En vue de \og remplir \fg\ l'algèbre de Lie $\stqu(A)$, on fixe maintenant
$\alpha_0 \in \Eq \setminus \Sigma_{A_0}$; les résultats des calculs
qui vont suivre seront en fin de compte indépendants de ce choix.

\begin{cor}
Pour tout $\alpha \in \Eq \setminus \Sigma_{A_0}$,
$$
\log F_{\alpha_0,\alpha}(z_0) \in \stqu(A).
$$
\end{cor}

On a donc une application méromorphe
$\alpha \mapsto \log F_{\alpha_0,\alpha}(z_0)$ de $\Eq$ dans
$\stqu(A)$ (la méromorphie se vérifie par calcul direct, voir 
\ref{subsubsection:calculàdeuxpentes}). Soit maintenant
$\alpha \in \Eq$ quelconque et notons $\Delta_\alpha(A)$ le résidu en
$\alpha$ de l'application ci-dessus; nous précisons un peu
plus loin (fin de ce numéro) le mode de calcul de ce résidu. Bien entendu,
$\Delta_\alpha(A)$ s'annule si $\alpha \not\in \Sigma_{A_0}$.

\begin{cor}
$\Delta_\alpha(A) := \Res_{\alpha} \log F_{\alpha_0,\alpha}(z_0) \in \stqu(A)$.
\end{cor}

Ce résultat peut être renforcé: on vérifie que $A \leadsto \Delta_\alpha(A)$
est fonctoriel et $\otimes$-compatible (au sens des éléments \og Lie-like \fg).

\begin{prop}
Il existe $\Delta_\alpha \in \stqu$ dont les réalisations sont les $\Delta_\alpha(A)$.
\end{prop}

Notons que dans cette assertion les restrictions sur $z_0$ et les $\alpha$
n'interviennent plus. La dernière opération consiste à décomposer $\Delta_\alpha$
selon l'action de $\mathbf{G}_{q,p,s}^1 = \Eqvee \times \Cs$; nous le faisons en deux temps:
$$
\Delta_\alpha = \sum_{\delta \in \Z} \Delta_\alpha^{(\delta)} 
= \sum_{\delta \geq 1} \Delta_\alpha^{(\delta)},
\quad \text{puis:} \quad
\Delta_\alpha^{(\delta)} = \sum_{\beta \in \Eq} \Delta_\alpha^{(\delta,\beta)},
\quad \text{d'où enfin:} \quad
\Delta_\alpha =
\sum_{\delta \geq 1 \atop \beta \in \Eq} \Delta_\alpha^{(\delta,\beta)};
$$
la restriction à $\delta \geq 1$ est conséquence de la filtration par les pentes.
Dans cette décomposition, avec les conventions habituelles sur $A = A_U$:
$$
\forall \delta \geq 1 \;,\; \Delta_\alpha^{(\delta)} \in \sttq^{(\delta)}
\text{~et~} \Delta_\alpha^{(\delta)}(A) \in \gde(\C)
\quad \text{puis:} \quad
\forall \delta \geq 1 \;,\; \forall \beta \in \Eq \;,\;
\Delta_\alpha^{(\delta,\beta)} \in \sttq^{(\delta,\beta)} \subset \sttq^{(\delta)}.
$$

Pour comprendre ce que sont les composantes $\Delta_\alpha^{(\delta,\beta)}$, il est
commode de supposer que chacun des blocs $z^{\mu_i} A_i$ est lui-même décomposé selon
ses espaces caractéristiques, \ie\ que $A_i = \Diag(A_{i,1},\ldots,A_{i,\ell_i})$,
chaque $A_{i,i'}$ ayant une seule valeur propre $c_{i,i'}$. D'après les formules
\eqref{eqn:conjuguédeDelta}, on voit que le bloc $((i,i'),(j,j'))$ de
$\Delta_\alpha$ est multiplié par $t^{\mu_j - \mu_i} \gamma(c_{j,j'}/c_{i,i'})$. Donc,
si $\beta = \overline{c}$, la composante $\Delta_\alpha^{(\delta,\beta)}$ est formée
des blocs $((i,i'),(j,j'))$ tels que $\mu_j - \mu_i = \delta$ et
$c_{j,j'}/c_{i,i'} \equiv c \pmod{q^\Z}$, \ie:
$$
\pi(c_{j,j'}/c_{i,i'}) = \beta.
$$

Par ailleurs, les calculs explicites de résidus (qui seront rappelés en
\ref{subsubsection:calculàdeuxpentes}) montrent que les dénominateurs qui sont
à l'origine des pôles de la fonction méromorphe
$\overline{c} \mapsto \log F_{\overline{c_0},\overline{c}}(z_0)$
sont de la forme $q^m c_{i,i'} c^{\mu_i} - c_{j,j'} c^{\mu_j}$, $m \in \Z$. Les pôles
$\alpha = \overline{c}$ éventuels vérifient donc nécessairement
$c_{i,i'} c^{\mu_i} \equiv c_{j,j'} c^{\mu_j} \pmod{q^\Z}$, autrement dit:
$$
\pi(c_{i,i'}/c_{j,j'}) = \alpha^\delta.
$$

On en déduit\footnote{C'est essentiellement dans cette condition que l'on a dû
rectifier les formules erronées de \cite{RS3}, en changeant $\beta = \overline{c}$
en $\beta^{-1}$.}:

\begin{lem}
\label{lem:conditionrectifiée}
Si $\alpha^\delta \neq \beta^{-1}$, alors $\Delta_\alpha^{(\delta,\beta)} = 0$.
\end{lem}

Pour chaque couple $(\delta,\beta) \in \N^* \times \Eq$, on a donc produit $\delta^2$
éléments non triviaux $\Delta_\alpha^{(\delta,\beta)}$ de $\sttq^{(\delta,\beta)}$: ceux
tels que $\alpha^\delta = \beta^{-1}$

\begin{thm}
\label{thm:thprincipalpentesentières}
(i) Les $\Delta_\alpha^{(\delta,\beta)}$ telles que $\alpha^\delta = \beta^{-1}$
engendrent avec $\tau$ une sous-algèbre $L_q^1$ de $\sttqu$, graduée par le monoïde
$(\N^* \times \Eq) \cup \{0\}$ et telle que la catégorie des représentations
de $L_q^1 \ltimes \mathbf{G}_{q,p,s}^1$ soit équivalente à $\EE_q^1$. \\
(ii) Pour tout $(\delta,\beta) \in \N^* \times \Eq$, on peut choisir
$\delta$ parmi les $\delta^2$ racines $\delta$\iemes\ de $\beta^{-1}$ dans
$\Eq$, soient $\alpha_i \in \Eq$, $i = 1,\ldots,\delta$, de telle sorte que
$\tau$ et les $\Delta_i^{(\delta,\beta)} := \Delta_{\alpha_i}^{(\delta,\beta)}$,
$\beta \in \Eq$, $i = 1,\ldots,\delta$, forment une base de $L_q^1$.
\end{thm}

C'est essentiellement cette réduction aux représentations d'une algèbre de Lie
graduée libre qui a permis la résolution du problème inverse dans \cite{RS3}.

\begin{rem}
Nous ne disposions pas dans ce précédent travail d'un choix canonique convenable
des $\Delta_i^{(\delta,\beta)}$. Nous obtiendrons cependant ici (voir
\ref{subsubsection:basecanonique}) un tel choix qui se comporte de plus agréablement
sous l'action du groupe formel; mais malheureusement seulement pour des valeurs
génériques de $q$.
\end{rem}

\paragraph{Calcul des résidus.}

Soit $f(c)$ une fonction analytique dans un voisinage épointé de $c_0 \in \C$.
Nous noterons $\res_{c = c_0} f(c)$, ou, s'il n'y a pas de risque de confusion,
$\res_{c_0} f$ le résidu en $c_0$ de $f$, c'est à dire, en termes plus intrinsèques,
de la forme différentielle $df$. \\

Soit maintenant $f$ une fonction méromorphe $q$-invariante sur $\C^*$, que l'on
peut donc identifier à une fonction méromorphe sur $\Eq$.
Notons $\e(x) := e^{2 \ii \pi x}$ et $\Lambda_\tau := \Z + \Z \tau$ (on rappelle
que $q = \e(\tau)$). Du diagramme:
$$
\xymatrix{
\C \ar@<0ex>[rr]^\e \ar@<0ex>[d] & & \C^* \ar@<0ex>[d] \\
\C/\Lambda_\tau \ar@<0ex>[rr]^\sim & & \C^*/q^\Z
}
$$
on déduit que l'uniformisante canonique sur $\Eq$ est
$dx = \dfrac{1}{2 \ii \pi} \, \dfrac{dc}{c} \cdot$ Le facteur $\dfrac{1}{2 \ii \pi}$
ne semble (pour le moment) pas avoir d'importance dans notre affaire et nous
poserons donc, si $\alpha = \overline{c}$ et $\alpha_0 = \overline{c_0}$
$$
\Res_{\alpha = \alpha_0} f(z) = \dfrac{1}{c_0} \, \res_{c=c_0} f(c).
$$
On peut vérifier par calcul direct qu'en effet, $f$ étant $q$-invariante, cette
expression ne dépend que de $\alpha_0 = \overline{c_0}$ (ce qui ne serait pas
le cas avec $\res$). S'il n'y a pas de risque de confusion, on notera plus
simplement $\Res_{\alpha_0} f$.

% 4.2

\subsection{Calculs explicites de $q$-dérivées étrangères}
\label{subsection:calculs explicites}

Le calcul des $\Delta_\alpha(A)$, $\Delta_\alpha^{(\delta)}(A)$ et
$\Delta_\alpha^{(\delta,\beta)}(A)$ par application directe de la définition
comporte plusieurs étapes non linéaires: équations définissant $F_{\overline{c}}$;
produits $F_{\overline{c}}^{-1} F_{\overline{d}}$; prises de logarithme. Pourtant, nous
allons voir que \emph{le calcul du résultat peut être rendu totalement linéaire}.
En fait, il se ramène au cas des matrices à deux pentes extraites de $A$. \\

Comme indiqué au \ref{subsubsection:groupedeStokes}, le choix du point-base rend
indispensable certaines précautions. À partir du \ref{subsubsection:calculàdeuxpentes},
il apparaît des expressions qui n'ont de sens que si des facteurs de la forme
$\thqc(z_0)$ sont non nuls. Ceci ne peut être garanti que si les exposants des
matrices manipulées n'appartiennent pas à $[-z_0;q]$. Comme on travaille par
principe dans une catégorie tannakienne, les exposants doivent pouvoir être
multipliés. On est donc conduit à poser les conventions suivantes:
\begin{enumerate}
\item Pour $z_0$ donné, on choisit un sous-groupe $G$ de $\Eq$ qui ne contient pas
$\overline{-z_0}$.
\item Les formules explicites énoncées sont données pour la sous-catégorie de $\EE_q(G)$
formée des objets dont tous les exposants (\ie\ leurs classes) sont dans $G$.
\item Les constructions galoisiennes peuvent être recollées quand on fait varier $z_0$
et $G$ à l'aide des identifications mentionnées en \ref{subsubsection:propabeltens} et
en \ref{subsubsection:prérequisGaloislocal}.
\item La catégorie $\EE_q$ est la limite inductive de toutes les catégories $\EE_q(G)$.
\end{enumerate}
Nous choisissons, comme nous l'avons fait dans \cite{RS3}, de ne pas mentionner
explicitement ces restrictions\footnote{C'est d'ailleurs l'usage quand il s'agit de
correspondance de Riemann-Hilbert pour les équations différentielles: une fois
choisi un point base pour le groupe fondamental, il n'est pas vrai que toute équation
à laquelle s'applique la théorie soit définie \emph{en ce point}. Le petit exorcisme
qui consiste à changer de point base selon les nécessités est implicite.}, afin de ne
pas alourdir les notations. \\

Cependant, le théorème \ref{thm:conjuguésmeilleurebasecanonique} à la fin de
\ref{subsubsection:RestrictionàE^retramification} et le théorème
\ref{thm:groupefondamentalsauvage} au \ref{subsection:structgroufonsau} sont
des énoncés qui ne comportent pas des expressions dépendant du point-base et
ils sont donc valides sans restriction pour la catégorie $\EE_q^r$.

% 4.2.1

\subsubsection{Linéarisation}
\label{subsubsection:linéarisation}

Dans ce qui suit, on fixe $A_0$ et on prend $A = A_U$. Pour tout $\delta \in \N^*$,
on notera $A^{\leq \delta}$ la matrice obtenue en remplaçant par $0$ tous les $U_{i,j}$
tels que $\mu_j - \mu_i > \delta$; de même, on notera $A^{(\delta)}$ la matrice
obtenue en remplaçant par $0$ tous les $U_{i,j}$ tels que $\mu_j - \mu_i \neq \delta$
(\og matrice à une couche \fg). Avec ces notations,
$A_U \equivd A_V \Leftrightarrow A_U^{\leq \delta} = A_V^{\leq \delta}$
(la relation $\equivd$ a été définie par l'équation \eqref{eqn:defequivd}, paragraphe
\og Filtration $q$-Gevrey \fg\ de \ref{subsubsection:structaff}). \\

\includegraphics[width=12cm]{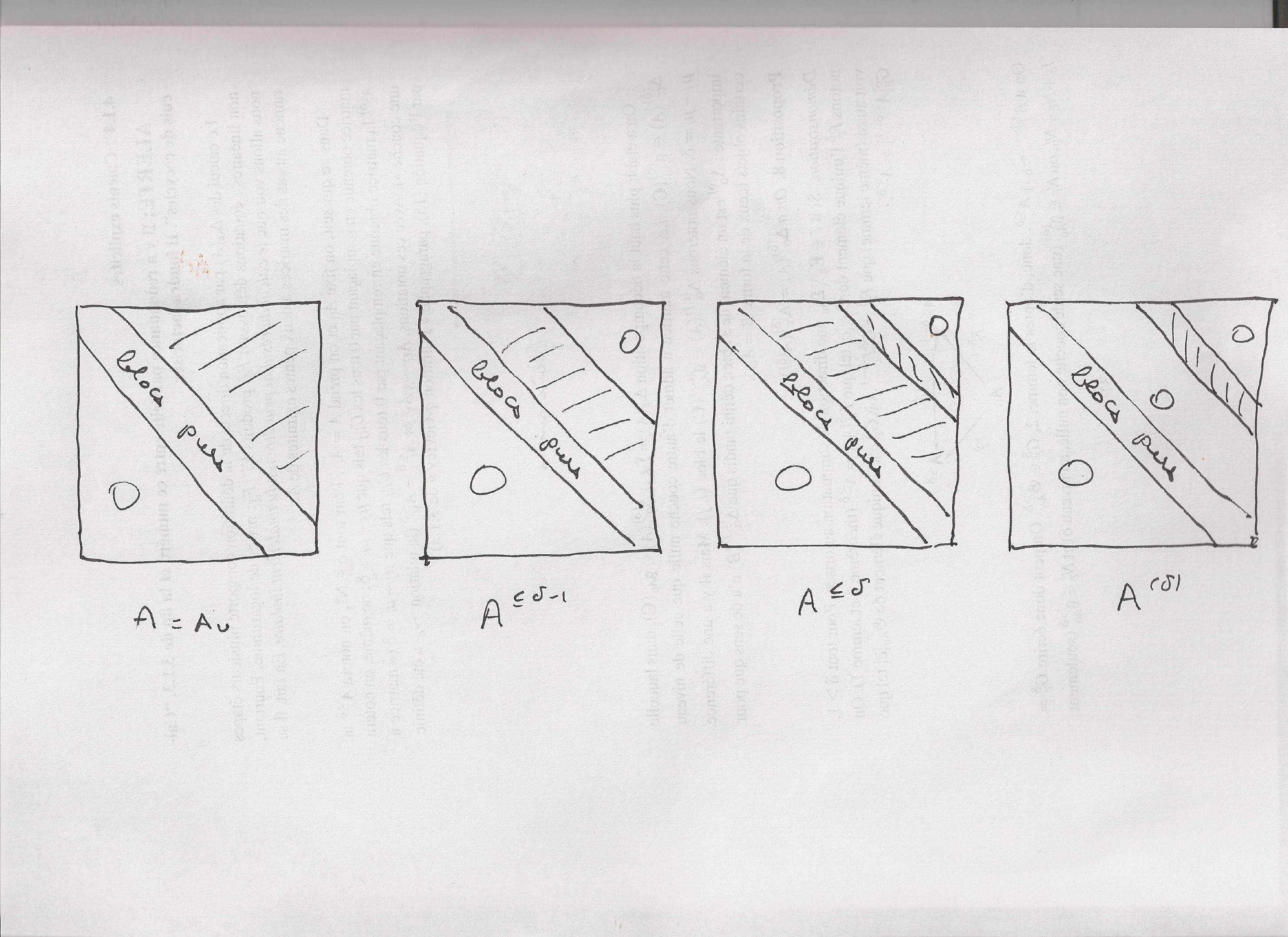} 

\bigskip

On a établi plus haut la décomposition
$\Delta_\alpha(A) = \sum\limits_{\delta \geq 1} \Delta_\alpha^{(\delta)}(A) \in \g(\C)$,
dans laquelle $\Delta_\alpha^{(\delta)}(A) \in \gde(\C)$, \ie\ cette matrice
n'admet d'autre couche nulle que celle de niveau $\mu_j - \mu_i = \delta$.
Nous noterons $\Delta_\alpha^{(i,j)}(A) \in \mathfrak{g}^{(i,j)}_{A_0}(\C)$ le
bloc $(i,j)$. Mais il y a une différence importante: $\Delta_\alpha^{(\delta)}$
est un élément de $\st$ bien défini tandis que $\Delta_\alpha^{(i,j)}(B)$ n'a
de sens que pour certains objets (ceux de la forme $B = A_U$). \\

\includegraphics[width=12cm]{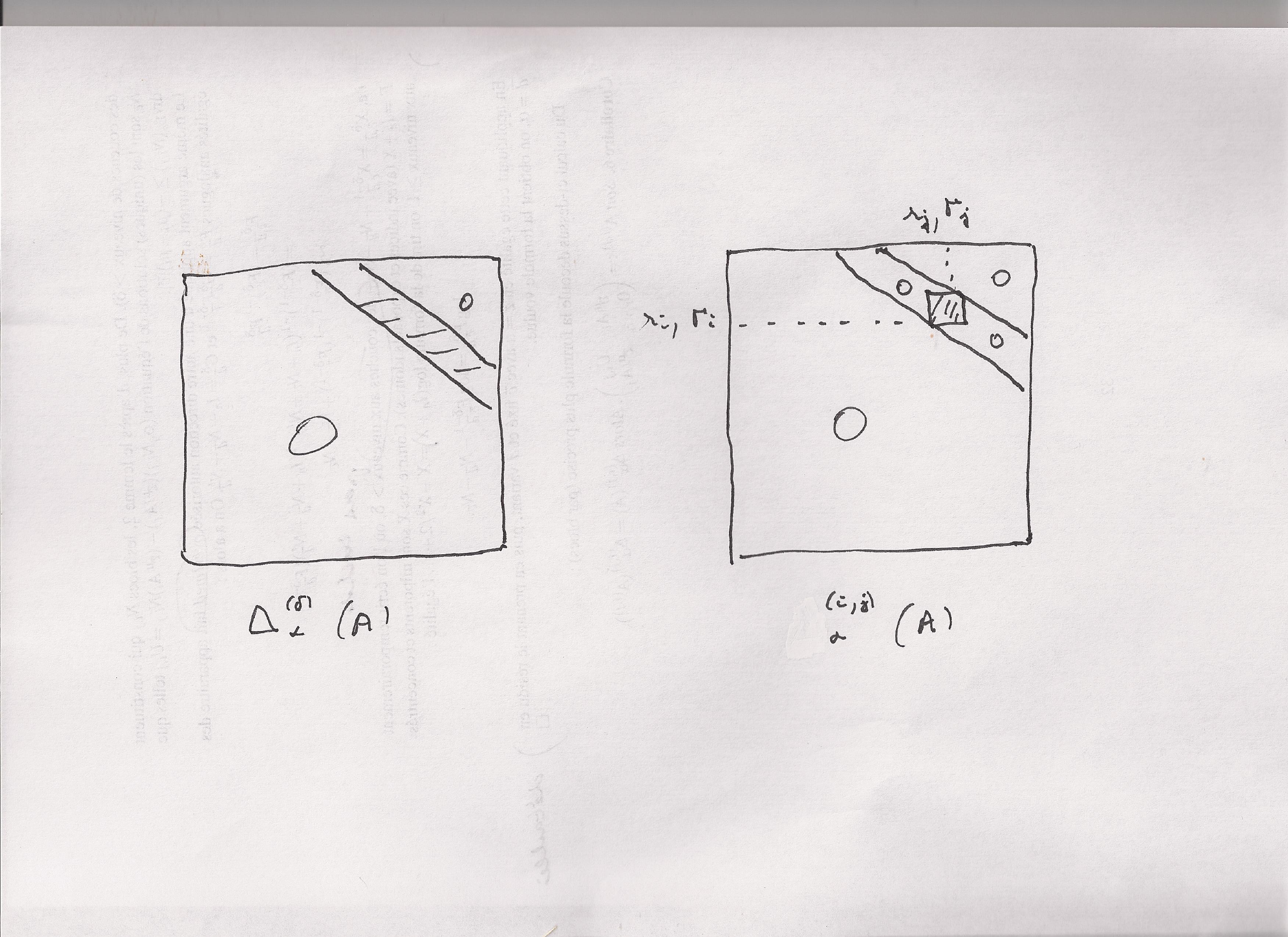}

\begin{prop}
On a $\Delta_\alpha^{(\delta)}(A) = \Delta_\alpha^{(\delta)}(A^{(\delta)})$.
\end{prop}
\begin{proof}
Soit $\overline{c} \in \Eq \setminus \Sigma_{A_0}$ une direction de sommation
autorisée. Pour tout $\delta \geq 1$, notons $F^\delta_{\overline{c}}$ l'unique
élément de $\G[\overline{c}]$ tel que $F^\delta_{\overline{c}}[A_0] = A^{\leq \delta}$
(théorème \ref{thm:sommationalgébrique} et lemme \ref{lem:groupepolaire}).
On voit aussi (même lemme) que
$F^\delta_{\overline{c}} = G^\delta_{\overline{c}} F^{\delta-1}_{\overline{c}}$, où
$G^\delta_{\overline{c}}$ est l'unique élément de $\G[\overline{c}]$ tel que
$G^\delta_{\overline{c}}[A^{\leq \delta-1}] = A^{\leq \delta}$:
$$
\xymatrix{
A^{\leq \delta-1} \ar@<0ex>[rr]^{G^\delta_{\overline{c}}} & & A^{\leq \delta} \\
& A_0 \ar@<0ex>[lu]^{F^{\delta-1}_{\overline{c}}} \ar@<0ex>[ru]_{F^\delta_{\overline{c}}} &
}
$$
On a $A^{\leq \delta-1} \equiv_{\delta-1} A^{\leq \delta}$, d'où, d'après le lemme
\ref{lem:F_{U,V}etniveaudelta}, $G^\delta_{\overline{c}} \in \Gd$. On peut donc
écrire:
$$
G^\delta_{\overline{c}} = I_n + N_{\overline{c}} + N'_{\overline{c}},
$$
où $N_{\overline{c}} \in \gde$ (une seule couche non nulle, au niveau $\delta$)
et $N'_{\overline{c}} \in \gds$ (seulement des couches de niveaux $> \delta$).
De plus, d'après le lemme \ref{lem:F_{U,V}etniveaudelta}, les blocs $N_{i,j}$
qui constituent $N_{\overline{c}}$ sont les (uniques) solutions de l'équation
$$
(\sq N_{i,j}) (z^{\mu_j} A_j) - (z^{\mu_i} A_i) N_{i,j} = U_{i,j}
$$
telles que $\div_\Eq(N_{i,j}) \geq - (\mu_j - \mu_i) [\overline{c}]$. \\
Le même argument appliqué à une autre direction autorisée $\overline{d}$
fait apparaitre des égalités analogues
$F^\delta_{\overline{d}} = G^\delta_{\overline{d}} F^{\delta-1}_{\overline{d}}$ et
$G^\delta_{\overline{d}} = I_n + N_{\overline{d}} + N'_{\overline{d}}$. \\
La fin du calcul utilise les propriétés de la filtration, essentiellement
le fait que:
$$
\gd(\C) \g^{\geq \delta'}(\C) \subset \g^{\geq \delta + \delta'}(\C).
$$
On a donc :
\begin{align*}
  F^\delta_{\overline{c},\overline{d}} &=
  (F^\delta_{\overline{c}})^{-1} F^\delta_{\overline{d}} 
  = (F^{\delta-1}_{\overline{c}})^{-1} (I_n + N_{\overline{c}} + N'_{\overline{c}})^{-1}
  (I_n + N_{\overline{d}} + N'_{\overline{d}}) (F^{\delta-1}_{\overline{d}}) \\
  &= (F^{\delta-1}_{\overline{c}})^{-1} (F^{\delta-1}_{\overline{d}}) +
  N_{\overline{d}} - N_{\overline{c}} + N', \quad N' \in \gds(\C).
\end{align*}
Écrivant temporairement $F^\delta_{\overline{c},\overline{d}} = I_n + X$
et $F^{\delta-1}_{\overline{c},\overline{d}} = I_n + Y$,
on a donc $X,Y \in \g^{\geq 1}(\C)$ et:
$$
X = Y + N_{\overline{d}} - N_{\overline{c}} + N'.
$$
Comme $X$ et $Y$ sont nilpotents et concentrés aux niveaux $\geq 1$, on tire
des formules $\log(I_n + X) = X - X^2/2 + \cdots$ et
$\log(I_n + Y) = Y - Y^2/2 + \cdots$ l'égalité:
$$
\log F^\delta_{\overline{c},\overline{d}} = \log F^{\delta-1}_{\overline{c},\overline{d}} +
N_{\overline{d}} - N_{\overline{c}} + N'', \quad N'' \in \gds(\C).
$$
En appliquant cette égalité en $z = z_0$ avec $\overline{c}$ fixé et
$\overline{d}$ variant, puis en prenant le résidu en $\overline{d} = \alpha$,
on obtient la formule:
\begin{align*}
\Delta_\alpha(A^{\leq \delta})
  &= \Delta_\alpha(A^{\leq \delta-1}) + \Res_{\overline{d} = \alpha} N_{\overline{d}}(a) \\
  &= \Delta_\alpha(A^{\leq \delta-1}) + \Delta_\alpha^{(\delta)}(A^{(\delta)})
\end{align*}
d'après la caractérisation ci-dessus de $N$. Le calcul successif des
$\Delta_\alpha(A^{\leq \delta})$ pour $\delta = 1,\ldots$ consiste donc à ajouter
couche par couche les $\Delta_\alpha^{(\delta)}(A^{(\delta)})$, d'où la conclusion.
\end{proof}

Du calcul ci-dessus découle la formule plus précise (par blocs):

\begin{cor}
Soit $A^{(i,j)} :=
\begin{pmatrix} z^{\mu_i} A_i & U_{i,j} \\
 0_{r_j \times r_i} & z^{\mu_j} A_j \end{pmatrix}$. Alors
$\Delta_\alpha^{(i,j)}(A) = \Delta_\alpha^{(i,j)}(A^{(i,j)})$.
\end{cor}

\paragraph{Un autre outil de dévissage.}

Nous mentionnons ici, pour usage ultérieur, un autre lemme utile.

\begin{lem}
\label{lem:dévissageparblocs}
Soit $A := \begin{pmatrix} z^{\mu_1} A_1 & U \\  0 & z^{\mu_2} A_2 \end{pmatrix}$,
et supposons que $A_1,A_2$ admettent des décompositions diagonales par blocs
$A_i = \Diag(A_{i,1},A_{i,2})$, $i = 1,2$.
Soit $U := \begin{pmatrix} U_{1,1} & U_{1,2} \\ U_{2,1} & U_{2,2} \end{pmatrix}$
la décomposition par blocs correspondante. Alors:
$$
\Delta_\alpha(A) = \begin{pmatrix} 0 & N \\ 0 & 0 \end{pmatrix},
\quad \text{où} \quad
N := \begin{pmatrix} N_{1,1} & N_{1,2} \\ N_{2,1} & N_{2,2} \end{pmatrix},
\quad \text{avec} \quad
\begin{pmatrix} 0 & N_{i,j} \\ 0 & 0 \end{pmatrix} :=
\Delta_\alpha\begin{pmatrix} z^{\mu_1} A_{1,i} & U_{i,j} \\  0 & z^{\mu_2} A_{2,j} \end{pmatrix}.
$$
\end{lem}
\begin{proof}
L'isomorphisme méromorphe $F: A_0 \rightarrow A$ est de la forme
$F = \begin{pmatrix} I_{r_1} & f \\ 0 & I_{r_2} \end{pmatrix}$, où
$f = \begin{pmatrix} f_{1,1} & f_{1,2} \\ f_{2,1} & f_{2,2} \end{pmatrix}$,
chaque $f_{i,j}$ étant calculé comme le bloc surdiagonal de l'isomorphisme
méromorphe $\begin{pmatrix} I_{r_{1,i}} & f \\ 0 & I_{r_{2,j}} \end{pmatrix}$
de $\begin{pmatrix} z^{\mu_1} A_{1,i} & 0 \\  0 & z^{\mu_2} A_{2,j} \end{pmatrix}$
dans $\begin{pmatrix} z^{\mu_1} A_{1,i} & U_{i,j} \\  0 & z^{\mu_2} A_{2,j} \end{pmatrix}$:
en effet, l'équation fonctionnelle $(\sq f) (z^{\mu_2} A_2) - (z^{\mu_1} A_1) f = U$
équivaut au système des quatre équations
$(\sq f_{i,j}) (z^{\mu_2} A_{2,j}) - (z^{\mu_1} A_{1,i}) f_{i,j} = U_{i,j}$. On a donc quatre
calculs indépendants, suivis de l'évaluation en $z_0$ et de la prise de résidu.
\end{proof}

\paragraph{Rappels de théorie spectrale.}

% références:
% Bourbaki, Théories spectrales, I § 5
% Reed & Simon, Functional Analysis, vol I, § VII et VIII
% Rudin, Analyse fonctionnelle, ch. 10
% Schechter, Principles of functional Analysis, § 6.3
% Quelque chose de JL Lions

Soient $E$ un $\C$-espace vectoriel de dimension finie et $\phi \in \Lin(E)$.
Notons $E = \bigoplus\limits_{\lambda \in \Sp\ \phi} E_\lambda$ la décomposition en
espaces caractéristiques et $(\Pi_\lambda)_{\lambda \in \Sp\ \phi}$ la famille des
projecteurs associés. Par convention $E_\lambda$ et $\Pi_\lambda$ sont triviaux
si $\lambda \not\in \Sp\ \phi$. \\

Soit $f(z) := (z \Id_E - \phi)^{-1}$ la résolvante: $f$ est méromorphe sur $\C$
à valeurs dans $\Lin(E)$, ses pôles sont les éléments de $\Sp\ \phi$ et:
$$
\res_{z = \lambda} f(z) = \Pi_\lambda.
$$

Soient maintenant $P \in \GL_r(\C)$ et $Q \in \GL_s(\C)$, que l'on suppose diagonaux
par blocs associés à leurs espaces caractéristiques: $P = \Diag(P_1,\ldots,P_k)$,
$P_i \in \GL_{r_i}(\C)$ (donc $r = r_1 + \cdots + r_k$), $\Sp\ P_i = \{\lambda_i\}$,
les $\lambda_i$ deux à deux distincts; et $Q = \Diag(Q_1,\ldots,Q_l)$,
$Q_j \in \GL_{s_j}(\C)$ (donc $s = s_1 + \cdots + s_l$), $\Sp\ Q_j = \{\mu_j\}$,
les $\mu_j$ deux à deux distincts. \\

Notons $\Phi_{P,Q}$ l'endomorphisme de $E := \Mat_{r,s}(\C)$ défini par
$\Phi_{P,Q}(X) = P X Q^{-1}$. On décompose $X$ en blocs $X_{i,j} \in \Mat_{r_i,s_j}(\C)$
et de même $E = \bigoplus\limits_{1 \leq i \leq k \atop 1 \leq j \leq l} E_{i,j}$. Alors
le spectre de $\Phi_{P,Q}$ est l'ensemble des $\lambda_i/\mu_j$ et les espaces
caractéristiques sont les
$$
E_\lambda := \bigoplus_{\lambda_i/\mu_j = \lambda} E_{i,j}.
$$
Si l'on convient d'identifier $\Mat_{r_i,s_j}(\C)$ au sous-espace $E_{i,j}$ de $E$,
alors la matrice $\Pi_\lambda(X)$, qui s'obtient en remplaçant par $0$ tous les
blocs $X_{i,j}$ tels que $\lambda_i/\mu_j \neq \lambda$, peut être identifiée à
$\sum\limits_{\lambda_i/\mu_j = \lambda} X_{i,j}$.

% 4.2.2

\subsubsection{Calcul à deux pentes}
\label{subsubsection:calculàdeuxpentes}

On prendra:
$$
A = \begin{pmatrix} z^{\mu_1} A_1 & z^{\mu_1} U A_2 \\ 0 & z^{\mu_2} A_2 \end{pmatrix},
$$
avec $A_i \in \GL_{r_i}(\C)$, $i = 1,2$, $\mu_1,\mu_2 \in \Z$, $\mu_1 < \mu_2$ et
$U \in \Mat_{r_1,r_2}(\Ka)$. La forme particulière du terme non diagonal
$z^{\mu_1} U A_2$ est destinée à simplifier les formules. On pose
$\delta := \mu_2 - \mu_1 \in \N^*$ et bien entendu
$A_0 := \begin{pmatrix} z^{\mu_1} A_1 & 0 \\ 0 & z^{\mu_2} A_2 \end{pmatrix}$. \\

Rappelons, de la section \ref{subsection:Notations} de notations, la fonction
theta d'usage courant $\thq$ et les fonctions $\thqc$. On introduit:
$$
V(c,z) := \theta^\delta_{q,c} U(z),
\quad \text{que l'on développe en série de Laurent:} \quad
V(c,z) = \sum_{m \in \Z} V_m(c) z^m.
$$
La formule suivante se déduit directement de l'égalité
$V(qc,z) = (qc/z)^\delta V(c,z)$ qui est elle-même conséquence de l'égalité
$\theta_{q,qc}(z) = (qc/z) \thqc(z)$:
\begin{equation}
\label{eqn:delta-périodicitédeVmbis}
V_{m-\delta}(qc) = (qc)^\delta V_m(c).
\end{equation}

\begin{prop}
\label{prop:casgénériqueàdeuxpentes}
Soit $c \in \C^*$ tel que $\overline{c} \in \Eq \setminus \Sigma_{A_0}$. Alors
l'unique isomorphisme méromorphe $F_{\overline{c}}: A_0 \rightarrow A$ tel que
$F_{\overline{c}} \in \G[\overline{c}]$ (théorème \ref{thm:sommationalgébrique};
voir le lemme \ref{lem:groupepolaire} pour la notation) est de la forme:
$$
F_{\overline{c}} = \begin{pmatrix} I_{r_1} & f_{\overline{c}} \\ 0 & I_{r_2} \end{pmatrix},
\quad \text{~où:} \quad
f_{\overline{c}} := \dfrac{1}{\theta^\delta_{q,c}} \;
\sum_{m \in \Z} (q^m c^\delta \Id - \Phi_{A_1,A_2})^{-1}(V_m).
$$
\end{prop}

La condition \og $\overline{c}$ est une direction de sommation autorisée \fg\
est précisément celle qui garantit l'inversibilité des endomorphismes
$q^m c^\delta \Id - \Phi_{A_1,A_2}$. Le facteur de tête $\dfrac{1}{\theta^\delta_{q,c}}$
est celui qui produit les pôles sur $[-c;q]$. À l'aide de la relation
\eqref{eqn:delta-périodicitédeVmbis}, on peut vérifier (bien que cela ne soit
en principe pas nécessaire en vertu du théorème  \ref{thm:sommationalgébrique})
que l'expression ci-dessus ne change pas quand on remplace $c$ par $qc$.

\begin{proof}
On a un diagramme commutatif de transformations de jauge méromorphes sur $\C^*$:
$$
\xymatrix{
  {\begin{pmatrix} c^{\mu_1} A_1 & 0 \\ 0 & c^{\mu_2} A_2 \end{pmatrix}}
  \ar@<0ex>[rr]^\Theta \ar@<0ex>[dd]_{G_{\overline{c}}} & &
  {\begin{pmatrix} z^{\mu_1} A_1 & 0 \\ 0 & z^{\mu_2} A_2 \end{pmatrix}}
  \ar@<0ex>[dd]^{F_{\overline{c}}} \\
  & & \\
  {\begin{pmatrix} c^{\mu_1} A_1 & c^{\mu_1} V A_2 \\ 0 & c^{\mu_2} A_2 \end{pmatrix}}
  \ar@<0ex>[rr]_\Theta & &
  {\begin{pmatrix} z^{\mu_1} A_1 & z^{\mu_1} U A_2 \\ 0 & z^{\mu_2} A_2 \end{pmatrix}}
}
$$
Les flèches horizontales proviennent de la même matrice
$\Theta :=
\begin{pmatrix} \theta^{\mu_1}_{q,c} \, I_{r_1} & 0 \\
0 & \theta^{\mu_2}_{q,c} \, I_{r_2} \end{pmatrix}$.
Pour que la transformation horizontale basse soit justifiée, il faut et il suffit
que
$z^{\mu_1} U A_2 = \dfrac{\sq(\theta^{\mu_1}_{q,c})}{\theta^{\mu_2}_{q,c}} c^{\mu_1} V A_2$,
ce qui revient à la définition de $V$ donnée plus haut. \\

La commutativité signifie que $G_{\overline{c}} = \Theta^{-1} F_{\overline{c}} \Theta$,
\ie\ 
$G_{\overline{c}} = \begin{pmatrix} I_{r_1} & g_{\overline{c}} \\ 0 & I_{r_2} \end{pmatrix}$,
où $g_{\overline{c}} = \theta^\delta_{q,c} f_{\overline{c}}$. La flèche verticale droite
est celle qui définit $F_{\overline{c}}$; elle sera justifiée si la flèche verticale
gauche l'est. \\

Cette dernière équivaut à la relation:
$$
(\sq G_{\overline{c}})
\begin{pmatrix} c^{\mu_1} A_1 & 0 \\ 0 & c^{\mu_2} A_2 \end{pmatrix} =
\begin{pmatrix} c^{\mu_1} A_1 & c^{\mu_1} V A_2 \\ 0 & c^{\mu_2} A_2 \end{pmatrix}
G_{\overline{c}},
$$
qui elle-même se traduit par:
$$
(\sq g_{\overline{c}}) (c^{\mu_2} A_2) = (c^{\mu_1} A_1) g_{\overline{c}} + c^{\mu_1} V A_2,
$$
soit encore:
$$
c^\delta \sq g_{\overline{c}} - \Phi_{A_1,A_2}(g_{\overline{c}}) = V.
$$
La condition de polarité sur $f_{\overline{c}}$ équivaut à l'absence de pôles
de $g_{\overline{c}}$ sur $\C^*$. On développe donc cette fonction \emph{a priori}
inconnue en série de Laurent: $g_{\overline{c}} = \sum\limits_{m \in \Z} g_m z^m$
et l'on aboutit au système:
$$
\forall m \in \Z \;,\;
c^\delta q^m g_m - \Phi_{A_1,A_2}(g_m) = V_m.
$$
L'assertion voulue s'ensuit.
\end{proof}

\paragraph{Application.}

On prend $A_1$ et $A_2$ diagonales par blocs correspondant à leurs espaces
caractéristiques: $A_1 = \Diag(A_{1,1},\ldots,A_{1,k}) \in \GL_r(\C)$,
$A_{1,i} \in \GL_{r_i}(\C)$, $\Sp\ A_{1,i} = \{\lambda_{1,i}\}$, les $\lambda_{1,i}$
deux à deux distincts pour $i = 1,\ldots,k$; et similairement
$A_2 = \Diag(A_{2,1},\ldots,A_{2,l}) \in \GL_s(\C)$,
$A_{2,j} \in \GL_{s_j}(\C)$, $\Sp\ A_{2,j} = \{\lambda_{2,j}\}$, les $\lambda_{2,j}$
deux à deux distincts pour $j = 1,\ldots,l$. On décompose toute matrice
$X \in \Mat_{r,s}(\C)$ en blocs $X_{i,j} \in \Mat_{r_i,s_j}(\C)$ pour $i = 1,\ldots,k$,
$j = 1,\ldots,l$. Pour tout $d \in \C^*$, selon les notations du dernier paragraphe
de \ref{subsubsection:linéarisation} (\og Rappels de théorie spectrale \fg):
$$
\Pi_d(X) = \sum_{\lambda_{1,i}/\lambda_{2,j} = d} X_{i,j} \quad\text{(qui peut être nul).}
$$
Pour tout $d \in \C^*$, on code comme suit\footnote{Rappelons que d'après les
formules rectifiées de \ref{subsubsection:qdérivéesétrangères} (voir en particulier
le lemme \ref{lem:conditionrectifiée}), c'est $\beta := \pi(\lambda_{2,j}/\lambda_{1,i})$
qui apparaitra dans la $q$-dérivée étrangère $\Delta_\alpha^{(\delta,\beta)}$ que nous
sommes en train de calculer.} les racines $\delta$\iemes\ de
$\overline{d} \in \Eq$: pour tout $m = 0,\ldots,\delta-1$, on note $c_{l,m}$,
$l = 0,\ldots,\delta-1$, les racines $\delta$\iemes\ de $q^{-m} d$ dans $\C^*$;
les racines $\delta$\iemes\ de $\overline{d}$ sont donc les $\overline{c_{l,m}}$. \\

Pour $c$ au voisinage d'un $c_{l,m}$ donné, posant $\phi(c) := f_{\overline{c}}(z_0)$,
on écrit (voir les notations du paragraphe \og Calcul des résidus \fg\ à la fin
du numéro \ref{subsubsection:groupedeStokes}):
\begin{align*}
\phi(c)
&= \dfrac{z_0^l}{\thq(z_0/c)^\delta} \times
\dfrac{1}{c^\delta q^l - d} \times \Pi_d(V_l(c))  + \cdots \\
&= \dfrac{(z_0/q)^l}{\thq(z_0/c_{l,m})^\delta} \times
\dfrac{1}{c^\delta - c_{l,m}^\delta} \times \Pi_d(V_l(c_{l,m}))  + \cdots, \\
&= \dfrac{(z_0/q)^l}{\thq(z_0/c_{l,m})^\delta} \times
\dfrac{1}{\delta c_{l,m}^{\delta-1}} \times \dfrac{1}{c - c_{l,m}} \times
\Pi_d(V_l(c_{l,m}))  + \cdots, \\
&= \dfrac{(z_0/q)^l}{\thq(z_0/c_{l,m})^\delta} \times c_{l,m} \times
\dfrac{1}{\delta d} \times \dfrac{1}{c - c_{l,m}} \times
\Pi_d(V_l(c_{l,m}))  + \cdots,
\end{align*}
d'où l'on tire, en posant $\alpha_{l,m} := \overline{c_{l,m}}$:
$$
\Res_{q,\alpha_{l,m}} \phi =
\dfrac{(z_0/q)^l}{\delta d \thq(z_0/c_{l,m})^\delta} \, \Pi_d(V_l(c_{l,m})).
$$

\begin{cor}
Les \og dérivées étrangères \fg\ non nulles $\Delta_\alpha^{(\delta,\beta)}(A)$
associées à $A$ sont précisément les:
$$
\Delta_{\alpha_{l,m}}^{(\delta,\overline{d^{-1}})}(A) =
\begin{pmatrix} 0_r & \Res_{q,\alpha_{l,m}} \phi \\ 0_{s,r} & 0_s \end{pmatrix}
$$
calculées ci-dessus.
\end{cor}

\paragraph{Un exemple.}

Nous détaillons le cas $r = 2$ car les équations d'ordre $2$ sont celles
qui servent le plus dans les applications (fonctions $q$-spéciales). De
plus, génériquement, tous les calculs peuvent s'y ramener. On prendra ici:
$$
A := \begin{pmatrix} a z^k & u \\ 0 & b z^l \end{pmatrix}, u \in \Rw.
$$
On notera $\delta := l - k$, $d := a/b$ et $\beta := \overline{d^{-1}}$. On est
donc en train de calculer des $q$-dérivées étrangères $\Delta_\alpha^{(\delta,\beta)}$.

\begin{prop}
\label{prop:lecas2x2}
(i) L'ensemble $\Sigma_{A_0}$ est formé des $\delta^2$ racines $\delta$\iemes\ de
$\overline{d}$ dans $\Eq$. Plus précisément, si pour tout $m = 0,\ldots,\delta-1$,
les $c_{l,m}$, $l = 0,\ldots,\delta-1$, sont les $\delta$ racines $\delta$\iemes\ de
$q^{-m} d$ dans $\C^*$, on a:
$$
\Sigma_{A_0} = \{\overline{c_{l,m}} \tq l,m = 0,\ldots,\delta-1\}.
$$
(ii) Soient $f_m(c) := \dfrac{z_0^m v_m(c)}{b \thq(z_0/c)^\delta}$, où l'on a développé
$u z^{-k} \theta_{q,c}^\delta = \sum\limits_{m \in \Z} v_m(c) z^m$, et notons
$\alpha_{l,m} := \overline{c_{l,m}}$ et $\beta := \overline{d^{-1}}$. Alors:
$$
\Delta_{\alpha_{l,m}}^{(\delta,\beta)}(A) =
\begin{pmatrix} 0 & \frac{f_m(c_{l,m})}{\delta d} \\ 0 & 0 \end{pmatrix}.
$$
Tous les autres $\Delta_{\alpha'}^{(\delta',\beta')}(A)$ sont nuls.
\end{prop}
\begin{proof}
L'assertion (i) est évidente. Notons, pour $\overline{c} \in \Eq \setminus \Sigma_{A_0}$,
$F_{\overline{c}} = \begin{pmatrix} 1 & f_{\overline{c}} \\ 0 & 1 \end{pmatrix}$.
D'après le calcul général:
$$
f_{\overline{c}}(z) = \dfrac{1}{\theta_{q,c}^\delta}
\sum_{m \in \Z} \dfrac{v_m(c)}{q^m c^\delta - a/b}
\Longrightarrow
\phi(c) = \sum_{m \in \Z} \dfrac{f_m(c)}{q^m c^\delta - d},
$$
où l'on a introduit l'évaluation $\phi(c) := f_{\overline{c}}(z_0)$ au point base.
D'après le calcul général, qu'il est facile de vérifier directement,
$v_{m-\delta}(qc) = (qc)^\delta v_m(c)$, $f_{m-\delta}(qc) = f_m(c)$ et
$\phi(qc) = \phi(c)$ comme il se doit. \\
Les $\delta^2$ pôles de $\phi$ sur $\Eq$ sont les $\overline{c_{l,m}}$; au voisinage
de $c = c_{l,m}$, on a:
$$
\phi(c) = \dfrac{f_m(c)}{q^m (c^\delta - c_{l,m}^\delta)} + \cdots =
\dfrac{f_m(c_{l,m})}{q^m (c - c_{l,m}) \delta c_{l,m}^{\delta-1}} + \cdots =
\dfrac{c_{l,m}}{c - c_{l,m}} \dfrac{f_m(c_{l,m})}{q^m \delta q^{-m} d} + \cdots,
$$
d'où l'on tire enfin
$$
\Res_{q,c_{l,m}} \phi(c) = \dfrac{f_m(c_{l,m})}{\delta d} \cdot
$$
Cela justifie la formule de l'énoncé. La nullité des autres
$\Delta_{\alpha'}^{(\delta',\beta')}(A)$ vient de même.
\end{proof}

% 4.2.3

\subsubsection{Effet d'une dilatation de la variable $z$}
\label{subsubsection:dilatation}

Nous aurons besoin au \ref{subsubsection:basecanonique} puis à nouveau au
\ref{subsubsection:actiondeGpsurSt'} de la conséquence suivante:

\begin{cor}
\label{cor:Delta(A(lambdaz))}
Outre les notations que ci-dessus, soient $\lambda \in \C^*$ et
$A'(z) := A(\lambda z)$. \\
(i) Notant $c'_{l,m} := \lambda^{-1} c_{l,m}$, on a:
$$
\Sigma_{A'_0} = \{\overline{c'_{l,m}} \tq l,m = 0,\ldots,\delta-1\} =
\overline{\lambda} \Sigma_{A_0}.
$$
(ii) Soient $\alpha'_{l,m} := \overline{c'_{l,m}}$, $d' := \lambda^{-\delta} d$
et $\beta' := \overline{{d'}^{-1}} = \overline{\lambda}^\delta \beta$. Alors:
$$
\Delta_{\alpha'_{l,m}}^{(\delta,\beta')}(A') = 
\left(\dfrac{\thq(z_0/c_{l,m})}{\thq(z_0/c'_{l,m})}\right)^\delta
\lambda^m \Delta_{\alpha_{l,m}}^{(\delta,\beta)}(A).
$$
\end{cor}
\begin{proof}
Les pentes de $A'$ sont les mêmes que celles de $A$, d'où même niveau $q$-Gevrey
$\delta' = \delta$. On a avec des notations évidentes $a' = \lambda^k a$,
$b' = \lambda^l b$ et $a'/b' = \lambda^{-\delta} a/b = \lambda^{-\delta} d = d'$;
on peut donc prendre $c'_{l,m} := \lambda^{-1} c_{l,m}$, d'où l'assertion (i). \\
De plus, $u'(z) = u(\lambda z)$ et $A'_0(z) = A_0(\lambda z)$. En particulier,
du développement $u = \sum u_p z^p$ on déduit $u' = \sum u'_p z^p$ avec
$u'_p = \lambda^p u_p$. \\
Notons $\thqd(z) = \sum\limits_{n \in \Z} t_n^{(\delta)} z^n$, d'où:
$$
u z^{-k} \thqc^\delta = \sum_{n,p} u_p z^p z^{-k} t_n^{(\delta)} z^n c^{-n}
\Longrightarrow
v_m(c) = \sum_{n+p-k = m} u_p t_n^{(\delta)} c^{-n} =
\sum_{n+p = m} u_p t_{n+k}^{(\delta)} c^{-(n+k)}.
$$
On en déduit, si $c' = \lambda^{-1} c$:
$$
v'_m(c') = v'_m(\lambda^{-1} c) =
\sum_{n+p = m} u'_p t_{n+k}^{(\delta)} (\lambda^{-1} c)^{-(n+k)} = \lambda^{n+k} v_m(c).
$$
Le coefficient surdiagonal $N'$ de $\Delta_{\alpha'_{l,m}}^{(\delta,\beta')}(A')$ est donc,
d'après la proposition:
$$
N' = \dfrac{z_0^m v'_m(c'_{l,m})}{\delta a' \left(\thq(z_0/c'_{l,m})\right)^\delta} =
\left(\dfrac{\thq(z_0/c_{l,m})}{\thq(z_0/c'_{l,m})}\right)^\delta \lambda^m N,
$$
où $N$ désigne le coefficient surdiagonal de $\Delta_{\alpha_{l,m}}^{(\delta,\beta)}(A)$.
Comme
$\Delta_{\alpha_{l,m}}^{(\delta,\beta)}(A) = \begin{pmatrix} 0 & N \\ 0 & 0 \end{pmatrix}$ et
$\Delta_{\alpha'_{l,m}}^{(\delta,\beta')}(A') = \begin{pmatrix} 0 & N' \\ 0 & 0 \end{pmatrix}$,
la conclusion s'ensuit.
\end{proof}

Remarquons pour usage ultérieur que, si $\phi \in \Hom_{gr}(\C^*,\C^*)$ est tel que
$\phi(q) = \lambda$, alors $\lambda^m = \phi(d/c_{l,m}^\delta)$.

% 4.2.4

\subsubsection{Une base canonique conditionnelle de $\sttqu$}
\label{subsubsection:basecanonique}

\paragraph{Préliminaires.}

Fixons $\delta \in \N^*$ et $\beta \in \Eq$. D'après \cite[\S 3.3]{RS3}, l'espace
vectoriel $\left(\dfrac{\sttqu}{[\sttqu,\sttqu]}\right)^{(\delta,\beta)}$ est engendré
par les $\delta^2$ classes des $\Delta_\alpha^{(\delta,\beta)}$, $\alpha \in \Eq$,
$\alpha^\delta = \beta$. De plus, il est de dimension $\delta$ car dual de l'espace
$K_{0,\delta} = \C \oplus \C z \oplus \cdots \oplus \C z^{\delta - 1}$. On peut réaliser
l'accouplement comme suit; soient $a,b \in \Cs$ tels que $\pi(b/a) = \beta$ et
$k,l \in \Z$ tels que $l - k = \delta$, aucune autre condition n'étant posée. Alors
la formule:
$$
\langle \Delta,u \rangle := \Delta_\alpha^{(\delta,\beta)}(A_u), \quad
A_u := \begin{pmatrix} a z^k & u \\ 0 & b z^l \end{pmatrix}
$$
réalise cette dualité. On supposera désormais que $k = 0$ et $b = 1$:
$$
A_u := \begin{pmatrix} a & u \\ 0 & z^\delta \end{pmatrix}.
$$

Si pour tous $\delta \in \N^*$ et $\beta \in \Eq$ on choisit $\delta$ éléments
$\Delta_i^{(\delta,\beta)}$, $i = 0,\ldots,\delta-1$ parmi les $\delta^2$ éléments
$\Delta_\alpha^{(\delta,\beta)}$ de telle sorte que leurs classes forment une base
de $\left(\dfrac{\sttqu}{[\sttqu,\sttqu]}\right)^{(\delta,\beta)}$, et si l'on adjoint
à la famille de tous les $\Delta_i^{(\delta,\beta)}$ ($\delta$, $\beta$ et $i$ variant
comme indiqué) le générateur $\tau$ de $\sttq^{(0,\overline{1})}$, on obtient comme
énoncé dans le théorème \ref{thm:thprincipalpentesentières} une base d'une algèbre
de Lie graduée $L_q^1$ dont le rôle est crucial dans la correspondance de
Riemann-Hilbert-Birkhoff (on l'a vu en \ref{subsubsection:qdérivéesétrangères},
on le verra à nouveau en \ref{subsection:groupedeGaloispentesarbitraires}). Pour
réaliser un tel choix, il suffit (par la dualité énoncée plus haut) que les $\delta$
formes linéaires $u \mapsto \langle \Delta_i^{(\delta,\beta)},u \rangle$ sur $K_{0,\delta}$
soient linéairement indépendantes.

\paragraph{Relations entre les $\delta^2$ $q$-dérivées étrangères $\Delta_\alpha^{(\delta,\beta)}$.}

On fixe donc $a \in \C^*$ et $\delta \in \N^*$. Notons $u_j := z^j$ les éléments
de la base canonique de $K_{0,\delta}$ et:
$$
A_j := A_{u_j} = \begin{pmatrix} a & z^j \\ 0 & z^\delta \end{pmatrix},
j = 0,\ldots,\delta-1.
$$
Avec les notations de \ref{subsubsection:calculàdeuxpentes}, on a donc $d = a$.
On choisit pour $c$ une détermination particlulière de la racine $\delta$\ieme\
de $a$ (un tel choix sera précisé juste avant le théorème \ref{thm:basecanonique}).
On pose alors, en cohérence avec \ref{subsubsection:dilatation}:
$$
\beta := \overline{d^{-1}}, \quad
c_{l,m} := \zeta_\delta^{-l} q_\delta^{-m} c \quad \text{~et~}
\alpha_{l,m} := \overline{c_{l,m}}.
$$
Enfin, pour simplifier l'application des formules du \ref{subsubsection:dilatation}
(en particulier du corollaire \ref{cor:Delta(A(lambdaz))}), on pose:
$$
\psi_{l,m}(u) := \left(\thq(z_0/c_{l,m})\right)^\delta \Delta_{\alpha_{l,m}}^{(\delta,\beta)}(A_u).
$$
Nous voulons trouver $\delta$ parmi les formes linéaires $\psi_{l,m}$,
$l,m = 0,\ldots,\delta-1$ sur $K_{0,\delta}$ qui soient indépendantes. Pour cela,
nous appliquerons aux $\psi_{l,m}(u_j)$ le corollaire \ref{cor:Delta(A(lambdaz))}
pour des valeurs bien choisies de $\lambda$. On aura donc:
$$
A'_u := A_u(\lambda z) =
\begin{pmatrix} a & u' \\ 0 & \lambda^\delta z^\delta \end{pmatrix} \quad
\text{et en particulier:~}
A'_j := A_j(\lambda z) =
\begin{pmatrix} a & \lambda^j z^j  \\ 0 & \lambda^\delta z^\delta \end{pmatrix}.
$$

\paragraph{Dilatation de facteur $\lambda = \zeta_\delta$.}

On a dans ce cas $\lambda^\delta = 1$. Avec les notations de
\ref{subsubsection:dilatation}, on voit que $a' = a$, $d' = d$,
$c'_{l,m} = \zeta_\delta^{-1} c_{l,m} = c_{l+1,m}$, si l'on convient de calculer
l'indice $l$ modulo $\delta$ (ce qui est cohérent puisqu'il n'intervient que
via le facteur $\zeta_\delta^{-l}$), de sorte que $\alpha_{\delta,m} = \alpha_{0,m}$;
et donc $\beta' = \beta$ et $\alpha'_{l,m} = \alpha_{l+1,m}$. Le corollaire
\ref{cor:Delta(A(lambdaz))} donne ici:
$$
\Delta_{\alpha'_{l,m}}^{(\delta,\beta')}(A_j') =
\left(\dfrac{\thq(z_0/c_{l,m})}{\thq(z_0/c'_{l,m})}\right)^\delta
\zeta_\delta^m \; \Delta_{\alpha_{l,m}}^{(\delta,\beta)}(A_j).
$$
La linéarité en $u_j' = \zeta_\delta^j u_j$ de $\Delta(A_j')$ entraîne:
$$
\Delta_{\alpha'_{l,m}}^{(\delta,\beta')}(A_j') =
\zeta_\delta^j \; \Delta_{\alpha'_{l,m}}^{(\delta,\beta')}(A_j).
$$
On a donc en combinant:
$$
\Delta_{\alpha_{l+1,m}}^{(\delta,\beta)}(A_j) = 
\left(\dfrac{\thq(z_0/c_{l,m})}{\thq(z_0/c'_{l,m})}\right)^\delta \zeta_\delta^{m-j} \;
\Delta_{\alpha_{l,m}}^{(\delta,\beta)}(A_j),
$$
d'où, en chassant les dénominateurs:
\begin{equation}
\label{eqn:basecanfacteurzeta}
\psi_{l+1,m}(u_j) = \zeta_\delta^{m-j} \psi_{l,m}(u_j).
\end{equation}

\paragraph{Dilatation de facteur $\lambda = q_\delta$.}

On a maintenant $\lambda^\delta = q$, d'où:
$$
A'_j = \begin{pmatrix} a & q_\delta^j z^j \\ 0 & q z^\delta \end{pmatrix},
$$
d'où encore $a' = a$ mais $d' = d/q$, donc $\beta' = \beta$ et:
$$
c'_{l,m} = q_\delta^{-1} c_{l,m} =
\begin{cases} c_{l,m+1} \text{~si~} m < \delta - 1, \\
q^{-1} c_{l,0} \text{~si~} m = \delta - 1, \end{cases}
$$
et, dans tous les cas, $\alpha'_{l,m} = \alpha_{l,m+1}$ (où $m$ est ici compté
modulo $\delta$). \\

D'après le corollaire \ref{cor:Delta(A(lambdaz))}:
$$
\Delta_{\alpha_{l,m+1}}^{(\delta,\beta)}(A_j') =
\left(\dfrac{\thq(z_0/c_{l,m})}{\thq(z_0/c'_{l,m})}\right)^\delta
q_\delta^m \; \Delta_{\alpha_{l,m}}^{(\delta,\beta)}(A_j).
$$
Dans le cas spécial où $m = \delta - 1$, $c'_{l,\delta-1} = q^{-1} c_{l,0}$,
le facteur theta s'écrit:
$$
\left(\dfrac{\thq(z_0/c_{l,\delta-1})}{\thq(z_0/c'_{l,\delta - 1})}\right)^\delta =
\left(\dfrac{\thq(z_0/c_{l,\delta-1})}{\thq(q z_0/c_{l,0})}\right)^\delta =
(c_{l,0}/z_0)^\delta \left(\dfrac{\thq(z_0/c_{l,\delta-1})}{\thq(z_0/c_{l,0})}\right)^\delta =
(a/z_0^\delta) \left(\dfrac{\thq(z_0/c_{l,\delta-1})}{\thq(z_0/c_{l,0})}\right)^\delta.
$$
On peut donc écrire pour abréger l'égalité valable pour $m = 0,\ldots,\delta-1$:
$$
\left(\dfrac{\thq(z_0/c_{l,m})}{\thq(z_0/c'_{l,m})}\right)^\delta =
\left[\dfrac{a}{z_0^\delta}\right]_{m = \delta-1}
\left(\dfrac{\thq(z_0/c_{l,m})}{\thq(z_0/c_{l,m+1})}\right)^\delta,
$$
où le facteur $\dfrac{a}{z_0^\delta}$ entre crochets $\left[-\right]_{m = \delta-1}$
est par convention présent si $m = \delta-1$ et absent sinon et où $m$ est compté
modulo $\delta$. Notre formule précédente devient donc:
$$
\thq(z_0/c_{l,m+1})^\delta \, \Delta_{\alpha_{l,m+1}}^{(\delta,\beta)}(A_j') =
\left[\dfrac{a}{z_0^\delta}\right]_{m = \delta-1} q_\delta^m \psi_{l,m}(u_j).
$$
Pour élucider le facteur $\Delta_{\alpha_{l,m+1}}^{(\delta,\beta)}(A_j')$, on élimine le
facteur $q$ de $z^\delta$ par une transformation de jauge (vérifiable directement):
$$
A''_j := \begin{pmatrix} a & q_\delta^j z^{j+1} \\ 0 & z^\delta \end{pmatrix}
\overset{F}{\longrightarrow}
\begin{pmatrix} a & q_\delta^j z^j \\ 0 & q z^\delta \end{pmatrix} = A'_j,
\quad \text{où~} F :=  \begin{pmatrix} 1 & 0 \\ 0 & z \end{pmatrix},
$$
d'où l'on tire, par fonctorialité et linéarité, si $j < \delta - 1$:
$$
\Delta_{\alpha_{l,m+1}}^{(\delta,\beta)}(A'_j) =
F(z_0) \Delta_{\alpha_{l,m+1}}^{(\delta,\beta)}(A''_j) F(z_0)^{-1} =
q_\delta^j F(z_0) \Delta_{\alpha_{l,m+1}}^{(\delta,\beta)}(A_{j+1}) F(z_0)^{-1} =
\dfrac{q_\delta^j}{z_0} \Delta_{\alpha_{l,m+1}}^{(\delta,\beta)}(A_{j+1}),
$$
le calcul de la conjugaison par $F(z_0)$ étant justifié par la forme triangulaire
supérieure stricte de $\Delta_{\alpha_{l,m+1}}^{(\delta,\beta)}(A_{j+1})$. On en déduit
une première formule:
$$
\psi_{l,m+1}(u_{j+1}) =
\left[\dfrac{a}{z_0^\delta}\right]_{m = \delta-1} z_0 q_\delta^{m-j} \psi_{l,m}(u_j).
$$
Si maintenant $j = \delta-1$, on ramène le $z^\delta$ qui se trouve en position
$(1,2)$ de $A''_j$ à un élément de la base des $u_j$ par une transformation de
jauge (vérifiable directement):
$$
A'''_{\delta-1} := \begin{pmatrix} a & a q_\delta^{\delta-1} \\ 0 & z^\delta \end{pmatrix}
\overset{G}{\longrightarrow}
\begin{pmatrix} a & q_\delta^{\delta-1} z^\delta \\ 0 & z^\delta \end{pmatrix} = A''_{\delta-1},
\quad \text{où:~} G :=  \begin{pmatrix} 1 & q_\delta^{\delta-1} \\ 0 & 1 \end{pmatrix},
$$
d'où l'on tire, par fonctorialité et linéarité:
$$
\Delta_{\alpha_{l,m+1}}^{(\delta,\beta)}(A'_{\delta-1}) = 
(FG)(z_0) \Delta_{\alpha_{l,m+1}}^{(\delta,\beta)}(A'''_{\delta-1}) (FG(z_0))^{-1} =
\dfrac{a q_\delta^{\delta-1}}{z_0} \, \Delta_{\alpha_{l,m+1}}^{(\delta,\beta)}(A_0),
$$
d'où enfin:
$$
\psi_{l,m+1}(u_0) = \left[\dfrac{a}{z_0^\delta}\right]_{m = \delta-1}
\dfrac{z_0}{a} q_\delta^{m - (\delta-1)} \psi_{l,m}(u_{\delta-1}).
$$
On peut réunir les deux cas en une seule formule avec la convention analogue
sur le facteur entre crochets (et où l'on compte $j$ et $m$ modulo $\delta$):
\begin{equation}
\label{eqn:basecanfacteurqdelta}
\psi_{l,m+1}(u_{j+1}) = \left[\dfrac{a}{z_0^\delta}\right]_{m = \delta-1}
\left[\dfrac{1}{a}\right]_{j = \delta-1} z_0 \, q_\delta^{m-j} \, \psi_{l,m}(u_j).
\end{equation}

\paragraph{Puissances de la fonction theta et \og bonnes valeurs \fg\ de $q$.}

Rappelons que la fonction theta $\thq$, $\lmod q \rmod > 1$, a été définie par
la formule $\thq(z) := \sum\limits_{m \in \Z} q^{-m(m+1)/2} z^m$, $z \in \Cs$,  et
que les coefficients $t_n^{(\delta)}$, $\delta \in \N^*$, $n \in \Z$, l'ont été
(au \ref{subsubsection:dilatation}) par le développement en série de Laurent:
$$
\thqd(z) = \sum\limits_{n \in \Z} t_n^{(\delta)} z^n.
$$
Ces séries étant absolument convergentes, on a la formule explicite:
$$
t_n^{(\delta)} = \sum_{m_1 + \cdots + m_\delta = n} q^{- (m_1(m_1+1) + \cdots + m_\delta(m_\delta+1))/2},
$$
qui entraîne que chaque $t_n^{(\delta)}$ est une fonction holomorphe de $q$ sur l'ouvert
connexe $\lmod q \rmod > 1$. Comme $q > 0 \Rightarrow t_n^{(\delta)}(q) > 0$, cette
fonction est non triviale et l'ensemble de ses zéros est discret donc dénombrable.
En conséquence, l'ensemble de tous les zéros de toutes les fonctions $t_n^{(\delta)}$
est dénombrable. \\

Il est évident que $t_n^{(1)}(q) = q^{-n(n+1)/2}$ ne s'annule jamais; c'est aussi vrai
de $t_n^{(2)}(q)$ en vertu du calcul suivant:
\begin{align*}
t_n^{(2)}(q)
&= \sum_{l+m=n} q^{-(l(l+1) + m(m+1))/2} \\
&= \sum_{l+m=n} q^{-n(n+1)/2} q^{lm} \\
&= q^{-n(n+1)/2} \sum_m q^{m(n-m)} \\
&= q^{-n(n+1)/2} \sum_m (q^2)^{-m(m+1)/2} (q^{n+1})^m \\
&= q^{-n(n+1)/2} \theta_{q^2}(q^{n+1}),
\end{align*}
qui ne peut s'annuler pour $\lmod q \rmod > 1$ en vertu de la formule du triple
produit et de l'impossibilité de la relation $q^{n+1} = - q^k$, $k \in \Z$. 

\begin{defn}
\label{defn:bonnevaleur}
On dira que $q$ a une \emph{bonne valeur} si:
$$
\forall \delta \in \N^* \;,\; \forall n \in \Z \;,\; t_n^{(\delta)}(q) \neq 0.
$$
\end{defn}

C'est donc une propriété génériquement satisfaite (\ie\ hors d'un ensemble au plus
dénombrable). Nous verrons que toute valeur de $q$ n'est pas bonne. \\

Le calcul des $t_n{(2)}$ indiqué plus haut permet facilement d'obtenir la formule:
$$
\theta_q^2(z) 
= \theta_{q^2}(q) \theta_{q^2}(z^2) + \theta_{q^2}(1) \theta_{q^2}(q z^2) z^{-1} 
= \theta_{q^2}(q) \theta_{q^2}(z^2) + \theta_{q^2}(1) \theta_{q^2}(q^{-1} z^2) z
$$
Il est difficile de ne pas espérer qu'une telle formule admette des généralisations
aux puissances supérieures. Selon le Professeur Bruce Berndt, les puissances de
fonctions theta apparaissent dans le volume III des \og Ramanujan Notebooks \fg\
\cite{BerndtRamanujanNBIII}; et en effet, il y donne la preuve (\emph{Entry 29}) d'une
telle relation dont l'égalité ci-dessus est conséquence facile; et l'\emph{Entry 30}
semble donner la clé d'un calcul de $\theta_q^3(z)$ et peut-être de $\theta_q^4(z)$,
mais je n'ai pas été capable de m'en assurer. \\

Changgui Zhang (communication personnelle) a développé un mode de calcul qui lui
permet d'établir itérativement de telles identités. Il m'a signalé avoir déduit
d'une telle identité que $t_0^{(3)}(q)$ s'annule au moins pour une valeur réelle
négative de $q$. J'espère qu'il publiera ses résultats et vais en donner une preuve
toute différente, dont je crois qu'elle est également susceptible de généralisation.
On vérifie d'abord aisément le calcul suivant:
\begin{align*}
t_0^{(3)}(q)
&= \sum_{a,b,c \in \Z \atop a+b+c = 0} q^{- (a^2 + a + b^2 + b + c^2 + c)/2} \\
&= \sum_{a,b,c \in \Z \atop a+b+c = 0} q^{ab + bc + ca} \\
&= \sum_{a,b \in \Z} q^{-(a^2 + ab + b^2)} \\
&= f(q^{-1}),
\end{align*}
où l'on a posé, pour $\lmod x \rmod < 1$,
$$
f(x) := \sum_{a,b \in \Z} x^{a^2 + ab + b^2} = \sum_{n \geq 0} r(n) x^n.
$$
Ici $r(n)$ désigne le nombre de couples $(a,b) \in \Z \times \Z$ tels que
$a^2 + ab + b^2 = n$ (le calcul est entièrement justifié par le fait que
cette forme quadratique est définie positive). L'entier:
$$
R(n) := r(0) + \cdots + r(n) = \text{card} \{(a,b) \in \Z \times \Z \tq a^2 + ab + b^2 \leq n\}
$$
admet, lorsque $n \to + \infty$, l'équivalent $C n$, où $C$ est l'aire de l'ellipse
$x^2 + xy + y^2 \leq n$, soit $2 \pi/\sqrt{3}$. On en déduit par théorème abélien
standard que, lorsque $x \to 1^-$:
$$
\dfrac{f(x)}{1-x} = \sum_{n \geq 0} R(n) x^n \sim \sum C n x^n = C \dfrac{x}{(1-x)^2}
\Longrightarrow f(x) \sim \dfrac{C}{1-x} \cdot
$$
D'autre part, l'exposant $a^2 + ab + b^2$ est pair si, et seulement si $a$ et $b$
le sont. La partie paire de la série $f(x)$ est donc la sous-somme sur les tels
couples, c'est-à-dire $f(x^4)$, et l'on a donc $f(-x) = 2 f(x^4) - f(x)$. Ainsi,
toujours pour $x \to 1^-$:
$$
f(x^4) \sim \dfrac{C}{1-x^4} \sim \dfrac{C}{4(1-x)} \Longrightarrow
f(-x) \sim \dfrac{-C}{2(1-x}) \Longrightarrow
\lim_{x \to -1 \atop x > -1} f(x) = - \infty.
$$
Comme $f(0) = 1$, la fonction s'annule au moins une fois sur $\left]-1,0\right[$
et $t_0^{(3)}(q)$ s'annule donc au moins une fois sur $\left]-1,-\infty\right[$.

\paragraph{Une base canonique conditionnelle.}

On opère désormais les choix suivants, en cohérence avec
\ref{subsubsection:calculàdeuxpentes} et \ref{subsubsection:dilatation}:
\begin{enumerate}
\item On prend pour chaque $\beta \in \Eq$ le représentant $d$ de $\beta^{-1}$
tel que $d^{-1}$ appartienne à la la couronne fondamentale:
$$
1 \leq \lmod d^{-1} \rmod < \lmod q \rmod;
$$
\item pour $\delta \in \N^*$ donné, on appelle $c$ l'unique racine $\delta$\ieme\
de $d$ dont un argument appartient à $\left]-2\pi/\delta,0\right]$, puis on pose,
pour $l,m = 0,\ldots,\delta-1$:      
$$
c_{l,m} := \zeta_\delta^{-l} q_\delta^{-m} c \quad \text{~et~}
\alpha_{l,m} := \overline{c_{l,m}};
$$
\item on note $\Delta_l^{(\delta,\beta)} := \Delta_{\alpha_{l,0}}^{(\delta,\beta)}$ pour
$l = 0,\ldots,\delta-1$.
\end{enumerate}

On voit donc que le module de $c_{l,m}$ appartient à
$\left]\lmod q_\delta\rmod^{-m-1},\lmod q_\delta\rmod^{-m}\right]$
et son argument à $\left]-2 (l+1) \pi/\delta,-2 l \pi/\delta\right]$; et aussi
que ces conditions déterminent parfaitement les $c_{l,m}$, une fois $\delta \in \N^*$
et $\beta \in \Eq$ fixés.

\begin{thm}
\label{thm:basecanonique}
On suppose que $q$ a une bonne valeur. Alors en adjoignant le générateur $\tau$
de $\sttq^{(0,\overline{1})}$ à la famille de tous les $\Delta_l^{(\delta,\beta)}$,
$\delta \in \N^*$, $\beta \in \Eq$, $l = 0,\ldots,\delta-1$, on obtient une base
d'une algèbre de Lie graduée $L_q^1$ possédant les propriétés énoncées dans le
théorème \ref{thm:thprincipalpentesentières}.
\end{thm}
\begin{proof}
Combinant les notations de \ref{subsubsection:basecanonique} à celles de la
proposition \ref{prop:lecas2x2} de \ref{subsubsection:calculàdeuxpentes}, on voit
que, si $u = u_j$ (avec ici $k = 0$):
$$
\sum v_m(c) z^m = u \theta_{q,c}^\delta = \sum t_n^{(\delta)} z^{n+j}/c^n,
$$
d'où $v_m(c) = t_{m-j}^{(\delta)}/c^{m-j}$, qui est non nul puisque $q$ a une bonne
valeur, ce qui implique que $f_m(c) \neq 0$. \\
Il en découle qu'aucun $\Delta_{\alpha_{l,m}}^{(\delta,\beta)}(A_j)$ ne s'annule, donc
aucun $\psi_{l,m}(u_j)$ non plus. On déduit alors de la formule de transformation
\eqref{eqn:basecanfacteurzeta} l'égalité matricielle:
$$
\left(\psi_{i,0}(u_j)\right)_{0 \leq i,j \leq \delta-1} =
\left(\zeta_\delta^{-ij}\right)_{0 \leq i,j \leq \delta-1}
\Diag\left(\psi_{0,0}(u_j)\right)_{0 \leq j \leq \delta-1}.
$$
Le premier facteur du membre droit est une matrice de Vandermonde inversible,
le second facteur une matrice diagonale inversible puisqu'aucun $\psi_{l,m}(u_j)$
ne s'annule. Le membre gauche est donc inversible, ce qui permet de conclure.
\end{proof}

% 4.3

\subsection{Pentes arbitraires: vue d'ensemble}
\label{subsection:vue d'ensemble}

Soit $\EE_q$ la catégorie de toutes les équations aux $q$-différences
sur $K$, que nous identifions à la catégorie des matrices en forme
\eqref{eqn:formestandardtriangulaire} (\cf\ \ref{subsub:filgrad})
où chaque $B_i \in \GL_{r_i}(K)$ est pur de pente $\mu_i$ et où chaque 
$U_{i,j} \in \Mat_{r_i,r_j}(\C[z,z^{-1}])$; de plus, selon van der Put et
Reversat \cite{vdPR}, on peut supposer que $B_i \in \GL_{r_i}(\C[z,z^{-1}])$.
La sous-catégorie $\EE_{q,p}$ des purs sera donc identifiée à la catégorie des
matrices $A_0 := \gr A_U$, qui ont la forme \eqref{eqn:formestandarddiagonale}
avec les mêmes conventions sur les $B_i$. \\

Les groupes tannakiens associés sont notés $\Gq$ et $\Gqp$. Le groupe de Stokes
$\Stq$ est le noyau du morphisme $i^*: \Gq \rightarrow \Gqp$ dual de l'inclusion
$i$ de $\EE_{q,p}$ dans $\EE_q$; comme le foncteur $\gr$ de $\EE_q$
dans $\EE_{q,p}$ est une rétraction de l'inclusion (\ie\ $\gr \circ i$ est
le foncteur identité de $\EE_{q,p}$), le morphisme $\gr^*: \Gqp \rightarrow \Gq$
est une section de $i^*$ (\ie\ $i^* \circ \gr^*$ est le morphisme identité
de $\Gqp$) et l'on retrouve la décomposition en produit semi-direct:
$$
\Gq = \Stq \ltimes \Gqp.
$$
On note $\stq := \Lie(\Stq)$. Bien que $\Gqp$ ne soit plus un groupe abélien, 
nous verrons en \ref{subsubsection:Gp=GpsxGpu} que l'on a encore
une décomposition de $\Gqp$ en produit direct
$\Gqpu \times \Gqps$, où $\Gqpu = \C$, qui agit sur les 
$q$-logarithmes; et l'on adjoint encore la \og composante logarithmique \fg\ 
$\Gqpu$ au groupe de Stokes, en posant $\Sttq := \Stq \times \Gqpu$ (le produit
direct est justifié car $\Gqpu$ commute avec $\Stq$) puis $\sttq := \Lie(\Sttq)$,
d'où les relations:
$$
\Gq = \Sttq \ltimes \Gqps = \sttq \ltimes \Gqps, \quad
\sttq = \stq \oplus \C \tau, \quad [\sttq,\C \tau] = 0.
$$
Le sens de l'objet hybride $\sttq \ltimes \Gqps$ a été expliqué au numéro
\ref{subsubsection:correspondanceEE1<->Rep(Gq1)}. Les commutations seront
justifiées au numéro \ref{subsubsection:Gp=GpsxGpu}.

% 4.3.1

\subsubsection{Le treillis des ramifications}
\label{subsubsection:treillis}

Rappelons (\cf\ \ref{subsubsection:restrictionspentes}) que $\EE_q^r$,
désigne la sous-catégorie pleine de $\EE_q$ formée des objets dont toutes les
pentes appartiennent à $\frac{1}{r} \Z$. C'est une sous-catégorie tannakienne
de groupe $\Gqr$. Si $r,s \in \N^*$ sont tels que $r | s$, l'inclusion
$\EE_q^r \subset \EE_q^s$ donne lieu par dualité à un morphisme
$\mathbf{G}_q^s \rightarrow \Gqr$.

\begin{prop}
Le morphisme $\mathbf{G}_q^s \rightarrow \Gqr$ est fidèlement plat (et donc
surjectif).
\end{prop}
\begin{proof}
On invoque \cite[prop. 2.2.1 (a) p. 139]{DM}:
\begin{itemize}
\item Le foncteur d'inclusion $\EE_q^r \leadsto \EE_q^s$ est pleinement fidèle
car la sous-catégorie $\EE_q^r$ de $\EE_q^s$ est pleine.
\item Tous sous-objet $M'$ dans $\EE_q^s$ d'un objet $M$ de $\EE_q^r$ est lui-même
dans un sous-objet de $\EE_q^r$ car (\cf\ \ref{subsubsection:polygonedeNewton})
$S(M') \subset S(M) \subset \frac{1}{r} \Z$.
\end{itemize}
\end{proof}

Puisque $\EE_q$ est la limite inductive des $\EE_q^r$ et donc son groupe $\Gq$
la limite projective des $\Gqr$, on voit que les morphismes $\Gq \rightarrow \Gqr$
sont eux-mêmes fidèlement plats (et donc surjectifs). \\

Rappelons encore de \ref{subsubsection:foncteursramif} les notations concernant
le foncteur de ramification $\text{Ram}_r: \EE_q^{r} \leadsto \EE_{q_r}^1$. Il s'en
déduit dualement un morphisme $\mathbf{G}_{q_r}^1 \rightarrow \Gqr$ dont on prouvera
au \ref{subsubsection:immersionsfermées} que c'est une immersion fermée.
Plus généralement, soient $r,r',s \in \N^*$ tels que $s = r r'$, donc $r | s$,
d'où un foncteur de ramification $\EE_q^{s} \leadsto \EE_{q_r}^{r'}$ qu'on notera
encore $\text{Ram}_r$; et dualement un morphisme
$\mathbf{G}_{q_r}^{r'} \rightarrow \mathbf{G}_q^s$.
Ce sont ces foncteurs et ces immersions fermées qui nous permettront d'étudier
les $\Gqr$. Ils s'organisent en deux treillis duaux l'un de l'autre:

$$
\xymatrix{
\EE_q \ar@{->>}[rr] & &
\EE_{q_r} \ar@{->>}[rr] & &
\EE_{q_s} \\
\EE_q^s \ar@{->>}[rr] \ar@{->>}[u] & &
\EE_{q_r}^{r'} \ar@{->>}[rr] \ar@<0ex>[u] & &
\EE_{q_s}^1 \ar@{->>}[u] \\
\EE_q^r \ar@{->>}[rr] \ar@{->>}[u] & &
\EE_{q_r}^1 \ar@{->>}[u] & & \\
\EE_q^1 \ar@{->>}[u] & & & &
}
\qquad \qquad \qquad 
\xymatrix{
\Gq \ar@{^{(}->}[d] & &
\mathbf{G}_{q_r} \ar@{_{(}->}[ll] \ar@{^{(}->}[d] & &
\mathbf{G}_{q_s} \ar@{_{(}->}[ll] \ar@{^{(}->}[d] \\
\mathbf{G}_q^s \ar@{^{(}->}[d] & &
\mathbf{G}_{q_r}^{r'} \ar@{_{(}->}[ll] \ar@{^{(}->}[d] & &
\mathbf{G}_{q_s}^1 \ar@{_{(}->}[ll] \\
\Gqr \ar@{^{(}->}[d] & &
\mathbf{G}_{q_r}^1 \ar@{_{(}->}[ll] & & \\
\mathbf{G}_q^1 & & & &
}
$$

% 4.3.2

\subsubsection{Immersions fermées}
\label{subsubsection:immersionsfermées}

\begin{prop}
Le morphisme $\mathbf{G}_{q_r}^1 \rightarrow \Gqr$ est une immersion fermée.
\end{prop}
\begin{proof}
En vertu de \cite[prop. 2.2.1 (a) p. 139]{DM}, cela résulte directement
du lemme ci-dessous.
\end{proof}
  
On notera pour simplifier $\EE := \EE_q^r$, $\EE' := \EE_{q_r}^1$ et
$R: \EE \leadsto \EE'$ le foncteur de ramification $\text{Ram}_r$. De même
on notera $q' := q_r$, $z' := z^{1/r}$ et $K' := K_r = K[z'] = \C(\{z'\})$.
On conviendra que si $A = (a_{i,j})$, alors $A \otimes B$ désigne la
matrice de blocs $a_{i,j} B$.

\begin{lem}
Tout objet de $\EE'$ est isomorphe à un sous-objet d'un $R(X)$, $X$
objet de $\EE$.
\end{lem}
\begin{proof}
On note $j$ un générateur de $\mu_r$ et $\tau: f(z') \mapsto f(jz')$, qui est
donc un générateur du groupe de Galois de l'extension cyclique $K'/K$. On note
$\sigma$ l'automorphisme $\sq$ de $K$ ainsi que son extension $\sigma_{q'}$ à $K'$
(donc $\sigma$ et $\tau$ commutent). \\
Soit $A'\in \GLn(K') = \GLn\left(\C(\{z'\})\right)$ un objet de $\EE'$.
Notant $A'_l := \tau^l A'$, on plonge $A'$ dans la matrice diagonale par blocs
$B := \Diag(A'_0,\ldots,A'_{r-1})$ via le morphisme
$\begin{pmatrix} I_n \\ 0_n \\ \vdots \\ 0_n \end{pmatrix}$ ($r$
blocs). \\
Soient:
$$
T_r := \begin{pmatrix}
0 & 1 & 0 & \ldots & 0 \\
0 & 0 & 1 & \ddots & 0 \\
\vdots & \vdots & \vdots & \ddots & \vdots \\
0 & 0 & 0 & \ldots & 1 \\
1 & 0 & 0 & \ldots & 0 \\
\end{pmatrix} \in \GL_r(\C)
\quad \text{~et~} \quad
T_{r,n} := T_r \otimes I_n = \begin{pmatrix}
0_n & I_n & 0_n & \ldots & 0_n \\
0_n & 0_n & I_n & \ddots & 0_n \\
\vdots & \vdots & \vdots & \ddots & \vdots \\
0_n & 0_n & 0_n & \ldots & I_n \\
I_n & 0_n & 0_n & \ldots & 0_n \\
\end{pmatrix} \in \GL_{nr}(\C).
$$
On vérifie immédiatement que $(\tau B) T_{r,n}$ et $T_{r,n} B$ sont tous
deux égaux à $\Diag(A'_1,\ldots,A'_{r-1},A'_0)$, d'où la relation:
$$
\tau B = T_{r,n} B T_{r,n}^{-1}.
$$
On diagonalise $T_r = P \Delta P^{-1}$, $\Delta := \Diag(1,j,\ldots,j^{r-1})$,
$P \in \GL_r(\C)$ et l'on pose $\overline{P} := P \otimes I_n$ puis
$C := \overline{P}^{-1} B \overline{P}$, qui est donc équivalente à $B$, admet
$A'$ comme sous-objet et vérifie de plus:
$$
\tau C = \overline{\Delta} C \overline{\Delta}^{-1}, \text{~où~}
\overline{\Delta} := \Delta \otimes I_n.
$$
Soient enfin $F := \Diag(1,z',\ldots,{z'}^{r-1}) \in \GL_r(\C(\{z'\})$ et 
$F := \Diag(1,q',\ldots,{q'}^{r-1}) \in \GL_r(\C)$, de sorte que:
$$
\tau F = \Delta F \text{~et~} \sq F = \Delta' F.
$$
Notant $\overline{F} := F \otimes I_n$ puis $D$ la matrice telle que
$C = \overline{F}[D]$, on voit que $D$ est équivalente à $C$ donc à $B$, donc
qu'elle admet $A'$ comme sous-objet et vérifie de plus (calcul direct laissé
au lecteur):
$$
\tau D = D,
$$
Autrement dit, $D \in \GLn(\Ka)$, ce qui achève la démonstration.
\end{proof}

\begin{ex}
Pour $r = 2$, le calcul est transparent. La matrice $A'\in \GLn(\C(\{z'\})$
se plonge par le morphisme $\begin{pmatrix} I_n \\ 0_n \end{pmatrix}$ dans
$\begin{pmatrix} A'(z') & 0_n \\ 0_n & A'(-z') \end{pmatrix}$, qui est équivalente
via l'isomorphisme $\begin{pmatrix} I_n & I_n \\ I_n & -I_n\end{pmatrix}$ à
$\begin{pmatrix} B(z) & z' C(z) \\ z' C(z) & B(z) \end{pmatrix}$, où l'on a posé
$A'(z') = B(z) + z' C(z)$; enfin, cette dernière matrice est à son tour équivalente
via l'isomorphisme $\begin{pmatrix} I_n & 0 \\ 0 & z' I_n\end{pmatrix}$ à la matrice
$\begin{pmatrix} B(z) & C(z) \\ q' z C(z) & q' B(z) \end{pmatrix} \in \GL_{2n}(\Ka)$.    
\end{ex}

% 4.3.3

\subsubsection{$\Gqp$ et $\Stqp$ suffisent}
\label{subsubsection:suffisancedesstokesnonramifiés}

On reprend les notations ci-dessus: $q'$, $z'$, $K'$, $\sigma$ et $\tau$.
De plus, on note $\EE := \EE_q^r$, $\EE' := \EE_{q'}^1 = \EE_{q_r}^1$ et
$\Gq,\Gqprime$,
leurs groupes tannakiens respectifs, ainsi que $\Gqp$, $\Stq$, $\Gqpprime$, $\Stqp$
les sous-groupes habituels de ces derniers. On a donc un diagramme commutatif
de suites exactes semi-scindées:
$$
\xymatrix{
1 \ar@<0ex>[r] & 
\Stqp \ar@<0ex>[d] \ar@<0ex>[r] & 
\Gqprime \ar@<0ex>[d] \ar@<0ex>[r] & 
\Gqpprime \ar@<0ex>[d] \ar@<0ex>[r] & 
1 \\
1 \ar@<0ex>[r] & 
\Stq \ar@<0ex>[r] & 
\Gq \ar@<0ex>[r] & 
\Gqp  \ar@<0ex>[r] & 
1 
}
$$
D'après \ref{subsubsection:immersionsfermées}, la flèche verticale du milieu
est une immersion fermée, donc les deux autres flèches verticales aussi (car
les inclusions $\Stqp \rightarrow \Gqprime$ et $\Stq \rightarrow \Gq$ le sont,
ainsi que $\gr^*: \Gqpprime \rightarrow \Gqprime$ et $\gr^*: \Gqp \rightarrow \Gq$).
On identifiera $\Stqp$ à un sous-groupe de $\Stq$, $\Gqp$ et $\Gqprime$ à des
sous-groupes de $\Gq$, etc.

\begin{prop}
Les sous-groupes $\Stqp$ et $\Gqp$ \og engendrent topologiquement \fg\ $\Gq$.
\end{prop}
\begin{proof}
C'est une façon abrégée de dire que le sous-groupe de $\Gq$ qu'ils engendrent est
Zariski-dense. En particulier, toute représentation rationnelle de $\Gq$ est
totalement déterminée par ses restrictions à $\Stqp$ et $\Gqp$. \\
On le prouvera à l'aide du critère de Chevalley invoqué sous la forme suivante:
si $\rho$ est une représentation rationnelle de $\Gq$ dans $\GLnc$ et si
$L \subset \C^n$ est une droite stable sous $\Stqp$ et sous $\Gqp$, on vérifiera
que $L$ est stable sous $\Gq$, \ie\ provient d'une sous-représentation de rang $1$.
Pour cela on utilisera la correspondance tannakienne par laquelle les
représentations rationnelles de $\Gq$ sont associées à des objets de $\EE$, etc. \\
Soient donc $\rho$ et $L$ comme ci-dessus et soit $A(z) \in \GLn(\Ka)$, objet
de $\EE$ qui lui correspond. On peut le prendre sous la forme $A_U$. Soient
$\rho' := \rho_{|\Gqprime}$ et $\rho_p := \rho_{|\Gqp}$. Puisque $L$ est stable sous
$\Gqp$ (\ie\ sous $\gr^* \Gqp$), il lui est associé un sous-objet de rang $1$
de $\gr A = A_0$, autrement dit, il existe $u \in \GL_1(K)= \Ka^*$ (objet de
rang $1$) et $X_0 \in \Mat_{n,1}(K)$, $X_0 \neq 0$ (monomorphisme) tels que:
\begin{equation}
\label{eqn:rhop}
A_0 X_0 = u \sigma X_0.
\end{equation}
De même, puisque $L$ est stable sous $\Stqp$ et, bien entendu, sous $\Gqpprime$
(car $\Gqpprime \subset \Gqp$) donc sous $\Gqprime$ (car ce dernier est engendré
par $\Stqp$ et $\Gqpprime$), il lui est associé un sous-objet de rang $1$ de
$A' := \text{Ram}_r(A)$ dans $\EE'$. Autrement dit, posant $A'(z') := A({z'}^r)$,
il existe $v \in \GL_1(K') = \C(\{z'\})^*$ (objet de rang $1$) et
$Y \in \Mat_{n,1}\left(\C(\{z'\})\right)$, $Y \neq 0$ (monomorphisme) tels que
\begin{equation}
\label{eqn:rho'}
A' Y = v \sigma Y.
\end{equation}
Le ramifié de \eqref{eqn:rhop} est $X_0':(\C,u') \rightarrow (\C^n,A_0')$, où
$u'(z') := u({z'}^r)$ et $X_0'(z') := X_0({z'}^r)$. Le gradué de \eqref{eqn:rho'}
est $Y_0:(\C,v) \rightarrow (\C^n,A_0')$ (car un objet de rang $1$ est son
propre gradué). Ces deux sous-objets de rang $1$ de $A_0'$ sont associés à
la même sous-représentation de $\rho'_{|G'_p} = \rho_{|G'_p} = (\rho_p)_{|G'_p}$,
donc définissent des sous-objets isomorphes (en tant que sous-objets) de $A'_0$.
Il existe donc $\phi: (\C,u') \isom (\C,v)$ tel que le diagramme suivant commute:
$$
\xymatrix{
(\C,u') \ar@<0ex>[dd]_\phi \ar@<0ex>[dr]^{X_0'} &  \\
& (\C^n,A'_0)  \\
(\C,v) \ar@<0ex>[ur]_{Y_0} & 
}
$$
autrement dit, $X_0' = Y_0 \phi = \phi Y_0$ (car $\phi$ est scalaire). \\
Posons $Z := \phi Y: (\C,u') \rightarrow (\C^n,A')$, morphisme dans $\EE'$ de
gradué $Z_0 = \phi Y_0 = X_0'$. Avec les notations en cours, $\tau Z_0 = Z_0$
(puisqu'il provient de la ramification de $X_0$. Mais $\tau Z_0 = (\tau Z)_0$.
Donc les morphismes $Z$ et $\tau Z$ ont même gradué. Le foncteur $\gr$ étant
fidèle, $Z = \tau Z$, \ie\ $Z$ provient par ramification d'un morphisme dans
$\EE$, autrement dit, d'un sous-objet $(\C,u) \rightarrow (\C^n,A)$ dont l'oubli
(foncteur fibre espace sous-jacent) est le même, soit $L$. Cela entraîne que
$L$ est stable sous $\Gq$.
\end{proof}

% 4.4

\subsection{Pentes arbitraires: le groupe de Galois formel}
\label{subsection:groupedeGaloisformel}

Le groupe formel (ou pur) pour des pentes arbitraires a été déterminé par
Marius van der Put et Marc Reversat \cite{vdPR} en utilisant par endroits
la théorie de Picard-Vessiot pour les $q$-différences \cite{vdPS1}. Nous
combinerons leurs résultats avec les précisions apportées par Virginie
Bugeaud \cite{VB}, qui leur a donné une forme plus explicite et plus commode
pour l'approche strictement tannakienne (et calculatoire) que nous suivons ici.

% 4.4.1

\subsubsection{Objets indécomposables et objets irréductibles}
\label{subsubsection:objindecobjirred}

Selon \cite{vdPR,VB}, tout objet \emph{irréductible} (ou simple) de $\EE_{q,p}$ est,
à isomorphisme près, de la forme:
$$
E(r,d,c) := \begin{pmatrix}
  0 & 1 & \ldots & 0 \\
  \vdots & \vdots & \ddots & \vdots \\
  0 & 0 & \ldots & 1 \\
  u & 0 & \ldots & 0
\end{pmatrix} \in \GL_r(K),
\quad \text{où} \quad u := q^{d(r-1)/2} c z^d.
$$
Ici, $r \in \N^*$, $d \in \Z$ et $d \wedge r = 1$. On peut imposer
$1 \leq \lmod c \rmod < \lmod q \rmod$ et il y a alors unicité.
Le facteur constant $q^{d(r-1)/2}$ n'est en soi pas significatif et vise
seulement à simplifier les formules. On vérifie sans peine que $E(r,d,c)$
est pur isocline de pente $d/r$. \\

Tout objet \emph{indécomposable} de $\EE_{q,p}$ est, à isomorphisme près, de la forme
$E(r,d,c) \otimes U_m$, où $U_m \in \GL_m(\C)$ est le bloc de Jordan unipotent
standard de taille $m$:
$$
U_m := \begin{pmatrix}
  1 & 1 & \ldots & 0 \\
  \vdots & \vdots & \ddots & \vdots \\
  0 & 0 & \ldots & 1 \\
  0 & 0 & \ldots & 1
\end{pmatrix} \in \GL_m(\C).
$$
L'objet $E(r,d,c) \otimes U_m$ est encore pur isocline de pente $d/r$. Ainsi,
les objets de $\EE_{q,p}$ sont, de manière essentiellement unique, des sommes
directes $\bigoplus E(r,d,c) \otimes U_m$.

\paragraph{Formulaire.}

On note:
$$
\zeta_r := e^{2 \ii \pi/r}, \text{~et~}
D_r := \Diag(1,\zeta_r,\ldots,\zeta_r^{r-1}), 
Z_r := (\zeta_r^{-ij})_{0 \leq i,j \leq r-1},
T_r := \begin{pmatrix}
  0 & 1 & \ldots & 0 \\
  \vdots & \vdots & \ddots & \vdots \\
  0 & 0 & \ldots & 1 \\
  1 & 0 & \ldots & 0
\end{pmatrix} \in \GL_r(\C).
$$
Cette dernière est donc une matrice de permutation cyclique. \\

On vérifie que $Z_r[T_r] = Z_r T_r Z_r^{-1} = D_r$, soit un isomorphisme:
$$
T_r \overset{Z_r}{\longrightarrow} D_r
$$
d'où également un isomorphisme dans la catégorie $\EE_{q_r}$, au dessus du corps
$K_r := \C(\{z_r\})$, $z_r^r = z$:
$$
a z_r^d T_r \overset{Z_r}{\longrightarrow} a z_r^d D_r
$$
Voici le formulaire associé:
\begin{enumerate}
\item $Z_r^{-1} = \dfrac{1}{r} \overline{Z_r}$ et ${}^t Z_r = Z_r$.
\item $D_r^{-1} = \overline{D_r}$.
\item $T_r^k D_r^l = \zeta_r^{kl} D_r^l T_r^k$, $k,l \in \Z$.
\item $T_r = Z_r^{-1} D_r Z_r = Z_r D_r^{-1} Z_r^{-1}$.
\end{enumerate}

\medskip

Soit maintenant $a$ une racine $r$\ieme\ arbitraire de $c$. Notons:
$$
G_{a,r,d} := \Diag(g_0,g_1,\ldots,g_{r-1}) \in \GL_r(K_r) \text{~où~}
g_j := \dfrac{1}{q^{j(j-1)d/2r} a^j z^{jd/r}},
\quad \text{puis} \quad F_{a,r,d} := Z_r G_{a,r,d}.
$$
On vérifie alors que $G_{a,r,d}[E(r,d,c)] = a z_r^d T_r$ et par conséquent
$F_{a,r,d}[E(r,d,c)] = a z_r^d D_r$; autrement dit, $G_{a,r,d}$ et $F_{a,r,d}$ sont
des isomorphismes dans la catégorie $\EE_{q_r}$:
$$
\xymatrix{
  E(r,d,c) \ar@<0ex>[rr]^{G_{a,r,d}} \ar@/^2pc/[rrrr]^{F_{a,r,d}}
  & & a z_r^d T_r \ar@<0ex>[rr]^{Z_r} & &
  a z_r^d D_r.
}
$$
Voici le formulaire associé:
\begin{enumerate}
\item $G_{a,r,d}(\zeta_r z_r) = D_r^{-d} G_{a,r,d}(z_r) = G_{a,r,d}(z_r) D_r^{-d}$.
\item $F_{a,r,d}(\zeta_r z_r) = T_r^d F_{a,r,d}(z_r) = F_{a,r,d}(z_r) D_r^{-d}$.
\end{enumerate}

% 4.4.2

\subsubsection{Éléments du groupe formel $\Gqp$}
\label{subsubsection:élémentsdeGp}

Outre le point-base $z_0$, on suppose choisie, pour tout $r$, une racine
$r$\ieme\ $z_{0,r}$ de $z_0$, avec les règles usuelles de compatibilité:
$z_{0,rs}^s = z_{0,r}$. \\

On interprète de la façon suivante un élément $\phi \in \Gqp$ du groupe formel:
à tout objet $A \in \GLn(K)$ de $\EE_{q,p}$, il associe une matrice carrée complexe
$\phi(A) \in \GLnc$, en accord avec les règles suivantes ($\otimes$-compatibilité
et fonctorialité):
\begin{enumerate}
\item $\phi(A \otimes B) = \phi(A) \otimes \phi(B)$.
\item Si $A \in \GLn(K)$, $B \in \GL_p(K)$ et $F \in \Mat_{p,n}(K)$ sont tels
que $(\sq F) A = B F$, alors $(\gr F)(z_0) \phi(A) = \phi(B) (\gr F)(z_0)$.
\end{enumerate}
La deuxième éqquation fait intervenir le gradué $\gr F$ parce que l'on utilise depuis
le début le foncteur fibre $\omega_{z_0} := \hat{\omega}_{z_0} \circ \gr$. Dans les
calculs qui vont suivre, les principaux morphismes utilisés seront $Z$, qui est
constant donc pur (de pente $0$); $G_{a,r,d}$, qui est pur par construction; et
$F_{a,r,d}$ qui est leur composé donc pur. On n'aura donc pas besoin de mentionner
explicitement $\gr$ et l'on posera:
$$
F_0 := F_{a,r,d}(z_{0,r}) \quad \text{et} \quad G_0 := G_{a,r,d}(z_{0,r}).
$$

Selon \cite{vdPR,VB}, un tel $\phi$ est déterminé par la donnée d'un unique triplet:
$$
(\lambda,h,\gamma) \in \C \times \Q^\vee \times \Eqvee,
$$
de telle sorte qu'avec les notations données au
\ref{subsubsection:objindecobjirred},
\begin{equation}
\label{eqn:phi(A)u}
\phi(U_m) = U_m^\lambda
\end{equation}
et
\begin{equation}
\label{eqn:phi(A)s}
\phi\left(E(r,d,c)\right) =
h(d/r) \, \gamma(a) \; G_0^{-1} \, \gamma(T_r) \, G_0(z_{0,r}) \;
\Diag\left(1,\gamma(q_r),\ldots,\gamma(q_r^{r-1})\right)^d.
\end{equation}
Nous résumerons les formules \eqref{eqn:phi(A)u} et \eqref{eqn:phi(A)s} par
la notation:
\begin{equation}
\label{eqn:notationphi<->triplet}
\phi \leftrightarrow (\lambda,h,\gamma),
\end{equation}
où $\phi \in \Gqp$ et $(\lambda,h,\gamma) \in \C \times \Q^\vee \times \Eqvee$.
La bijection de $\Gqp$ sur $\C \times \Q^\vee \times \Eqvee$ ainsi décrite n'est
pas un isomorphisme de groupes, ce point sera précisé aux numéros suivants
(\ref{subsubsection:multiplicationGp} et \ref{subsubsection:Gp=GpsxGpu}). \\

Pour rendre ces formules plus maniables, notant à nouveau $\zeta_r := e^{2 \ii \pi/r}$,
qui engendre $\mu_r$, nous remarquons que $\gamma(\zeta_r) \in \mu_r$ (car $\gamma$
est un morphisme de groupe) et que $\gamma(q_r) \in \mu_r$ (car en outre
$\gamma(q) = 1$). On peut donc écrire $\gamma(\zeta_r) = \zeta_r^k$ et
$\gamma(q_r) = \zeta_r^l$, $k,l \in \Z$. Notant pour simplifier $A := E(r,d,c)$,
on déduit alors de l'égalité ci-dessus et du formulaire du numéro précédent:
\begin{equation}
\label{eqn:phi(A)bis}
F_0 \phi(A) F_0^{-1} = h(d/r) \gamma(a) D_r^k T_r^{-ld}
\quad \text{et} \quad
G_0 \phi(A) G_0^{-1} = h(d/r) \gamma(a) T_r^k D_r^{ld}.
\end{equation}

\begin{rem}
L'effet de la ramification sur la non-commutativité se lit sur ces formules.
En effet, si $F$ et $G$ étaient des morphismes au dessus de $K$ et pas seulement
au dessus de $K_r$, on aurait:
$F_0 \phi(A) F_0^{-1} = \phi\left(F[A]\right) = \phi\left(a z^{d/r} D_r\right) =
h(d/r) \gamma(a) D_r^k$
et de même $G_0 \phi(A) G_0^{-1} = h(d/r) \gamma(a) T_r^k$.
\end{rem}

\paragraph{Deux éléments particuliers.}

On a prouvé dans \cite{JSGAL} que le sous-groupe de $\Eqvee$ formé des morphismes
\emph{continus} $\Eq \rightarrow \Cs$ (c'est donc une dorte de \og forme compacte \fg\
de $\Eqvee$) est Zariski-dense, et qu'il est abélien libre engendré par les éléments
$\gamma_1,\gamma_2 \in \Eqvee$ définis comme suit; si l'on a $\lmod u \rmod = 1$
et $y \in \R$, alors:
$$
\gamma_1(u q^y) = u \quad \text{et~} \quad \gamma_2(u q^y) = e^{2 \ii \pi y}.
$$
Les couples $(k,l) \in \Z \times \Z$ qui leur sont respectivement associés sont
$(k_1,l_1) = (1,0)$ et $(k_2,l_2) = (0,1)$. Notons $\phi_1,\phi_2 \in \Gqp$ les
éléments correspondants pour lesquels $\lambda = 0$ et $h = 1$ (morphisme trivial).
Alors:
\begin{align}
\label{eqn:effetgamma1gamma2}
F_0 \phi_1(A) F_0^{-1} = \gamma_1(a) D_r &\text{~~~~~et~~~~~} 
G_0 \phi_1(A) G_0^{-1} = \gamma_1(a) T_r, \\
F_0 \phi_2(A) F_0^{-1} = \gamma_2(a) T_r^{-d} &\text{~~~~~et~~~~~} 
G_0 \phi_2(A) G_0^{-1} = \gamma_2(a) D_r^{d}.
\end{align}

% 4.4.3

\subsubsection{Multiplication dans $\Gqp$}
\label{subsubsection:multiplicationGp}

Rappelons que, par définition, si $\phi,\phi' \in G$, leur produit $\phi \phi'$
est caractérisé par les relations $(\phi\phi')(A) := \phi(A) \phi'(A)$. \\

Soient $\phi \in \Gqp$ et $(\lambda,h,\gamma) \in \C \times \Q^\vee \times \Eqvee$
tels que, avec la notation \eqref{eqn:notationphi<->triplet},
$\phi \leftrightarrow (\lambda,h,\gamma)$. Soient $\phi_s,\phi_u \in \Gqp$ tels que
$\phi_s \leftrightarrow (0,h,\gamma)$ et $\phi_u \leftrightarrow (\lambda,1,1)$
(on note selon l'usage $1$ tout morphisme trivial d'un groupe dans $\C^*$).
Soit $A := E(r,d,c) \otimes U_m$. On déduit immédiatement de \eqref{eqn:phi(A)u}
et \eqref{eqn:phi(A)s} que $\phi(A) = \phi_s(A) \phi_u(A) = \phi_u(A) \phi_s(A)$.
On a donc:
$$
\phi = \phi_s \phi_u = \phi_u \phi_s.
$$
L'ensemble des $\phi_s \leftrightarrow (0,h,\gamma)$ sera noté $\Gqps$, celui
des $\phi_u \leftrightarrow (\lambda,1,1)$ sera noté $\Gqpu$. On notera alors
pour abréger:
\begin{equation}
\label{eqn:notationphi<->tripletbis}
\phi_s \leftrightarrow (h,\gamma) \quad \text{et} \quad
\phi_u \leftrightarrow \lambda.
\end{equation}

Soient $\phi,\phi' \in \Gqpu$ et $\lambda, \lambda' \in \C$ tels que, selon
la notation \eqref{eqn:notationphi<->tripletbis}, $\phi \leftrightarrow \lambda$
et $\phi' \leftrightarrow \lambda'$. Avec les mêmes notations, on voit que
$\phi(A) = I_r \otimes U_m^\lambda$, $\phi'(A) = I_r \otimes U_m^{\lambda'}$ et
$(\phi\phi')(A) = I_r \otimes U_m^{\lambda + \lambda'}$, \ie\ $\phi\phi' \in \Gqpu$
et $\phi\phi' \leftrightarrow \lambda + \lambda'$. Ainsi, $\Gqpu$ est un
sous-groupe de $\Gqp$ isomorphe à $\C$. \\

Soient maintenant $\phi,\phi' \in \Gqps$ et
$(h,\gamma), (h',\gamma') \in \Q^\vee \times \Eqvee$
tels que, selon la notation \eqref{eqn:notationphi<->tripletbis},
$\phi \leftrightarrow (h,\gamma)$ et $\phi' \leftrightarrow (h',\gamma')$.
L'effet de ces éléments sur le facteur droit $U_m$ du produit tensoriel
$E(r,d,c) \otimes U_m$ est de le remplacer par $U_m^0 = I_m$, on n'a donc pas
besoin d'en tenir compte dans le calcul qui suit et l'on prend $A := E(r,d,c)$.
Soient $t := h(1/r)$ et $t' := h'(1/r)$. On reprend les notations du numéro
précédent, et on introduit les couples $(k,l), (k',l') \in \Z \times \Z$
respectivement associés à $\gamma, \gamma' \in \Eqvee$. On trouve alors
selon \eqref{eqn:phi(A)bis}:
\begin{align*}
G_0 \phi(A) G_0^{-1} &= t^d \gamma(a) T_r^k D_r^{ld}, \\
G_0 \phi'(A) G_0^{-1} &= {t'}^d \gamma'(a) T_r^{k'} D_r^{l'd}, \\
G_0 (\phi\phi')(A) G_0^{-1}
&= (tt')^d (\gamma \gamma')(a) T_r^k D_r^{ld} T_r^{k'} D_r^{l'd} \\
&= (tt')^d (\gamma \gamma')(a) \zeta_r^{-k' l d} T_r^{k+k'} D_r^{(l+l')d} \\
&= u^d (\gamma \gamma')(a) T_r^{k+k'} D_r^{(l+l')d},
\end{align*}
où l'on a posé $u := t t' \zeta_r^{-k' l}$. On a invoqué au passage la troisième
formule du formulaire à la fin de \ref{subsubsection:objindecobjirred}, qui
implique en particulier que
$T_r^k D_r^{ld} T_r^{k'} D_r^{l'd} = \zeta_r^{-k' l d} T_r^{k+k'} D_r^{(l+l')d}$. Si l'on
trouve $h'' \in \Q^\vee$ tel que $h''(1/r) = u$, on en déduira que
$\phi \phi' \in \Gqps$ et que $\phi \phi' \leftrightarrow (h'',\gamma \gamma')$.
Or:
$$
\gamma'(\gamma(q_r)) = \gamma'\left(\zeta_r^{-l}\right) = \zeta_r^{-k' l}
\Longrightarrow
\zeta_r^{-k' l d} = \gamma'\left(\gamma\left(q^{-d/r}\right)\right).
$$
On est amené à introduire le morphisme $\epsilon(\gamma,\gamma') \in \Q^\vee$
défini par:
$$
\epsilon(\gamma,\gamma'):
\delta \mapsto \gamma'\left(\gamma\left(q^{-\delta}\right)\right).
$$
On a alors $h'' := h h' \epsilon(\gamma,\gamma') \in \Q^\vee$ et comme voulu
$\phi \phi' \leftrightarrow (h h' \epsilon(\gamma,\gamma'),\gamma \gamma')$.
En résumé:

\begin{prop}
\label{prop:multGp}
Via la bijection $\phi \leftrightarrow (\lambda,h,\gamma)$ de $\Gqp$ sur le
produit cartésien $\C \times \Q^\vee \times \Eqvee$, la multiplication de $\Gqp$
est donnée par la formule:
\begin{equation}
\label{eqn:multGp}
(\lambda,h,\gamma) \star (\lambda',h',\gamma') :=
(\lambda + \lambda',h h' \epsilon(\gamma,\gamma'),\gamma \gamma').
\end{equation}
\end{prop}

% 4.4.4

\subsubsection{Structure de $\Gqp$}
\label{subsubsection:Gp=GpsxGpu}

\paragraph{$\Gqp$ est un produit direct.}

Bien que $\Gqp$ ne soit plus un groupe abélien, on constate que l'on a encore
une décomposition de $\Gqp$ en produit direct $\Gqpu \times \Gqps$, où
$\Gqpu = \C$, qui agit sur les $q$-logarithmes; et l'on adjoint encore la
\og composante logarithmique \fg\ $\Gqpu$ au groupe de Stokes, en posant
$\Sttq := \Stq \times \Gqpu$, puis $\sttq := \Lie(\Sttq)$, d'où les relations:
$$
G = \Sttq \ltimes \Gqps = \sttq \ltimes \Gqps, \quad
\sttq = \stq \oplus \C \tau.
$$
(Voir \ref{subsubsection:correspondanceEE1<->Rep(Gq1)} pour la signification
de $\sttq \ltimes \Gqps$.) Ce qui va changer dans le cas des pentes arbitraires,
c'est la structure de $\Gqps$ et donc son action adjointe sur $\sttq$ (que l'on
abordera en \ref{subsection:groupedeGaloispentesarbitraires}).

\paragraph{$\Gqps$ est une extension centrale.}

Le groupe $\Gqps$ s'insère dans une suite exacte, qui est une extension
centrale non scindée:
$$
1 \rightarrow \Q^\vee \rightarrow \Gqps 
\rightarrow \Eqvee \rightarrow 1.
$$
Concrètement, on peut le décrire en définissant une loi de composition sur
le produit cartésien ensembliste $\Q^\vee \times \Eqvee$:
$$
(h,\gamma) \star (h',\gamma') := (\epsilon(\gamma,\gamma') h h', \gamma \gamma'),
$$
où $\epsilon: \Eqvee \times \Eqvee \rightarrow \Q^\vee$ est l'application 
définie par:
$$
\epsilon(\gamma,\gamma'):
\delta \mapsto \gamma'\left(\gamma\left(q^{-\delta}\right)\right).
$$
On a bien $\epsilon(\gamma,\gamma') \in \Q^\vee$. On vérifie de plus que l'application
$\epsilon: \Eqvee \times \Eqvee \rightarrow \Q^\vee$ est bilinéaire.

% 4.4.5

\subsubsection{Restriction à $\EE_q^r$ et ramification}
\label{subsubsection:RestrictionàE^retramification}

\paragraph{Le groupe $\Gqpsr$.}

Le lien avec les catégories intermédiaires $\EE_q^r$ (définies, pour tout
$r \in \N^*$, par la condition que les pentes sont dans $\frac{1}{r} \Z$) est 
le suivant. Le groupe $\Gqps$ est la limite projective des $\Gqpsr$. On identifie
alors $\Gqpsr$ à un quotient de $\Gqps$ via le diagramme de suites exactes:
\begin{equation}
\label{eqn:diagrammedesuitesexactes}
\xymatrix{
\EE_q & &
1 \ar@<0ex>[r] & \Q^\vee \ar@<0ex>[r] \ar@<0ex>[d] &
\Gqps \ar@<0ex>[r] \ar@<0ex>[d] & 
\Eqvee \ar@<0ex>[r] \ar@<0ex>[d] & 1 \\
\EE_q^r \ar@<0ex>[u] & &
1 \ar@<0ex>[r] & (\frac{1}{r} \Z)^\vee \ar@<0ex>[r] &
\Gqpsr \ar@<0ex>[r] & 
\Eqvee \ar@<0ex>[r] & 1
}
\end{equation}
La flèche $\Q^\vee \rightarrow (\frac{1}{r} \Z)^\vee$ est duale de l'inclusion
$\frac{1}{r} \Z \rightarrow \Q$ (et donc surjective); la flèche 
$\Eqvee \rightarrow \Eqvee$ est l'égalité; la surjection 
$\Gqps \rightarrow \Gqpsr$ s'en déduit. \\

D'autre part, $(\frac{1}{r} \Z)^\vee$ s'identifie à $\Cs$ par l'isomorphisme
$h \mapsto h(1/r)$, d'où une description alternative de la loi de groupe
sur $\Gqpsr$ donnant lieu à une extension:
$$
1 \rightarrow \Cs \rightarrow \Gqpsr \rightarrow \Eqvee \rightarrow 1.
$$
Soit $\phi \in \Gqpsr$; si $\phi \leftrightarrow (h,\gamma)$,
on écrira $\phi \leftrightarrow (t,\gamma)$ avec $t := h(1/r)$. Alors:
$$
\left(\phi \leftrightarrow (t,\gamma) \text{~et~}
\phi' \leftrightarrow (t',\gamma')\right)
\Longrightarrow \phi \phi' \leftrightarrow (t t' \eta(\gamma,\gamma'),\gamma \gamma')
\text{~où~}
\eta(\gamma,\gamma') := \gamma'\left(\gamma\left(q^{-1/r}\right)\right) \in \Cs.
$$

\paragraph{Description de $\Gqpsr$ par ramification.}

On obtient une description approchée de $\Gqpsr$ à l'aide d'une ramification.
Le foncteur de ramification $\text{Ram}_r: \EE_q^r \leadsto \EE_{q_r}^1$ donne
par dualité tannakienne un diagramme commutatif de suites exactes:
$$
\xymatrix{
\EE_{q_r}^1 & &
1 \ar@<0ex>[r] & \Z^\vee \ar@<0ex>[r] \ar@<0ex>[d] &
\mathbf{G}_{q_r,p,s}^r \ar@<0ex>[r] \ar@<0ex>[d] & 
\mathbf{E}_{q_r}^\vee \ar@<0ex>[r] \ar@<0ex>[d] & 1 \\
\EE_q^r \ar@<0ex>[u] & &
1 \ar@<0ex>[r] & (\frac{1}{r} \Z)^\vee \ar@<0ex>[r] &
\Gqpsr \ar@<0ex>[r] & 
\mathbf{E}_q^\vee \ar@<0ex>[r] & 1
}
$$
La flèche verticale gauche est l'isomorphisme dual de celui de multiplication
par $r$ de $\frac{1}{r} \Z$ dans $\Z$; la flèche verticale centrale est une 
immersion fermée ; la flèche verticale droite est l'immersion fermée duale de
l'isogénie de degré $r$ de $\Eq$ sur $\mathbf{E}_{q_r}$ induite par l'inclusion
$q^\Z \subset {q_r}^\Z$. On déduit en particulier de ce diagramme que l'injection
de $\mathbf{G}_{q_r,p,s}^1$ dans $\Gqpsr$ a pour conoyau $\mu_r$.

% 4.5

\subsection{Le groupe de Galois local dans le cas de pentes arbitraires}
\label{subsection:groupedeGaloispentesarbitraires}

% 4.5.1

\subsubsection{Le groupe de Galois de $\EE_q^r$ et le groupe fondamental sauvage }
\label{subsubsection:groupefondamentalsauvage}

Notre but est ici de décrire le groupe de Galois $\Gqr$ de $\EE_q^r$ à partir de
son \emph{groupe fondamental sauvage}, un sous-groupe Zariski-dense explicitement
déterminé par générateurs et relations. De plus, ce dernier sera lui-même incarné
par un objet hybride de la forme $L \ltimes \Gamma$, où $L$ est une sous-algèbre
de Lie graduée libre du logarithme de la composante unipotente de $\Gqr$ et où
$\Gamma$ un sous-groupe de la composante semi-simple du facteur formel, \ie\ pur,
de $\Gqr$, tels que $L$ soit stable sous l'action de $\Gamma$. \\

Nous utiliserons pour cela le foncteur de ramification
$\text{Ram}_r: \EE_q^r \leadsto \EE_{q_r}^1$ qui induit une immersion fermée
$\mathbf{G}_{q_r}^1 \hookrightarrow \Gqr$ (\cf\ \ref{subsubsection:immersionsfermées}),
laquelle se restreint en des immersions fermées des composantes de Stokes et des
composantes pures. Pour éviter la multiplication des indices et exposants, nous
adaptons certaines des conventions de
\ref{subsubsection:suffisancedesstokesnonramifiés}: on note sans, resp. avec prime
les objets attachés à $\EE_q^r$, resp. à $\EE_{q_r}^1$. Ainsi $G := \mathbf{G}_{q}^r$
et $G' := \mathbf{G}_{q_r}^1$; de même $\St$, $\St'$ désignent respectivement leurs
sous-groupes de Stokes et $G_p$, $G_p'$ leurs composantes pures. On a vu en
\ref{subsubsection:suffisancedesstokesnonramifiés} le diagramme commutatif:
$$
\xymatrix{
1 \ar@<0ex>[r] & 
\St' \ar@<0ex>[d] \ar@<0ex>[r] & 
G' \ar@<0ex>[d] \ar@<0ex>[r] & 
G_p' \ar@<0ex>[d] \ar@<0ex>[r] & 
1 \\
1 \ar@<0ex>[r] & 
\St \ar@<0ex>[r] & 
G \ar@<0ex>[r] & 
G_p  \ar@<0ex>[r] & 
1 
}
$$
qui permet d'identifier $\St'$ à un sous-groupe de $\St$, donc de $G$. On a surtout
vu que $\St'$ et $G_p$ engendrent ensemble un sous-groupe Zariski-dense du groupe $G$
que nous cherchons à décrire. \\

On sait d'autre part que dans les décompositions en parties unipotentes et semi-simples
$G_p = G_{p,u} \times G_{p,s}$ et $G'_p = G'_{p,u} \times G'_{p,s}$, l'inclusion de $G'$
dans $G$ induit l'identité $G'_{p,u} \rightarrow G_{p,u}$ et que cette composante commute
avec $\St'$. On voit donc que, notant $\st' := \Lie(\St')$ et, comme précédemment,
$\C \tau = \Lie(G_{p,u}) = \Lie(G'_{p,u})$, l'algèbre de Lie $\stt' := \st' \oplus \C \tau$
(via son exponentielle) et $G_{p,s}$ engendrent topologiquement $G$. \\

L'action par conjugaison de $G_{p,s}$ sur $\St$ induit son action adjointe sur $\stt$.
Nous allons montrer que cette dernière laisse stable $\stt'$: on pourra donc identifier
le groupe de Galois $G$ de $\EE_q^r$ à l'objet hybride $\stt' \ltimes G_{p,s}$. Pour cela,
il suffit de se restreindre à la sous-algèbre de Lie graduée $L'$ de $\stt'$ engendrée
par $\tau$ et par les $q$-dérivées étrangères\footnote{On prendra garde à ce que nous avons
introduit les notations $\Delta_\alpha$, etc, pour des objets définis sur la catégorie
$\EE_q^1$ (donc relatifs à la \emph{base} $q$) mais que nous reprenons ici ces notations
pour des objets définis sur la catégorie $\EE_{q_r}^1$ (donc relatifs à la \emph{base}
$q_r$).} $\Delta_\alpha^{(\delta,\overline{c})}$, $\alpha, \overline{c} \in \mathbf{E}_{q_r}$. \\

On invoque \ref{subsection:groupedeGaloisformel}. Puisque $G_p = G_{p,u} \times G_{p,s}$,
l'action par conjugaison du groupe $G_{p,s}$ sur le facteur $\C \tau$ de $\stt'$ est
triviale et il suffit de décrire son action sur les $\Delta \in \stt'$, et même plus
particulièrement sur les $q$-dérivées étrangères $\Delta_\alpha^{(\delta,\beta)}$,
$\delta \in \N^*$, $\alpha, \beta \in \Eqr$. \\

Soit donc $\phi \in G_{p,s}$, $\phi \leftrightarrow (h,\gamma)$. En vertu de la suite
exacte écrite à la ligne inférieure du diagramme \eqref{eqn:diagrammedesuitesexactes}
de \ref{subsubsection:RestrictionàE^retramification}, nous traiterons d'abord le cas 
$\phi \leftrightarrow (h,1)$ au \ref{subsubsection:actiondeh} (action du \og tore
theta \fg\ $(\frac{1}{r} \Z)^\vee$); puis le cas $\phi \leftrightarrow (1,\gamma)$, où
$\gamma \in \mathbf{E}_q^\vee$ au numéro \ref{subsubsection:actiondegamma}.
Le formulaire complet sera résumé au \ref{subsubsection:actiondeGpsurSt'}.

% 4.5.2

\subsubsection{Action du tore theta $(\frac{1}{r} \Z)^\vee$}
\label{subsubsection:actiondeh}

Soient donc $A$ un objet de $\EE_q^r$ , que l'on peut supposer à deux pentes d'après
\ref{subsubsection:linéarisation}. D'après \ref{subsubsection:objindecobjirred}, et
le lemme \ref{lem:dévissageparblocs} de \ref{subsubsection:linéarisation}, on peut
prendre $A$ de la forme:
$$
A = \begin{pmatrix} E(r_1,d_1,c_1) \otimes U_{m_1} & U \\
0 & E(r_2,d_2,c_2) \otimes U_{m_2} \end{pmatrix}.
$$
Les tailles des blocs diagonaux de $A$ sont $n_1 := r_1 m_1$ et $n_2 := r_2 m_2$.
Les pentes de $A$ calculées dans $\EE_q$ sont $\mu_1 := d_1/r_1$ et $\mu_2 :=d_2/r_2$.
Comme c'est un objet de $\EE_q^r$, on a $r = r_1 s_1 = r_2 s_2$, $s_1,s_2 \in \N^*$ et
les pentes de l'objet $A'(z_r) := A(z_r^r)$ de $\EE_{q_r}^1$ sont $r \mu_1 = d_1 s_1$ et
$r \mu_2 = d_2 s_2$, de différence
$\delta := d_2 s_2 - d_1 s_1 = r(\mu_2 - \mu_1) \in \N^*$. \\

Soit d'autre part $\Delta_\alpha^{(\delta,\beta)} \in L'$, où $\delta$ vient d'être
calculé (d'autres valeurs donneraient un résultat trivial), $\alpha,\beta \in \Eqr$.
On a:
$$
\Delta_\alpha^{(\delta,\beta)}(A) = \begin{pmatrix} 0 & N \\ 0 & 0 \end{pmatrix}
$$
pour une certaine matrice rectangulaire complexe $N \in \Mat_{n_1,n_2}(\C)$. \\

Soit $\phi \leftrightarrow (h,1)$, $h \in (\frac{1}{r} \Z)^\vee$. L'image d'un bloc pur
de pente $\mu \in \frac{1}{r} \Z$ est le scalaire $h(\mu)$. L'image du bloc
$E(r_i,d_i,c_i) \otimes U_{m_i}$ de taille $n_i$, $i = 1,2$, est donc le bloc scalaire
$h(d_i/r_i) I_{n_i}$, d'où $\phi(A) = \phi(A_0) = \Diag(h(d_1/r_1) I_{n_1},h(d_2/r_2) I_{n_2})$
et: \begin{align*}
\phi(A)^{-1} \Delta_\alpha^{(\delta,\beta)}(A) (\phi(A)) &=
\begin{pmatrix} h(d_1/r_1) I_{n_1} & 0 \\ 0 & h(d_2/r_2) I_{n_2} \end{pmatrix}^{-1}
\begin{pmatrix} 0 & N \\ 0 & 0 \end{pmatrix}
\begin{pmatrix} h(d_1/r_1) I_{n_1} & 0 \\ 0 & h(d_2/r_2) I_{n_2} \end{pmatrix} \\
&= \begin{pmatrix} 0 & \dfrac{h(d_2/r_2)}{h(d_1/r_1)} N \\ 0 & 0 \end{pmatrix} 
= h(d_2/r_2 - d_1/r_1) \Delta_\alpha^{(\delta,\beta)}(A) 
= h(\delta/r) \Delta_\alpha^{(\delta,\beta)}(A).
\end{align*}

L'action de $\phi \leftrightarrow (h,1)$ sur la composante de degré $\delta \in \N^*$
de $L'$ est donc scalaire de facteur $h(\delta/r)$. (Rappelons que $\delta$ est à
la fois le degré de l'élément étudié dans l'algèbre de Lie de Stokes et le niveau
$q$-Gevrey calculé dans $\EE_{q_r}$, mais que le niveau $q$-Gevrey calculé dans $\EE_q$
est $\delta/r$.)

% 4.5.3

\subsubsection{Action de $\mathbf{E}_q^\vee$}
\label{subsubsection:actiondegamma}

On garde les mêmes notations qu'en \ref{subsubsection:actiondeh} pour $A$ et pour
$\Delta_\alpha^{(\delta,\beta)}$.
Soit $\phi \leftrightarrow (1,\gamma)$, $\gamma \in \mathbf{E}_q^\vee$. On pourra
se restreindre au cas où $\gamma$ est l'un des générateurs topologiques $\gamma_1$
et $\gamma_2$ de $\mathbf{E}_q^\vee$. On notera donc $\phi_1 \leftrightarrow (1,\gamma_1)$
et $\phi_2 \leftrightarrow (1,\gamma_2)$. Dans tous les cas, on a $\phi(U_m) = I_m$ et
$\phi(A) = \phi(A_0) = \Diag(\phi(E_1) \otimes I_{m_1},\phi(E_2) \otimes I_{m_2})$, où
l'on a abrégé $E_i := E(r_i,d_i,c_i)$, $i = 1,2$. Ainsi, chacun des blocs de $\phi(A)$ 
est somme directe de blocs $\phi(E_i)$. \\

De la proposition \ref{prop:casgénériqueàdeuxpentes} on déduit que le calcul de
$\Delta_\alpha^{(\delta,\beta)}(A)$ donne le même résultat si l'on remplace
les facteurs unipotents $U_{m_i}$ par leur partie semi-simple triviale $I_{m_i}$.
Mais les blocs diagonaux sont alors remplacés par les $E_i \otimes I_{m_i}$, des
sommes directes de blocs $E_i$. Appliquant le lemme \ref{lem:dévissageparblocs}
de \ref{subsubsection:linéarisation}, on peut à nouveau prendre $A$ de la forme:
$$
A = \begin{pmatrix} E(r_1,d_1,c_1) \otimes U_{m_1} & U \\
0 & E(r_2,d_2,c_2) \otimes U_{m_2} \end{pmatrix}.
$$

On applique maintenant les calculs de \ref{subsubsection:objindecobjirred}; notant
$F := \begin{pmatrix} F_{a_1,r_1,d_1} & 0 \\ 0 & F_{a_2,r_2,d_2} \end{pmatrix}$, où $a_i$
est une racine $r_i$\ieme\ arbitraire de $c_i$, $i = 1,2$, on a un isomorphisme dans
$\EE_{q_r}$:
$$
\xymatrix{
{A = \begin{pmatrix} E(r_1,d_1,c_1) & U \\ 0 & E(r_2,d_2,c_2) \end{pmatrix}}
\ar@<0ex>[rr]^F & &
{B := \begin{pmatrix} a_1 z^{d_1/r_1} D_{r_1} & V \\ 0 & a_2 z^{d_2/r_2} D_{r_2} \end{pmatrix}}
},
$$
que l'on peut préciser ainsi; $z^{d_1/r_1} = z_r^{d_1 s_1}$, $z^{d_2/r_2} = z_r^{d_2 s_2}$ et,
puisque $z_{r_1} = z_r^{s_1}$ et $z_{r_2} = z_r^{s_2}$:
\begin{align*}
F = F(z_r) &=
\begin{pmatrix} F_{a_1,r_1,d_1}(z_r^{s_1}) & 0 \\ 0 & F_{a_2,r_2,d_2}(z_r^{s_2}) \end{pmatrix}, \\
D_{r_1} = \begin{pmatrix} 1 & 0 & \ldots & 0 \\ 0 & \zeta_{r_1} & \ldots & 0 \\
  \vdots & \vdots & \ddots & \vdots \\ 0 & 0 & \ldots & \zeta_{r_1}^{r_1 - 1} \end{pmatrix}
& = \begin{pmatrix} 1 & 0 & \ldots & 0 \\ 0 & \zeta_{r}^{s_1} & \ldots & 0 \\
  \vdots & \vdots & \ddots & \vdots \\ 0 & 0 & \ldots & \zeta_{r}^{s_1(r_1 - 1)} \end{pmatrix} \\
D_{r_2} = \begin{pmatrix} 1 & 0 & \ldots & 0 \\ 0 & \zeta_{r_2} & \ldots & 0 \\
  \vdots & \vdots & \ddots & \vdots \\ 0 & 0 & \ldots & \zeta_{r_2}^{r_2 - 1} \end{pmatrix}
& = \begin{pmatrix} 1 & 0 & \ldots & 0 \\ 0 & \zeta_{r}^{s_2} & \ldots & 0 \\
  \vdots & \vdots & \ddots & \vdots \\ 0 & 0 & \ldots & \zeta_{r}^{s_2(r_2 - 1)} \end{pmatrix}
\end{align*}
On a vu plus haut (au début de \ref{subsubsection:élémentsdeGp}) que les $F_{a,d,r}$
et $G_{a,d,r}$ sont leurs propres gradués. La fonctorialité de $\Delta_\alpha^{(\delta,\beta)}$
\emph{dans la catégorie $\EE_{q_r}$} entraîne alors la formule:
$$
\Delta_\alpha^{(\delta,\beta)}(B) = \Delta_\alpha^{(\delta,\beta)}\left(F[A]\right) =
F_0 \Delta_\alpha^{(\delta,\beta)}(A) F_0^{-1},
\text{~où l'on a posé:~}
F_0 = F(z_{0,r}) = \begin{pmatrix} F_{a_1,r_1,d_1}(z_{0,r}) & 0 \\
  0 & F_{a_2,r_2,d_2}(z_{0,r}) \end{pmatrix}.
$$
Pour préciser le calcul, on écrira:
$$
\Delta_\alpha^{(\delta,\beta)}(B) = \begin{pmatrix} 0 & N \\ 0 & 0 \end{pmatrix},
\quad \text{où} \quad
N = (n_{i,j})_{1 \leq i \leq r_1 \atop 1 \leq j \leq r_2} \in \Mat_{r_1,r_2}(\C).
$$

\paragraph{Action de $\gamma_1 \in \mathbf{E}_q^\vee$}

En vertu des équations \eqref{eqn:effetgamma1gamma2} de la fin du numéro
\ref{subsubsection:élémentsdeGp}:
$$
F_0 \phi_1(A) F_0^{-1} =
\begin{pmatrix} \gamma_1(a_1) D_{r_1} & 0 \\ 0 & \gamma_1(a_2) D_{r_2} \end{pmatrix}.
$$
Ainsi:
\begin{align*}
F_0 \left(\phi_1(A)^{-1} \Delta_\alpha^{(\delta,\beta)}(A) \phi_1(A)\right) F_0^{-1}
&=
(F_0 \phi_1(A) F_0^{-1})^{-1}
(F_0 \Delta_\alpha^{(\delta,\beta)}(A) F_0^{-1})
(F_0 \phi_1(A) F_0^{-1}) \\
&=
(F_0 \phi_1(A) F_0^{-1})^{-1} \Delta_\alpha^{(\delta,\beta)}(B) (F_0 \phi_1(A) F_0^{-1}) \\
&=
\begin{pmatrix} \gamma_1(a_1) D_{r_1} & 0 \\ 0 & \gamma_1(a_2) D_{r_2} \end{pmatrix}^{-1}
\begin{pmatrix} 0 & N \\ 0 & 0 \end{pmatrix}
\begin{pmatrix} \gamma_1(a_1) D_{r_1} & 0 \\ 0 & \gamma_1(a_2) D_{r_2} \end{pmatrix} \\
& =\begin{pmatrix} 0 & \left(\gamma_1(a_1) D_{r_1}\right)^{-1} N
\left(\gamma_1(a_2) D_{r_2}\right)
\\ 0 & 0 \end{pmatrix} \\
& =\begin{pmatrix} 0 & \gamma_1(a_2/a_1) D_{r_1}^{-1} N D_{r_2} \\ 0 & 0 \end{pmatrix}.
\end{align*}

La matrice $N' := \gamma_1(a_2/a_1) D_{r_1}^{-1} N D_{r_2}$ s'écrit
$(n'_{i,j})_{1 \leq i \leq r_1 \atop 1 \leq j \leq r_2}$ , où:
$$
n'_{i,j} = \gamma_1(a_2/a_1) \zeta_r^{j s_2 - i s_1} =
\gamma_1\left(\dfrac{a_2 \zeta_{r_2}^j}{a_1 \zeta_{r_1}^i}\right)
$$
car $\gamma_1(u) = u$ pour toute racine de l'unité $u$. Or, notant $c \in \C^*$
un représentant de $\beta \in \Eqr$, les seuls coefficients non nuls de $N$, bloc
surdiagonal de $\Delta_\alpha^{(\delta,\beta)}(B) = \Delta_\alpha^{(\delta,\overline{c})}(B)$
sont ceux tels que $\dfrac{a_2 \zeta_{r_2}^j}{a_1 \zeta_{r_1}^i} \equiv c \pmod{q_r^\Z}$,
et l'on vient de voir qu'ils sont multipliés par $\gamma_1(c)$. On a donc:
$$
F_0 \left(\phi_1(A)^{-1} \Delta_\alpha^{(\delta,\overline{c})}(A) \phi_1(A)\right) F_0^{-1} =
\gamma_1(c) (F_0 \Delta_\alpha^{(\delta,\overline{c})}(A) F_0^{-1})
\Longrightarrow
\phi_1(A)^{-1} \Delta_\alpha^{(\delta,\overline{c})}(A) \phi_1(A) =
\gamma_1(c) \Delta_\alpha^{(\delta,\overline{c})}(A).
$$
À noter qu'en principe l'élément $c$ ne devrait intervenir qu'à travers sa classe
$\beta = \overline{c}$ modulo $q_r^\Z$, mais comme $\gamma_1(q_r) = \gamma_1(q^{1/r}) = 1$
par définition de $\gamma_1$, le résultat ci-dessus est cohérent. En fait, $\gamma_1(\beta)$
est bien défini et l'on peut même écrire:
$$
\phi_1(A)^{-1} \, \Delta_\alpha^{(\delta,\beta)}(A) \, \phi_1(A) =
\gamma_1(\beta) \, \Delta_\alpha^{(\delta,\beta)}(A).
$$
En conclusion: l'action par conjugaison de $\phi_1$ sur $\Delta_\alpha^{(\delta,\beta)}$
est scalaire, c'est la multiplication par $\gamma_1(\beta)$:
\begin{equation}
\label{eqn:actiongamma1}
\phi_1^{-1} \, \Delta_\alpha^{(\delta,\beta)} \, \phi_1 =
\gamma_1(\beta) \, \Delta_\alpha^{(\delta,\beta)}.
\end{equation}

\begin{rem}
Dans le cas des pentes entières, on avait vu que l'action de tout
$\gamma \in \mathbf{E}_q^\vee$ sur $\Delta_\alpha^{(\delta,\overline{c})}$ était la
multiplication par $\gamma(c)$ et celle de tout $h \in \Z^\vee$ la multiplication
par $h(\delta)$  (décomposition de Fourier pour le groupe semi-simple
$\mathbf{G}_{q,p,s}^1$). Ces deux faits restent vrais ici; le changement principal
viendra de $\gamma_2$ (ci-dessous).
\end{rem}

\paragraph{Action de $\gamma_2 \in \mathbf{E}_q^\vee$}

Toujours en vertu des équations \eqref{eqn:effetgamma1gamma2} de la fin du numéro
\ref{subsubsection:élémentsdeGp}:
$$
F_0 \phi_2(A) F_0^{-1} = 
\begin{pmatrix} \gamma_2(a_1) T_{r_1}^{-d_1} & 0 \\
  0 & \gamma_2(a_2) T_{r_2}^{-d_2} \end{pmatrix}.
$$
Ainsi:
\begin{align*}
F_0 \left(\phi_2(A)^{-1} \Delta_\alpha^{(\delta,\beta)}(A) \phi_2(A)\right) F_0^{-1}
&=
(F_0 \phi_2(A) F_0^{-1})^{-1}
(F_0 \Delta_\alpha^{(\delta,\beta)}(A) F_0^{-1})
(F_0 \phi_2(A) F_0^{-1}) \\
&=
(F_0 \phi_2(A) F_0^{-1})^{-1} \Delta_\alpha^{(\delta,\beta)}(B) (F_0 \phi_2(A) F_0^{-1}) \\
&=
\begin{pmatrix} \gamma_2(a_1) T_{r_1}^{-d_1} & 0 \\
  0 & \gamma_2(a_2) T_{r_2}^{-d_2} \end{pmatrix}^{-1}
\begin{pmatrix} 0 & N \\ 0 & 0 \end{pmatrix}
\begin{pmatrix} \gamma_2(a_1) T_{r_1}^{-d_1} & 0 \\
  0 & \gamma_2(a_2) T_{r_2}^{-d_2} \end{pmatrix} \\
& =\begin{pmatrix} 0 & \left(\gamma_2(a_1) T_{r_1}^{-d_1}\right)^{-1} N
\left(\gamma_2(a_2) T_{r_2}^{-d_2}\right)
\\ 0 & 0 \end{pmatrix} \\
& = \gamma_2(a_2/a_1) \begin{pmatrix} 0 & N'' \\ 0 & 0 \end{pmatrix},
\end{align*}
où l'on a posé $N'' := T_{r_1}^{d_1} N T_{r_2}^{-d_2}$. Pour affiner ce calcul, on note
$c \in \C^*$ un représetant de $\beta \in \Eqr$. Les seuls cas de non nullité de
$\Delta_\alpha^{(\delta,\beta)}(B)$ sont ceux où
$c \equiv \frac{a_2 \zeta_r^j}{a_1 \zeta_r^i} \pmod{q_r^\Z}$.
Comme $\gamma_2(\zeta_r) = 1$ (par définition de $\gamma_2)$, on a donc en fait:
$$
F_0 \left(\phi_2(A)^{-1} \Delta_\alpha^{(\delta,\overline{c})}(A) \phi_2(A)\right) F_0^{-1} =
\gamma_2(c) \begin{pmatrix} 0 & N'' \\ 0 & 0 \end{pmatrix},
$$
ou encore:
$$
F_0 \left(\phi_2(A)^{-1} \Delta_\alpha^{(\delta,\beta)}(A) \phi_2(A)\right) F_0^{-1} =
\gamma_2(\beta) \begin{pmatrix} 0 & N'' \\ 0 & 0 \end{pmatrix},
$$
Pour comprendre cette action, nous aurons l'usage d'un lemme. À cet effet,
nous introduisons (à nouveau) l'automorphisme $\tau: z_r \mapsto \zeta_r z_r$,
générateur du groupe de Galois de l'extension cyclique $K_r/K$, qui commute
à $\sq$. Rappelons que
$B := F[A] = \begin{pmatrix} a_1 z^{d_1/r_1} D_{r_1} & V \\
  0 & a_2 z^{d_2/r_2} D_{r_2} \end{pmatrix}$.

\begin{lem}
Soit $T := \Diag(T_{r_1}^{d_1},T_{r_2}^{d_2})$. Alors $\tau B = T B T^{-1}$.
\end{lem}
\begin{proof}
On a bien entendu $\tau A = A$; de plus $\tau F = T F$ en vertu des formules
du \ref{subsubsection:objindecobjirred}. On en déduit:
\begin{align*}
\tau B &= \tau((\sq F) A F^{-1}) = (\tau \sq F) (\tau A) (\tau F)^{-1} =
(\sq \tau F) (\tau A) (\tau F)^{-1} \\
&= (\sq (T F)) A (T F)^{-1} =
T ((\sq F) A F^{-1}) T^{-1} = T B T^{-1}.
\end{align*}
\end{proof}

Par fonctorialité, on en déduit:
$$
\Delta_\alpha^{(\delta,\beta)}(\tau B) =
T \Delta_\alpha^{(\delta,\beta)}(B) T^{-1} =
\begin{pmatrix} 0 & T_{r_1}^{d_1} N T_{r_2}^{-d_2} \\ 0 & 0 & \end{pmatrix} =
\begin{pmatrix} 0 & N'' \\ 0 & 0 & \end{pmatrix},
$$
avec la notation introduite plus haut. On a donc:
$$
F_0 \left(\phi_2(A)^{-1} \Delta_\alpha^{(\delta,\beta)}(A) \phi_2(A)\right) F_0^{-1}
= \gamma_2(\beta) \Delta_\alpha^{(\delta,\beta)}(\tau B).
$$

Par ailleurs le calcul de $\Delta_\alpha^{(\delta,\beta)}(\tau B)$ relève du
corollaire \ref {cor:Delta(A(lambdaz))} de \ref{subsubsection:dilatation},
que l'on applique avec $\lambda := \zeta_r$ (et bien entendu $q_r$ à la place
de $q$). En utilisant la remarque qui suit le corollaire et le fait que
$\gamma_2(q_r) = \zeta_r$, on trouve, pour les seules valeurs non nulles,
donc pour toutes (notation: $\alpha = \overline{c}$, $\beta = \overline{d^{-1}}$):
$$
\Delta_{\overline{\zeta_r}^{-1} \alpha}^{(\delta,\overline{\zeta_r}^\delta \beta)}(\tau B) = 
\left(\dfrac{\theta_{q_r}(z_0/c)}{\theta_{q_r}(z_0/(\zeta_r^{-1} c))}\right)^\delta
\gamma_2(d/c^\delta) \; \Delta_{\alpha}^{(\delta,\beta)}(B).
$$
Pour se rapprocher des notations en cours dans ce numéro, on remplace d'abord
dans la formule ci-dessus $c$ par $\zeta_r c$ et $d$ par $\zeta_r^{\delta} d$
(la valeur de $d/c^\delta$ n'est donc pas modifiée):
$$
\Delta_{\alpha}^{(\delta,\beta)}(\tau B) = 
\left(\dfrac{\theta_{q_r}(z_0/\zeta_r c)}{\theta_{q_r}(z_0/c)}\right)^\delta
\gamma_2(d/c^\delta) \; \Delta_{\overline{\zeta_r} \alpha}^{(\delta,\overline{\zeta_r}^{-\delta} \beta)}(B);
$$
puis on remplace $c$ par $a$ et $d$ par $c$ (maintenant $\alpha = \overline{a}$
et $\beta = \overline{c^{-1}}$):
$$
\Delta_{\alpha}^{(\delta,\beta)}(\tau B) = 
\left(\dfrac{\theta_{q_r}(z_0/\zeta_r a)}{\theta_{q_r}(z_0/a)}\right)^\delta
\gamma_2(c/a^\delta) \; \Delta_{\overline{\zeta_r} \alpha}^{(\delta,\overline{\zeta_r}^{-\delta} \beta)}(B).
$$
On a alors:
\begin{align*}
F_0 \left(\phi_2(A)^{-1} \Delta_\alpha^{(\delta,\beta)}(A) \phi_2(A)\right) F_0^{-1}
&= \gamma_2(\beta) \Delta_\alpha^{(\delta,\beta)}(\tau B) \\
&= \gamma_2(c^{-1})
\left(\dfrac{\theta_{q_r}(z_0/\zeta_r a)}{\theta_{q_r}(z_0/a)}\right)^\delta
\gamma_2(c/a^\delta) \;
\Delta_{\overline{\zeta_r} \alpha}^{(\delta,\overline{\zeta_r}^{-\delta} \beta)}(B) \\
&= \gamma_2(a^{-\delta})
\left(\dfrac{\theta_{q_r}(z_0/\zeta_r a)}{\theta_{q_r}(z_0/a)}\right)^\delta
F_0 \, \Delta_{\overline{\zeta_r} \alpha}^{(\delta,\overline{\zeta_r}^{-\delta} \beta)}(A) \, F_0^{-1}.
\end{align*}
En se débarrassant de la conjugaison par $F_0$, il vient la formule recherchée:
$$
\phi_2(A)^{-1} \, \Delta_\alpha^{(\delta,\beta)}(A) \, \phi_2(A) =
\left(\overline{\gamma}_2(\alpha)\right)^\delta \;
\Delta_{\overline{\zeta_r} \alpha}^{(\delta,\overline{\zeta_r}^{-\delta} \beta)}(A),
$$
où l'on a posé la définition suivante:

\begin{defn}
\label{defn:nouveaumistigri}
On définit la fonction $\overline{\gamma}_2: \Eqr \rightarrow \Cs$ par la formule:
$$
\overline{\gamma}_2(\alpha) :=
\gamma_2(c^{-1}) \, \dfrac{\theta_{q_r}(z_0/\zeta_r c)}{\theta_{q_r}(z_0/c)}
$$
où $c \in \Cs$ est un représentant arbitraire de $\alpha$.
\end{defn}

On vérifie en effet immédiatement qu'en vertu de l'équation fonctionnelle de $\theta_{q_r}$
et de l'égalité $\gamma_2(q_r) = \zeta_r$, le membre de droite de cette égalité n'est pas
affecté lorsque l'on remplace $c$ par $q_r c$. En résumé:

\begin{equation}
\label{eqn:actiongamma2}
\phi_2^{-1} \, \Delta_\alpha^{(\delta,\beta)} \, \phi_2 =
\left(\overline{\gamma}_2(\alpha)\right)^\delta \;
\Delta_{\overline{\zeta_r} \alpha}^{(\delta,\overline{\zeta_r}^{-\delta} \beta)}.
\end{equation}

% 4.5.4

\subsubsection{Formulaire de l'action de $G_p$ sur $\St'$}
\label{subsubsection:actiondeGpsurSt'}

Résumons les résultats du \ref{subsection:groupedeGaloispentesarbitraires}. Si
$\Delta \in \stt$ (donc en particulier si $\Delta \in \stt'$) et si $\phi \in G$,
nous écrirons $\Delta^\phi := \phi \Delta \phi^{-1}$. De plus, pour simplifier les
notations, nous identifierons à l'élément $h \in (\frac{1}{r} \Z)^\vee$ l'élément
$\phi \in G_{p,s}$ tel que $\phi \leftrightarrow (h,1)$; et de même nous identifierons
à l'élément $\gamma \in \mathbf{E}_q^\vee$ l'élément $\phi \in G_{p,s}$ tel que
$\phi \leftrightarrow (1,\gamma)$. Rappelons enfin que, si $\alpha = \overline{c} \in \Eqr$
on a introduit à la fin du \ref{subsubsection:actiondegamma} (définition
\ref{defn:nouveaumistigri}):
$$
\overline{\gamma}_2(\alpha) :=
\gamma_2(c^{-1}) \, \dfrac{\theta_{q_r}(z_0/\zeta_r c)}{\theta_{q_r}(z_0/c)}.
$$

\begin{thm}
\label{thm:conjuguésdérivéesétrangères}
Soient $\delta \in \N^*$ et $\alpha,\beta \in \mathbf{E}_{q_r}$. Alors: \\
(i) $\forall h \in (\frac{1}{r} \Z)^\vee \;,\;
\left(\Delta_\alpha^{(\delta,\beta)}\right)^h =
h(\delta/r) \Delta_\alpha^{(\delta,\beta)}$. \\
(ii) $\left(\Delta_\alpha^{(\delta,\beta)}\right)^{\gamma_1} = 
\gamma_1(\beta) \; \Delta_\alpha^{(\delta,\beta)}$. \\
(iii) $\left(\Delta_\alpha^{(\delta,\beta)}\right)^{\gamma_2} =
\left(\overline{\gamma}_2(\alpha)\right)^\delta \;
\Delta_{\overline{\zeta_r} \alpha}^{(\delta,\overline{\zeta_r}^{-\delta} \beta)}$.
\end{thm}
\begin{proof}
Elle a été donnée tout au long du \ref{subsubsection:actiondegamma}.
\end{proof}

En vue de la définition du \og groupe fondamental sauvage \fg, nous aurons besoin
de préciser ces actions sur la base canonique de l'algèbre de Lie graduée $L_{q_r}^1$,
décrite au théorème \ref{thm:basecanonique} de \ref{subsubsection:basecanonique}.
Dans ce numéro, on se plaçait en fait dans la catégorie $\EE_q$, dont on étudiait
l'action du groupe $\Gqps$ sur $\sttq$; et on y définissait l'algèbre de Lie graduée
libre $L_{q}^1 \subset \sttqu$, tout cela relativement à la \og base \fg\ $q$. Ici nous
devons remplacer $q$ par $q_r$. Nous noterons encore $\Delta_l^{(\delta,\beta)}$ (les indices
$l \in \Z$ étant comptés modulo $\delta$) les éléments de la base canonique de $L_{q_r}^1$;
mais, bien entendu, avec $\beta \in \mathbf{E}_{q_r}$. \\

Il est immédiat à partir du théorème ci-dessus que l'on a, pour $l = 0,\ldots,\delta-1$
et avec les notations du théorème::
\begin{align*}
\left(\Delta_l^{(\delta,\beta)}\right)^h &= h(\delta/r) \; \Delta_l^{(\delta,\beta)}, \\
\left(\Delta_l^{(\delta,\beta)}\right)^{\gamma_1} &= \gamma_1(\beta) \; \Delta_l^{(\delta,\beta)}.
\end{align*}

Le cas de $\gamma_2$ est un peu plus compliqué et nécessite que l'on précise bien
les choix et les notations. \\

Fixons donc $\delta \in \N^*$ et $\beta \in \Eqr$. Soit $e$ l'unique représentant de $\beta$
dans la couronne fondamentale: $1 \leq \lmod e \rmod < \lmod q_r \rmod$; et soit $\phi$ son
argument pris dans $\left[0,2 \pi\right[$. On pose $d := e^{-1}$, de sorte que
$\lmod q_r \rmod^{-1} \leq \lmod d \rmod < 1$ et que $-\phi \in \left]-2 \pi,0\right]$
est un argument de $d$. Soit $c_l$ l'unique racine $\delta$\ieme\ de $d$ d'argument
$\dfrac{-\phi - 2 l \pi}{\delta} \in
\left]-(l+1) \dfrac{2 \pi}{\delta},-l \dfrac{2 \pi}{\delta}\right]$.
Donc $c_l = \zeta_\delta^{-l} c_0$. De plus,
$\lmod q_{r \delta} \rmod^{-1} \leq \lmod c_l \rmod < 1$.
Notons enfin $\alpha_l := \overline{c_l} \in \Eqr$. \\

Ces notations étant posées, par définition (théorème \ref{thm:basecanonique} de
\ref{subsubsection:basecanonique} adapté à la base $q_r$):
$$
\Delta_l^{(\delta,\beta)} = \Delta_{\alpha_l}^{(\delta,\beta)}, \quad l = 0,\ldots,\delta-1.
$$

Pour appliquer le (iii) du théorème, on pose $\beta' := \overline{\zeta_r}^{-\delta} \beta$.
Les notations correspondant à celles explicitées ci-dessus, appliquées à $\beta'$, donnent
$d' = \zeta_r^\delta d$. Les $c'_l$ sont donc (peut-être à l'ordre près) les $\zeta_r c_l$;
de plus, la suite des $c_l$ comme celle des $c'_l$ est géométrique de raison $\zeta_\delta^{-1}$,
ce qui entraîne que ce sont les mêmes suites à décalage près: il existe un entier $l_0$ tel
que $\zeta_r c_l = c'_{l+l_0}$ pour tout $l$. En particulier, $l_0$ est déterminé (à un multiple
près de $\delta$) par la condition $\zeta_r c_0 = c'_{l_0}$. Comme $\zeta_r c_0$ est une racine
$\delta$\ieme\ de $d'$ telle que $\lmod q_{r \delta} \rmod^{-1} \leq \lmod \zeta_r c_0 \rmod < 1$
et dont un argument est $2 \pi/r - \phi/\delta$, il suffit de trouver $l_0$ tel que:
$$
2 \pi/r - \phi/\delta \in \left]-(l_0+1) \dfrac{2 \pi}{\delta},-l_0 \dfrac{2 \pi}{\delta}\right].
$$
Nous noterons à la française $E(x)$ la partie entière d'un réel $x$. La condition ci-dessus
équivaut à $l_0 = E\left(\phi/2\pi - \delta/r\right)$.    

\begin{defn}
\label{defn:subtildécalage}
On définit la fonction $\ell: \N^* \times \Eqr \rightarrow \Z$ par la formule:
$$
\ell(\delta,\beta) := E\left(\phi/2\pi - \delta/r\right),
$$
où $\phi$ est l'argument pris dans $\left[0,2\pi\right[$ d'un représentant $e$
de $\beta$ choisi dans la couronne fondamentale $1 \leq \lmod e \rmod < \lmod q_r \rmod$.
\end{defn}

Dans cette définition, $r$ est implicite, puisque l'on s'est placé une fois pour toutes
dans la catégorie $\EE_{q_r}$.

\begin{cor}[du théorème \ref{thm:conjuguésdérivéesétrangères}]
\label{cor:conjuguésbasecanonique}
L'action adjointe de $\Gqps$ sur $L_{q_r}^1$ est décrite par les formules:
\begin{align*}
\left(\Delta_l^{(\delta,\beta)}\right)^h &= h(\delta/r) \; \Delta_l^{(\delta,\beta)}, \\
\left(\Delta_l^{(\delta,\beta)}\right)^{\gamma_1} &= \gamma_1(\beta) \; \Delta_l^{(\delta,\beta)}, \\
\left(\Delta_l^{(\delta,\beta)}\right)^{\gamma_2} &=
\overline{\gamma}_2(\alpha_l) \; \Delta_{l+\ell(\delta,\beta)}^{(\delta,\zeta_r^{-\delta}\beta)}.
\end{align*}
\end{cor}

Pour une formulation où n'apparaît pas explicitement le point base $z_0$ (ici via
la fonction $\overline{\gamma}_2$), nous introduisons une autre base de $L_{q_r}^1$
en posant:
\begin{equation}
\label{eqn:meilleurebasecanonique}
\forall \delta \in \N^* \:;\: \forall \beta \in \Eqr \;,\;
\forall l \in \{0,\ldots,\delta-1\} \;,\;
\Psi_l^{(\delta,\beta)} := \theta_{q_r}(z_0/c_l) \Delta_l^{(\delta,\beta)},
\end{equation}
où $c_l := c_{l,0}$ comme précédemment. Comme expliqué au début de
\ref{subsection:calculs explicites}, contrairement aux énoncés précédents
l'énoncé ci-dessous est valable sans restriction dans tout $\EE_q^r$.

\begin{thm}
\label{thm:conjuguésmeilleurebasecanonique}
On suppose que $q$ a une bonne valeur.
L'action adjointe de $\Gqps$ sur $L_{q_r}^1$ est décrite par les formules:
\begin{align*}
\left(\Psi_l^{(\delta,\beta)}\right)^h &= h(\delta/r) \; \Psi_l^{(\delta,\beta)}, \\
\left(\Psi_l^{(\delta,\beta)}\right)^{\gamma_1} &= \gamma_1(\beta) \; \Psi_l^{(\delta,\beta)}, \\
\left(\Psi_l^{(\delta,\beta)}\right)^{\gamma_2} &=
\gamma_2(c_l^{-1}) \; \Psi_{l+\ell(\delta,\beta)}^{(\delta,\zeta_r^{-\delta}\beta)}.
\end{align*}
\end{thm}
\begin{proof}
C'est la traduction directe du corollaire \ref{cor:conjuguésbasecanonique}.
\end{proof}

% 4.6

\subsection{Structure du groupe fondamental sauvage}
\label{subsection:structgroufonsau}

Nous formulons ici le théorème qui étend au cas des pentes arbitraires le théorème
3.10 de \cite{RS3}. Outre la relaxation de l'hypothèse d'intégrité des pentes, nous
proposons une légère amélioration en ce que la composante formelle (ou pure) est
elle-même décrite comme un groupe discret. \\

Dans la ligne basse de l'équation \eqref{eqn:diagrammedesuitesexactes} du numéro
\ref{subsubsection:RestrictionàE^retramification}:
$$
\xymatrix{
1 \ar@<0ex>[r] & (\frac{1}{r} \Z)^\vee \ar@<0ex>[r] &
\Gqpsr \ar@<0ex>[r] & 
\Eqvee \ar@<0ex>[r] & 1,
}
$$
on identifie $(\frac{1}{r} \Z)^\vee$ à $\Cs$ par $h \mapsto h(1/r)$, d'où une extension:
$$
1 \rightarrow \Cs \rightarrow \Gqpsr \rightarrow \Eqvee \rightarrow 1.
$$
Cela permet de décrire le groupe $\Gqpsr$ comme l'ensemble $\Cs \times \Eqvee$ muni
de la loi:
$$
(t,\gamma) \ast (t',\gamma') := (t t' \eta(\gamma,\gamma'),\gamma \gamma')
\text{~où~}
\eta(\gamma,\gamma') := \gamma'\left(\gamma\left(q^{-1/r}\right)\right) \in \Cs.
$$
Le sous-groupe de torsion $\mu_\infty := \text{Tor}(\Cs)$ de $\Cs$ formé des racines
de l'unité est infini, donc Zariski-dense dans $\Cs$. On a vu que le sous-groupe
abélien libre $<\gamma_1,\gamma_2>$ de $\Eqvee$ est également Zariski-dense. Comme
$\eta(\gamma,\gamma')$ est une racine de l'unité (on l'a déjà vu, mais cela sera
précisé ci-dessous), on a un sous-groupe Zariski-dense de $\Gqpsr$
(décrit comme ci-dessus par $\Cs \times \Eqvee$) d'ensemble sous-jacent
$\mu_\infty \times <\gamma_1,\gamma_2>$. des égalités:
$$
\gamma_1(\zeta_r) = \zeta_r, \quad \gamma_1(q_r) = 1, \quad 
\gamma_2(\zeta_r) = 1,  \quad \gamma_2(q_r) = \zeta_r,
$$
on tire:
$$
\eta(\gamma_1,\gamma_1) = \eta(\gamma_1,\gamma_2) = \eta(\gamma_2,\gamma_2) = 1,
\quad \eta(\gamma_2,\gamma_1) = \zeta_r^{-1},
$$
d'où, par bilinéarité de $\eta$:
$$
\eta\left(\gamma_1^{k_1} \gamma_2^{k_2},\gamma_1^{l_1} \gamma_2^{l_2}\right) = \zeta_r^{-k_2 l_1}.
$$
Identifiant $\mu_\infty$ à $\Q/\Z$ (par l'application $x \mapsto e^{2 \ii \pi x}$) et
$<\gamma_1,\gamma_2>$ à $\Z^2$, on obtient enfin le groupe d'ensemble sous-jacent
$(\Q/\Z) \times \Z \times \Z$ et de loi:
$$
(x,k_1,k_2) \ast (y,l_1,l_2) := (x + y - \text{cl}(k_2l_1/r), k_1 + l_1,k_2 + l_2).
$$
(On vérifie aisément que c'est bien une loi de groupe.) En accord\footnote{À ceci
près qu'on omet l'exposant ${}^{(0)}$ qui exprime dans \emph{loc. cit.} qu'il s'agit
d'une étude locale en $0$; et que l'on met en exposant ${}^r$ le dénominateur commun
des pentes.} avec \cite[\S 3.6]{RS3}, notons ce groupe $G_{p,s}^r$. On a donc une
suite exacte (maintenant en notation additive):
$$
0 \longrightarrow \Q/\Z \longrightarrow \Gqpsr \longrightarrow \Z^2 \longrightarrow 0.
$$

Soit par ailleurs $L_q^r$ l'algèbre de Lie graduée libre de base la famille de tous les
symboles $\Psi_{l}^{(\delta,\beta)}$, $\delta \in \N^*$, $l \in \{0,\ldots,\delta-1\}$,
$\beta \in \Eqr$ (famille graduée par $\N^* \times \Eqr$); à laquelle on adjoint l'élément
$\Psi^{(0)}$, qui correspond à $\nu$. On fait agir le groupe $\Gqpsr$ sur $L_q^r$ par les
règles suivantes:
\begin{enumerate}
\item L'action de tout $\Gqpsr$ sur $\Psi^{(0)}$ est triviale (l'identité).
\item L'action de $(x,0,0)$ multiplie $\Psi_{l}^{(\delta,\beta)}$ par
$e^{2 \ii \pi \delta x}$.
\item L'action de $(0,1,0)$ multiplie $\Psi_{l}^{(\delta,\beta)}$ par $\gamma_1(\beta)$.
\item L'action de $(0,0,1)$ envoie $\Psi_{l}^{(\delta,\beta)}$ sur
$\gamma_2(\alpha_l^{-1}) \Psi_{l + \ell(\delta,\beta)}^{(\delta,\overline{\zeta_r}^{-\delta} \beta)}$.
\end{enumerate}

\begin{defn}
Le \emph{groupe fondamental sauvage} de $\EE_q^r$ est l'objet hybride défini par
le produit semi-direct:
$$
\pi_{1,q,w}^r := L_q^r \ltimes G_{p,s}^r
$$
déterminé par l'action décrite ci-dessus.
\end{defn}

Le groupe de Galois $\Gqr$ qui est l'objet de cette étude est le groupe proalgébrique
tel que la catégorie tannakienne $\EE_{q}^r$ soit équivalente à la catégorie de ses
représentations rationnelles, $\Repc(\Gqr)$. Les constructions précédentes définissent
un morphisme de $\pi_{1,q,w}^r$ dans $\Gqr$, d'où un foncteur de restriction de
$\Repc(\Gqr)$ dans la catégorie $\Repc(\pi_{1,q,w}^r)$ des représentations du groupe
abstrait $\pi_{1,q,w}^r$. Les résultats de cette section sont résumés par le théorème
suivant, qui est le pendant de \cite[th. 3.10]{RS3}. Comme le théorème
\ref{thm:conjuguésmeilleurebasecanonique}, ce théorème est valable sans restriction
dans tout $\EE_q^r$.

\begin{thm}
\label{thm:groupefondamentalsauvage}
On suppose que $q$ a une bonne valeur. \\
(i) Le foncteur de restriction de $\Repc(\Gqr)$ dans $\Repc(\pi_{1,q,w}^r)$ est un
isomorphisme, \ie\ il est pleinement fidèle et bijectif sur les objets. \\
(ii) Il y a une bijection naturelle entre les classes d'isomorphisme de représentations
de $\pi_{1,q,w}^r$ et les classes d'isomorphismes d'objets de $\EE_{q}^r$. Les groupes
de Galois de tels objets sont les clotures de Zariski des représentations associées
de $\pi_{1,q,w}^r$.
\end{thm}

Lorsque $q$ a une mauvaise valeur, il n'est pas possible d'extraire de la famille des
$\Delta_\alpha^{(\delta,\beta)}$ une base telle que l'action de $\gamma_2$ soit une simple
combinaison de permutation et d'homothétie. Il n'est alors pas difficile de déduire
du théorème \ref{thm:conjuguésdérivéesétrangères} un analogue du théorème
\ref{thm:groupefondamentalsauvage} pour une base moins sympathique, mais la description
est alors moins parlante.

%%%%%%%%%%%%%%%%%%%%%%%%%%%%%%%%%%%%%%%%%%%%%%%%%%%%%%%%%%%%%%%%%%%%%%%%%%%%%

\bibliography{TALPA}

\begin{thebibliography}{10}

\bibitem{BV}
D.G. Babbitt and V.S. Varadarajan.
\newblock {\em {Local moduli for meromorphic differential equations.}}
\newblock {Centre National de la Recherche Scientifique. Ast\'erisque, 169-170.
  Paris: Soci\'et\'e Math\'ematique de France.}, 1989.

\bibitem{BerndtRamanujanNBIII}
Bruce~C. {Berndt}.
\newblock {\em {Ramanujan's notebooks. Part III.}}
\newblock New York etc.: Springer-Verlag, 1991.

\bibitem{Birkhoff1}
George~D. Birkhoff.
\newblock The generalized riemann problem for linear differential equations and
  the allied problems for linear difference and $q$-difference equations.
\newblock {\em Proc. Amer. Acad.}, 49:521--568, 1913.

\bibitem{VB}
Virginie {Bugeaud}.
\newblock {Groupe de Galois local des \'equations aux \(q\)-diff\'erences
  irr\'eguli\`eres.}
\newblock {\em {Ann. Inst. Fourier}}, 68(3):901--964, 2018.

\bibitem{DF}
P.~{Deligne}.
\newblock {Cat\'egories tannakiennes.}
\newblock {The Grothendieck Festschrift, Collect. Artic. in Honor of the 60th
  Birthday of A. Grothendieck. Vol. II, Prog. Math. 87, 111-195 (1990).}, 1990.

\bibitem{DM}
Pierre Deligne and J.S. Milne.
\newblock {Tannakian categories.}
\newblock {Hodge cycles, motives, and Shimura varieties, Lect. Notes Math. 900,
  101-228 (1982).}, 1982.

\bibitem{LDV}
Lucia Di~Vizio.
\newblock Arithmetic theory of {$q$}-difference equations: the {$q$}-analogue
  of {G}rothendieck-{K}atz's conjecture on {$p$}-curvatures.
\newblock {\em Invent. Math.}, 150(3):517--578, 2002.

\bibitem{LDV2}
Lucia Di~Vizio.
\newblock {Local analytic classification of $q$-difference equations with $|q|
  = 1$.}
\newblock {\em J. Noncommut. Geom.}, 3(1):125--149, 2009.

\bibitem{LDVZ}
Lucia Di~Vizio and Changgui Zhang.
\newblock {On $q$-summation and confluence.}
\newblock {\em Ann. Inst. Fourier}, 59(1):347--392, 2009.

\bibitem{DREL}
Thomas {Dreyfus} and Anton {Eloy}.
\newblock {$q$-Borel-Laplace summation for $q$-difference equations with two
  slopes.}
\newblock {\em {J. Difference Equ. Appl.}}, 22(10):1501--1511, 2016.

\bibitem{Etingof}
Pavel Etingof.
\newblock {Galois groups and connection matrices for $q$-difference equations.}
\newblock {\em Electron. Res. Announc. Am. Math. Soc.}, 1(1):1--9, 1995.

\bibitem{Frenkel}
Jean Frenkel.
\newblock {Cohomologie non ab\'elienne et espaces fibr\'es.}
\newblock {\em Bull. Soc. Math. Fr.}, 85:135--220, 1957.

\bibitem{ORS}
Yousuke Ohyama, Jean-Pierre Ramis, and Jacques Sauloy.
\newblock The space of monodromy data for the jimbo-sakai family of
  $q$-difference equations.
\newblock {\em Annales de la Facult\'e de Sciences de Toulouse}, 2020.

\bibitem{RS3}
Jean-Pierre {Ramis} and Jacques {Sauloy}.
\newblock {Le $q$-analogue du groupe fondamental sauvage et le probl\`eme
  inverse de la th\'eorie de Galois aux $q$-diff\'erences.}
\newblock {\em {Ann. Sci. \'Ec. Norm. Sup\'er. (4)}}, 48(1):171--226, 2015.

\bibitem{RSZ}
Jean-Pierre {Ramis}, Jacques {Sauloy}, and Changgui {Zhang}.
\newblock {\em {Local analytic classification of \(q\)-difference equations}},
  volume 355.
\newblock Paris: Soci\'et\'e Math\'ematique de France (SMF), 2013.

\bibitem{RZ}
Jean-Pierre Ramis and Changgui Zhang.
\newblock {$q$-Gevrey asymptotic expansion and Jacobi theta function.
  (D\'eveloppement asymptotique $q$-Gevrey et fonction th\^eta de Jacobi.)}.
\newblock {\em C. R., Math., Acad. Sci. Paris}, 335(11):899--902, 2002.

\bibitem{Roques2}
Julien Roques.
\newblock {Generalized basic hypergeometric equations.}
\newblock {\em Invent. Math.}, 184(3):499--528, 2011.

\bibitem{RoquesSauloy}
Julien {Roques} and Jacques {Sauloy}.
\newblock {Euler characteristics and q-difference equations.}
\newblock {\em {Annali della Scuola Normale Superiore di Pisa, Classe di
  Scienze}}, 20100.

\bibitem{JSAIF}
Jacques Sauloy.
\newblock Syst\`emes aux {$q$}-diff\'erences singuliers r\'eguliers:
  classification, matrice de connexion et monodromie.
\newblock {\em Ann. Inst. Fourier (Grenoble)}, 50(4):1021--1071, 2000.

\bibitem{JSGAL}
Jacques Sauloy.
\newblock Galois theory of {F}uchsian {$q$}-difference equations.
\newblock {\em Ann. Sci. \'Ecole Norm. Sup. (4)}, 36(6):925--968 (2004), 2003.

\bibitem{JSStokes}
Jacques Sauloy.
\newblock Algebraic construction of the {S}tokes sheaf for irregular linear
  {$q$}-difference equations.
\newblock {\em Ast\'erisque}, 296:227--251, 2004.
\newblock Analyse complexe, syst\`emes dynamiques, sommabilit\'e des s\'eries
  divergentes et th\'eories galoisiennes. I.

\bibitem{JSFIL}
Jacques Sauloy.
\newblock La filtration canonique par les pentes d'un module aux
  {$q$}-diff\'erences et le gradu\'e associ\'e.
\newblock {\em Ann. Inst. Fourier (Grenoble)}, 54(1):181--210, 2004.

\bibitem{vdPR}
Marius van~der Put and Marc Reversat.
\newblock Galois theory of $q$-difference equations.
\newblock {\em Ann. Fac. Sci. Toulouse Math. (6)}, 16(3):665--718, 2007.

\bibitem{vdPS1}
Marius van~der Put and Michael~F. Singer.
\newblock {\em Galois theory of difference equations}, volume 1666 of {\em
  Lecture Notes in Mathematics}.
\newblock Springer-Verlag, Berlin, 1997.

\end{thebibliography}

\end{document}